\documentclass[a4paper,11p]{amsart}

\usepackage[all,arc]{xy}
\usepackage[dvipsnames]{xcolor}
\usepackage[textwidth=16cm, textheight=24cm, marginratio=1:1]{geometry}

\usepackage{graphicx} % Required for inserting images
\usepackage[centertags]{amsmath}
\usepackage{amscd}
\usepackage{amsthm}
\usepackage{amssymb}
\usepackage{dsfont} %%Package for evaluation functions
\usepackage{MnSymbol} %%Package for the star operator
\usepackage{algorithm}
\usepackage{algpseudocode}
\usepackage{tikz-cd}
\usetikzlibrary{cd}
\usepackage{tikz}
\usepackage{relsize}
\usepackage{ytableau}
\usepackage{mdframed}
\usepackage{comment}
\usepackage{enumitem}
\usepackage[utf8]{inputenc}
\usepackage[T1]{fontenc}

\usepackage{todonotes}

\usepackage{textcomp}

\usepackage{hyperref}
\hypersetup{
    colorlinks,
    linkcolor={MidnightBlue},
    citecolor={Green}
}
\usepackage[nameinlink]{cleveref}

\usepackage[normalem]{ulem}

\usepackage{enumitem}

\newcommand{\m}{\mathfrak m}
\usepackage{aliascnt}

\newtheorem{defn0}{Definition}[section]

\newaliascnt{prop0}{defn0}
\newtheorem{prop0}[prop0]{Proposition}
\aliascntresetthe{prop0}

\newaliascnt{thm0}{defn0}
\newtheorem{thm0}[thm0]{Theorem}
\aliascntresetthe{thm0}

\newaliascnt{lemma0}{defn0}
\newtheorem{lemma0}[lemma0]{Lemma}
\aliascntresetthe{lemma0}

\newaliascnt{corollary0}{defn0}
\newtheorem{corollary0}[corollary0]{Corollary}
\aliascntresetthe{corollary0}

\newaliascnt{example0}{defn0}
\newtheorem{example0}[example0]{Example}
\aliascntresetthe{example0}

\newaliascnt{remark0}{defn0}
\newtheorem{remark0}[remark0]{Remark}
\aliascntresetthe{remark0}

\newtheorem{claim}{Claim}

\newaliascnt{conjecture0}{defn0}
\newtheorem{conjecture0}[conjecture0]{Conjecture}
\aliascntresetthe{conjecture0}

\crefname{example0}{Example}{Examples}
\Crefname{example0}{Example}{Examples}

\crefname{remark0}{Remark}{Remarks}
\Crefname{remark0}{Remark}{Remarks}

\crefname{thm0}{Theorem}{Theorems}
\Crefname{thm0}{Theorem}{Theorems}

\crefname{prop0}{Proposition}{Propositions}
\Crefname{prop0}{Proposition}{Propositions}

\crefname{lemma0}{Lemma}{Lemmas}
\Crefname{lemma0}{Lemma}{Lemmas}

\crefname{corollary0}{Corollary}{Corollaries}
\Crefname{corollary0}{Corollary}{Corollaries}

\crefname{defn0}{Definition}{Definitions}
\Crefname{defn0}{Definition}{Definitions}

\crefname{conjecture0}{Conjecture}{Conjectures}
\Crefname{conjecture0}{Conjecture}{Conjectures}

\DeclareFontFamily{U}{MnSymbolC}{}
\DeclareSymbolFont{MnSyC}{U}{MnSymbolC}{m}{n}
\DeclareFontShape{U}{MnSymbolC}{m}{n}{
    <-6>  MnSymbolC5
   <6-7>  MnSymbolC6
   <7-8>  MnSymbolC7
   <8-9>  MnSymbolC8
   <9-10> MnSymbolC9
  <10-12> MnSymbolC10
  <12->   MnSymbolC12}{}
\DeclareMathSymbol{\aprod}{\mathbin}{MnSyC}{'270}

\newenvironment{definition}{ \begin{defn0}}{\end{defn0}}
\newenvironment{proposition}{\bigskip \begin{prop0}}{\end{prop0}}
\newenvironment{theorem}{\bigskip \begin{thm0}}{\end{thm0}}
\newenvironment{lemma}{\bigskip \begin{lemma0}}{\end{lemma0}}
\newenvironment{corollary}{\bigskip \begin{corollary0}}{\end{corollary0}}
\newenvironment{example}{ \begin{example0}\rm}{\end{example0}}
\newenvironment{remark}{ \begin{remark0}\rm}{\end{remark0}}

\newtheorem{alg}{Algorithm}
\newcommand{\K}{\mathbb K}

\newcommand{\ale}[1]{{\color{magenta}#1}}

\DeclareMathOperator{\Sym}{Sym}

\title{Hankel and Multiplication Tensor Completions for Cactus Rank}

\author{Alessandra Bernardi, Joachim Jelisiejew, Oriol Reig Fité}

\address[Alessandra Bernardi, Oriol Reig Fité]{Universit\`a di Trento, Via Sommarive, 14 - 38123 Povo (Trento), Italy}
\email{alessandra.bernardi@unitn.it, oriol.reigfite@unitn.it}

\address[Joachim Jelisiejew]{Wydział Matematyki, Informatyki i Mechaniki, Uniwersytet Warszawski, ul.~Stefana Banacha 2, 02-097 Warsaw, Poland.}
\email{j.jelisiejew@uw.edu.pl}

\makeatletter
\@namedef{subjclassname@2020}{\textup{2020} Mathematics Subject Classification}
\makeatother


\begin{document}
\begin{abstract}
We show that the Hankel flat extension
formulation of the cactus algorithm is equivalent to a completion problem for
multiplication tensors of Artinian Gorenstein algebras. The unknown Hankel
moments are canonically identified with the undetermined tensor coefficients,
and under this identification the symbolic multiplication matrices and their
commutation equations coincide. This shows that the usual degree extension
formulation is a coordinate realization of a variable extension problem with
marked generators. We further use Borel-fixed and squat staircases to reduce
the family of candidate basis shapes in the resulting algorithm.
\end{abstract}
\maketitle

\def\OldIntro{
\section*{Introduction}

Decomposing a homogeneous polynomial into simpler pieces is a central problem in Algebraic Geometry, Computational Algebra, and Applications.
Among the many notions of complexity that arise in this context, the \emph{cactus rank} of a form $F$ measures the minimal length of a zero--dimensional scheme $Z\subset \mathbb{P}V$ whose degree-$d$ Veronese span contains $F$.
Compared to classical Waring rank, cactus rank is intrinsically scheme--theoretic: it allows non-reduced schemes (cf. \cite{Bernardi2018, BR13}).
The \emph{local cactus rank} further restricts to schemes supported at a single point, providing a refined invariant that is particularly natural from the viewpoint of Artinian local algebras.

Apolar duality offers an effective bridge between schemes and tensors.
Given a form $F\in \Sym^d(V)$ (or its divided--power analogue), a finite scheme $Z=\mathrm{Spec}(A)\subset \mathbb{P}V$ is apolar to $F$ if the defining ideal of $Z$ is contained in the apolar ideal $\mathrm{Ann}(F)$.
When $Z$ is Gorenstein, apolarity can be encoded by a single dual generator (Macaulay inverse system) and the associated Artinian Gorenstein algebra $A$ carries all the algebraic structure.
In particular, the multiplication in $A$ can be represented by its so-called \emph{structure tensor} (cf. \cite{MinBrk}), or equivalently by the family of multiplication matrices with respect to a basis.

This motivates the main viewpoint of the present paper.
In the literature, the cactus algorithm is usually presented in terms of Hankel/catalecticant matrices and flat extension (\cite{MR3250539,Alessandra, MourrainFOCM}), while Artinian Gorenstein algebras are often encoded through multiplication tensors (cf. \cite{Joachim}).
We propose a unified perspective that makes the two languages meet: we compare their underlying constraints, clarify similarities and differences, and show how the classical Hankel-based cactus algorithm can be recovered from the multiplication-tensor formalism.
Moreover, this viewpoint naturally leads to a \emph{completion} interpretation, which in turn allows extensions beyond the traditional degree-extension paradigm: besides adding higher-degree moments, one can also enlarge the ambient space by introducing new variables corresponding to additional basis elements of the unknown algebra.

More concretely, we show that apolarity to $F$ can be expressed via a \emph{restriction} of a multiplication tensor.
If $A$ is a length-$r$ Artinian Gorenstein algebra, its $(d-1)$-st iterated multiplication tensor (in symmetric form) yields a concise element
\[
\mu_{A,s}^{(d-1)} \in \Sym^d(A^*).
\]
Our first main result proves that, whenever $\operatorname{crk}(F)\le r$ (with $r\ge n+1$), there exist a vector space $W$ of dimension $r$, a concise tensor $G\in \Sym^d(W)$ of cactus rank $r$, and a linear map $\varphi:W\to V$ such that
\[
\Sym(\varphi)(G)=F,
\]
and moreover $G$ can be chosen to be (isomorphic to) a multiplication tensor of an Artinian Gorenstein algebra of length $r$.
This provides a conceptual explanation of why multiplication tensors are the natural ambient objects for cactus computations, and it isolates the role of the linear map $\varphi$ as a restriction operator.

Building on this, we give an equivalent formulation of the cactus algorithm of \cite{Alessandra}.
In the classical degree--extension approach, one completes a truncated dual generator by introducing unknown higher moments, and then enforces commutation relations for the induced multiplication matrices.
Here we reinterpret the same constraints as the defining equations of a \emph{tensor completion problem} for the multiplication tensor, equipped with \emph{marked generators} corresponding to the original variables.
This viewpoint has two advantages.
First, it is \emph{basis--free} at the conceptual level, while still allowing effective computations after a suitable gauge choice.
Second, it naturally accommodates \emph{variable--extension} strategies, where one enlarges the ambient space by new variables that represent additional basis elements of the unknown algebra, rather than only adding higher--degree moments.
We formalize this as the \emph{Multiplication Tensor--Completion with Marked Generators} (MT--CMG) problem and provide an explicit algorithmic template based on iterated multiplication tensors.

We also discuss two complementary computational aspects.
On the one hand, we give practical criteria to recognize when a concise form has \emph{minimal} cactus rank $n+1$ (cf. \cite{Buczynska_Buczynski__border}),  leading to a simple decision procedure based on low--degree apolar data and the Hessian.
On the other hand, once multiplication matrices are obtained, we explain how to recover the scheme structure effectively (including nonreduced information) via border/staircase rewriting relations.
Finally, we address an important implementation issue: the choice of bases.
Although the tensor formulation is invariant, concrete computations require selecting a basis of the unknown algebra.
We show that one may work, without loss of generality, in a \emph{staircase gauge} (complete staircase bases), which both reduces the search space and typically improves sparsity of the resulting multiplication matrices.

\textbf{Organization.}
\Cref{sec:prelim} collects background on symmetric tensors, apolarity, and basic notions of rank.
\Cref{Section2} recalls structure tensors of Artinian Gorenstein algebras and their symmetric iterated multiplication tensors.
\Cref{{Section:The cactus algorithm}}  reviews the cactus algorithm via degree extension following \cite{Alessandra}.
In \Cref{Section4}  we develop the restriction viewpoint and relate apolar schemes to multiplication tensors.
\Cref{Sec:Recovering the cactus algorithm}  shows how the classical degree--extension algorithm is recovered from our tensor formulation, and \Cref{{Section:Scheme structure from multiplication matrices}}  explains how to reconstruct scheme structure from multiplication matrices.
\Cref{Section 6} discusses staircase gauges, embedded dimension, and the practical choice of bases.
Finally, \Cref{Section8} introduces the MT--CMG completion problem and presents a variable--extension template via iterated multiplication tensors.

}

\section*{Introduction}

Two well-established ways of encoding a finite commutative algebra occur in
apparently distant parts of mathematics. Both draw on long, largely independent
traditions and have become basic tools in their respective areas. In
computational algebra and moment theory, one works with Hankel operators and
commuting multiplication matrices. In tensor geometry and algebraic complexity,
one works with multiplication tensors and with the compatible endomorphisms
recorded by their centroid. At the heart of both lies the same inverse problem:
recover a finite algebra from incomplete data. The main point of this paper is
that, in the cactus-rank setting, the two constructions lead to the same
symbolic completion problem.

The first point of view has classical roots in the theory of sums of powers.
Sylvester related representations of binary forms to kernels of catalecticant
matrices \cite{Sylvester1851}, while Macaulay's inverse systems encoded an
Artinian Gorenstein algebra by a single dual generator
\cite{Macaulay1916,IK}. For a form $F\in\Sym^d(V)$, the Waring rank is the
minimum length of a reduced zero-dimensional scheme $Z\subset\mathbb P V$ such
that
\[
F\in\langle\nu_d(Z)\rangle.
\]
The cactus rank is defined by the same Veronese-span condition, allowing
arbitrary zero-dimensional schemes
\cite{BR13,Bernardi2018,Buczynska_Buczynski__border}. By apolarity, this
condition is equivalent to
\[
I(Z)\subseteq\operatorname{Ann}(F),
\]
and for cactus rank computations one may work with Gorenstein apolar schemes
\cite{JarekWer}. After dehomogenizing $F$, the coefficients of the form
determine only a truncation of the corresponding Macaulay dual generator
$\Lambda$. Recovering an apolar algebra is therefore a completion problem for
a partially known linear functional.

The effective reconstruction of finite algebras from such data has a parallel
classical development. Once a basis is chosen, a zero-dimensional quotient
algebra is encoded by the commuting matrices of multiplication by its
generators; their joint spectral data recover the support and local structure
of the corresponding scheme \cite{MollerStetter1995,NormalForm}. Border-basis
methods turn the consistency of a proposed basis into commutation equations
\cite{NormalForm,tango}, while flat-extension theory determines when truncated
moments extend to a finite-rank Hankel operator \cite{CurtoFialkow1996}.
These ideas became basic tools in algorithms for polynomial systems and
symmetric tensors
\cite{ComonGolubLimMourrain,Bernard,Bernardi201351,MourrainFOCM}. In
particular, the cactus algorithm of \cite{Alessandra} chooses a staircase
basis, introduces the missing higher moments, and imposes the nondegeneracy of
a Hankel block together with the commutation of the resulting multiplication
matrices. We refer to this procedure as \emph{degree extension}.

The second point of view starts from the multiplication tensor itself. Structure
tensors of algebras are fundamental objects in algebraic complexity, where
tensor restriction formalizes the simulation of one bilinear or multilinear
map by another
\cite{MR3729273,derksen_et_al:LIPIcs.CCC.2022.9}. The centroid is a classical
invariant of an algebra. Its tensor analogue consists of tuples of
endomorphisms whose actions through the different tensor factors agree. In
tensor-isomorphism and decomposition problems, it provides a computable model
for hidden algebraic structure, detecting direct products and symmetries
\cite{Wilson,Brooksbank_Maglione_Wilson}. For the multiplication tensor of a
finite algebra, the centroid recovers the algebra itself. More recently, the
same compatible endomorphism algebra, called the $111$-algebra in
\cite{MinBrk}, has become a central invariant in the study of concise tensors
of minimal border rank. Under the conciseness and cyclicity hypotheses relevant
here, it characterizes tensors arising from multiplication in a finite
algebra.

In the Artinian Gorenstein setting, a dual generator $\Lambda\in A^*$ turns
the iterated multiplication into the symmetric tensor
\[
\mu_{A,s}^{(d-1)}\in\Sym^d(A^*),
\qquad
\mu_{A,s}^{(d-1)}(a_1,\ldots,a_d)
=
\Lambda(a_1\cdots a_d).
\]
In the apolar setting, this tensor is directly related to the input form:
restricting it to the subspace determined by the original generators recovers
$F$. We make this restriction picture precise in \Cref{Section4}; it explains
why multiplication tensors are the natural ambient objects for the completion
problem considered here.

Thus the two traditions encode the same hidden object but organize its unknown
data differently: the Hankel approach extends $\Lambda$ in degree, whereas
the multiplication-tensor approach enlarges the ambient vector space.
Thus the two traditions encode the same hidden object but organize its unknown
data differently: the Hankel approach extends $\Lambda$ in degree, whereas
the multiplication-tensor approach enlarges the ambient vector space. We show
that this difference is only a choice of coordinates, provided that the
original variables are retained as \emph{marked generators}. The outcome is
both structural and computational.

\medskip
\noindent\textbf{Main results.}
\begin{enumerate}[label=\textup{(\roman*)},leftmargin=*]

%\item
%\emph{Restriction by multiplication tensors.}
%Let $F\in\Sym^d(V)$ be concise and let $r\geq\dim V$. If
%\[
%\operatorname{crk}(F)\leq r,
%\]
%then $F$ is a linear restriction of the symmetric multiplication tensor of a
%length-$r$ Artinian Gorenstein algebra. This tensor is concise and has cactus
%rank exactly $r$; see
%\Cref{Thm: Restriction,Prop: MultTensorExtendsf}. In the reduced case
%$A=\K^r$, this recovers the familiar fact that Waring expressions are
%restrictions of the diagonal tensor. The result places cactus rank in the same
%restriction paradigm that underlies the use of multiplication tensors in
%algebraic complexity.

\item
\emph{Equality of the two symbolic completions.}
Fix a complete staircase basis $\{1,x_1,\ldots,x_n,b_{n+1},\ldots,b_{r-1}\}$
of the unknown algebra. In the Hankel construction, the unknowns are higher
moments $h_\alpha=\Lambda(x^\alpha)$; in the multiplication-tensor
construction, they are coefficients $\lambda_m$ involving the additional
basis directions. We give a canonical identification between the parameters
that enter the multiplication matrices and prove that, under this
identification,
\[
M_{x_i}(h)=\widetilde M_{x_i}(\lambda),
\qquad i=1,\ldots,n.
\]
This is an equality before specialization: the known entries, the unknown
entries, and their positions coincide. Hence the two approaches yield the
same nondegeneracy and commutation equations, not merely isomorphic algebras
after solving them; see
\Cref{clm:identification-mult-matrices,thm: Summary}.

\item
\emph{Variable extension and marked generators.}
The preceding equality shows that degree extension is a particular coordinate
realization of a more intrinsic problem: complete $F$ to a multiplication
tensor in a larger space while recording the original variables through a
marking $\iota:V^*\hookrightarrow A$.
We call this a variable-extension completion with marked generators; see
\Cref{def:MT-CMG}. The marking is essential after nonlinear changes of
generators: a convenient generating set for the abstract algebra need not
coincide with the variables of the input form.

\item
\emph{A smaller search space for bases.}
The completion problem is intrinsic, but its computation requires a basis. In
characteristic zero, linear changes of coordinates allow one to reduce the
search from arbitrary staircases to Borel-fixed ones. Allowing suitable
nonlinear changes leads to the smaller class of \emph{squat ideals}. We prove
that the closure of the isomorphism locus of every zero-dimensional algebra
contains a squat ideal. Consequently, a general choice of generators admits a
basis indexed by a squat staircase; see
\Cref{ref:nonlinearBaseChange:cor}.

This can drastically reduce the combinatorial search. For instance, in three
variables and length $16$, the numbers of all, Borel, and squat staircases are
respectively
\[
11297,\qquad 143,\qquad 4.
\]
We use these bases to formulate the variable-extension cactus algorithm in
\Cref{alg:cactus-squat-variable}.
\end{enumerate}

The reduction from arbitrary staircases to Borel-fixed staircases gives a
uniform improvement at the level of the combinatorial search: the individual
degree extension systems are of the same type, while the number of candidate
basis shapes is reduced. In contrast, the further reduction from Borel-fixed
to squat staircases is not claimed to give a uniform complexity improvement.
It may greatly reduce the number of bases to be tested, but the nonlinear
generators can increase the degrees of the basis elements in the original
coordinates and make the commutator systems more involved. It
does, however, expose a concrete trade-off between testing many simple bases
and testing fewer, more involved ones. Shrinking the combinatorial search can
make additional cases computationally accessible. The corresponding
computational regimes are described in \Cref{rmk: MomentsVsDegrees}. The
centroid viewpoint also yields a direct criterion for recognizing forms of
minimal cactus rank; see \Cref{alg:minimalcactus}.

Conceptually, the paper develops constructive consequences of the structural
theory of multiplication tensors in \cite{MinBrk}. Algorithmically, it
explains why the Hankel equations of \cite{Alessandra} are precisely
multiplication-tensor equations and reveals the additional freedom supplied
by variable extension, markings, and nonlinear choices of basis.

\medskip
\noindent\textbf{Organization of the paper.}
\Cref{sec:prelim} fixes the notation and recalls the necessary background on
symmetric tensors, apolarity, inverse systems, and cactus rank.
\Cref{Section2} reviews multiplication tensors of Artinian Gorenstein algebras
and their characterization through the centroid.
\Cref{Section:The cactus algorithm} recalls the Hankel-based cactus algorithm
and its degree-extension formulation.
In \Cref{Section4}, we relate apolar schemes to restrictions of multiplication
tensors and introduce completions with marked generators.
\Cref{Sec:Recovering the cactus algorithm} proves the canonical identification
between the Hankel moments and the coefficients of the tensor completion, and
shows that the corresponding symbolic multiplication matrices coincide.
\Cref{Section:Scheme structure from multiplication matrices} explains how the
support and the nonreduced structure of the apolar scheme can be recovered from
these matrices.
\Cref{Section 6} studies the choice of bases, first through staircase and
Borel-fixed ideals and then through squat ideals and nonlinear changes of
generators.
Finally, \Cref{Section8} combines these ingredients into a variable-extension
version of the cactus algorithm and discusses the resulting computational
trade-off.
\def\previousIntro{
The study of decompositions of homogeneous polynomials lies at the intersection
of algebraic geometry, commutative algebra, and tensor theory. Classical Waring
rank asks for a decomposition of a form as a sum of powers of linear forms.
Cactus rank is its scheme-theoretic analogue: the cactus rank of a form
$F\in \operatorname{Sym}^d(V)$ is the minimal length of a zero-dimensional
scheme $Z\subset \mathbb P V$ such that $F$ belongs to the linear span of
the Veronese image of $Z$. Thus cactus rank allows nonreduced schemes which are described by finite local algebras, Hilbert functions, and
Macaulay inverse systems; see for instance \cite{BR13, Bernardi2018}. 

Apolarity  provides the basic bridge between forms and schemes. Given a form
$F$, a finite scheme $Z\subset \mathbb P V$ is apolar to $F$ if
$I(Z)\subseteq \operatorname{Ann}(F)$. When the corresponding finite algebra
is Gorenstein, this condition can be expressed by a single dual generator
$\Lambda$, or inverse system. After choosing an affine chart and
dehomogenizing $F$, the computational problem becomes the following: extend
the truncated functional determined by $F$ to a full moment functional
$\Lambda$ whose Hankel operator has finite rank and whose quotient algebra
has certain length.

This point of view is the basis of the moment-matrix and Hankel
flat-extension methods developed in tensor decomposition and polynomial systems, notably in the work of \cite{Bernard}  and in the cactus-rank algorithm of \cite{Alessandra} (cf. also \cite{MR3250539,MourrainFOCM}). In this approach one introduces unknown higher
moments, builds truncated Hankel matrices, and imposes the commutation of the multiplication matrices. Once an extension is found, these commuting matrices reveal  the apolar scheme.

However, there is another natural way to encode the same apolar scheme.
The multiplication of an Artinian algebra $A$ is itself a tensor. In the
Gorenstein case, if $\Lambda\in A^*$ is a dual generator, the iterated
multiplication map gives a symmetric tensor
\[
\mu^{(d-1)}_{A,s}\in \operatorname{Sym}^d(A^*),
\qquad
\mu^{(d-1)}_{A,s}(a_1,\ldots,a_d)=\Lambda(a_1\cdots a_d).
\]
Recent work on structure tensors, centroids, and multiplication tensors of finite algebras \cite{MinBrk,Joachim,Jelisiejew_Bedlewo} shows that one can read much of the algebra structure directly from the tensor. In particular, the centroid gives criteria for deciding when a tensor is induced by multiplication in a finite algebra, and connects this tensor description with the algebraic properties of $A$.

The aim of this paper is to make precise the equivalence between these two
languages in the context of the cactus algorithm. Starting from a form $F$,
we show that an apolar scheme of length $r$ gives rise to a length-$r$
Artinian Gorenstein algebra whose symmetric iterated multiplication tensor
restricts to $F$. Conversely, a suitable multiplication-tensor completion of
$F$, equipped with marked generators corresponding to the original variables,
determines an apolar algebra and hence a cactus scheme. Thus cactus
computation may be formulated either as a Hankel flat-extension problem or as
a multiplication-tensor completion problem.

In particular, fix a staircase basis $B=\{1,x_1,\ldots,x_n,b_{n+1},\ldots,b_{r-1}\}$
for the unknown algebra. In the classical degree-extension algorithm (\cite{Bernard, Bernardi201351,MR3250539,Alessandra}), one
extends the dehomogenized form by introducing higher moments
$\Lambda(x^\alpha)$ of degree $>d$. In the multiplication-tensor
formulation, one instead extends the ambient vector space by variables dual to
the additional basis elements $b_{n+1},\ldots,b_{r-1}$, and leaves the
corresponding coefficients of the tensor undetermined. We prove that these two
sets of unknowns are canonically identified. Under this identification, the
symbolic multiplication matrices obtained from the Hankel operators are the
same as those recovered from the completed multiplication tensor. Therefore
the commutation equations in the Hankel algorithm are precisely the equations
saying that the completed tensor is the multiplication tensor of a
commutative Artinian Gorenstein algebra with the prescribed marked generators.

This equivalence is useful beyond a formal comparison. It shows that the usual degree-extension algorithm is a particular instance of a more general completion problem. Instead of only adding higher moments in the original coordinates, one may enlarge the ambient vector space by adding variables corresponding to further basis elements of the unknown algebra. In this sense, the cactus algorithm can be viewed as a variable-extension problem for multiplication tensors with marked generators. The marking is essential: after a nonlinear change of coordinates, the generators used to put the algebra in a convenient basis need not coincide with the original variables of the form.

This leads to the final part of the paper where we discuss the choice of the bases for the unknown algebra by using the so-called ``squat staircases", a class of monomials arising from nonlinear changes of coordinates and adapted to zero-dimensional algebras (cf. \cite{Alberelli20191, Bertone2013263}). We show that, after passing to a suitable degeneration and applying a general nonlinear change of variables, one may work with a basis indexed by a squat staircase. In the completion problem the original variables must be transported through this change of coordinates: the matching equations with the input form are
\[
f_\alpha=\Lambda(u^\alpha),
\]
where the $u_i$'s are the transported marked generators, and not simply $\Lambda(v^\alpha)$ in the new squat generators. This distinction is the key point that makes the variable-extension algorithm compatible with the original form $F$.

The resulting algorithm keeps the moment conditions of the Hankel method but uses the flexibility of the multiplication-tensor viewpoint. One computes Hankel matrices for the pulled back functional in a squat basis, keeps all entries expressed in the original moment variables, imposes commutation of the multiplication matrices, and adds the compatibility equations relating the squat generators to the transported marked generators. When a solution is found, the same multiplication matrices recover the apolar scheme by the standard eigenvector, generalized eigenspace, or border-rewriting procedures.

\medskip

The paper is organized as follows.
\Cref{sec:prelim} recalls apolarity, divided powers, inverse systems, and cactus rank.
\Cref{Section2} reviews multiplication tensors of Artinian Gorenstein algebras and
the centroid viewpoint. \Cref{Section:The cactus algorithm} recalls the cactus algorithm via Hankel
degree extension. \Cref{Section4} explains how apolarity can be expressed through
restrictions of multiplication tensors. \Cref{Sec:Recovering the cactus algorithm} proves the equivalence
between the Hankel and multiplication-tensor constructions by identifying
their parameters and multiplication matrices. \Cref{Section:Scheme structure from multiplication matrices} recalls how to
reconstruct the apolar scheme from the resulting multiplication matrices.
\Cref{Section 6} discusses the different kind of basis: staircase and squat. Then a nonlinear change of coordinates is described via marked generators. \Cref{Section8} combines these ingredients into a variable-extension version of the cactus algorithm in a squat basis.
}

\section*{Acknowledgements}

We are grateful to  %Joachim Jelisiejew and 
%Jakub Jagiełła for the insightful conversations that started this work, and to  %J. Jelisiejew,
Mateusz Michałek and all the organizers and participants of the Oberwolfach Graduate Seminar: "Modern Algebraic Geometry in Algebraic Combinatorics and Tensors". We thank Jarosław Buczyński, Maciej Gałązka, Bernard Mourrain and Daniele Taufer for all the discussions that benefited this project. We also thank the Simons Institute for having provided a fruitful environment during the Program "Complexity and Linear Algebra".
We thank Alessio Sammartano and Giulio Caviglia for the idea of squat algebras.

\medskip

This work has been supported by European Union’s HORIZON–MSCA-2023-DN-JD programme under the Horizon Europe (HORIZON) Marie Skłodowska-Curie Actions, grant agreement 101120296 (TENORS). We acknowledge the TensorDec Laboratory of the Department of Mathematics of the University of Trento, of which AB and ORF are currently members, for helpful discussions.
AB is member of GNSAGA (INdAM). 
AB has been partially funded by the European Union under NextGeneration EU. PRIN 2022 Prot. n. 2022ZRRL4C 004. Views and opinions expressed are however those of the authors only and do not necessarily reflect those of the European Union or European Commission. Neither the European Union nor the granting authority can be held responsible for them. JJ is supported by National Science Centre grant 2023/50/E/ST1/00336.

\section{Preliminaries and Background}\label{sec:prelim}

\medskip

    We fix $\mathbb{K}$ an algebraically closed field of characteristic zero. Let $V_1,\dots,V_d$ be finite–dimensional $\K$–vector spaces and let $
T \in V_1 \otimes \cdots \otimes V_d.
$ For each $i=1,\dots,d$ we consider the $i$-th \emph{flattening} of $T$, namely the linear map
\[
\mathrm{Flat}_i(T) : V_i^* \longrightarrow 
V_1 \otimes \cdots \otimes \widehat{V_i} \otimes \cdots \otimes V_d,
\]
obtained by contraction of $T$ with a covector in $V_i^*$ on the $i$-th factor.
We say that $T$ is \emph{concise} if for all $i$'s the map $\mathrm{Flat}_i(T)$ is injective.

\medskip

By considering all flattenings, we can see $T$ as a $\K$-multilinear map
\begin{equation}\label{eq: TensorAsMultMap}
    T:V_1^*\times \cdots \times V_d^*\to \K.
\end{equation} 
If $V_1=\cdots =V_d$, we say that $T$ is \emph{symmetric} if $T(v_1,\ldots,v_d)=T(v_{\sigma(1)},\ldots, v_{\sigma(d)})$ for every permutation $\sigma\in \mathcal{S}_d$. We denote by $\Sym^d(V)$ the space of symmetric tensors of order $d$. Once a basis of $V$ has been fixed, we can identify a symmetric tensor $F$ of order $d$ with a homogeneous polynomial of degree $d$. We say that such an $F$ is \emph{concise} if it does not belong to 
 $\Sym^d(W)$ for any proper subspace $W \subsetneq V$.
 Equivalently, $F$ is concise if it depends essentially on all $n+1$ variables, i.e. $F$ cannot be written using fewer variables using a linear transformation on its original variables. 
 
 Moreover, given $\Phi:W\to V$ a linear map we will still denote by $\Phi$ the induced map on the respective symmetric algebras, defined at the degree $d$ subspace as
\[ \Phi\big(w_1\cdots w_d\big)\ :=\ \Phi(w_1)\cdots \Phi(w_d),
 \quad w_1,\dots,w_d\in W,
\]
and extended by linearity.

% Let $W,V$ be finite-dimensional $k$-vector spaces and let $\Phi:W\to V$ be a linear map.
% For every $d\ge 0$ we denote by
% \[
% \Sym^d(\Phi):\Sym^d(W)\longrightarrow \Sym^d(V)
% \]
% the linear map induced on the $d$-th symmetric power\jjcomment{This notation seems very formal and not very much seem in nature(?), perhaps make it lighter, e.g., write $\Phi$ instead of $\Sym^d(\Phi)$?}, defined on decomposable tensors by
% \[
% \Sym^d(\Phi)\big(w_1\cdots w_d\big)\ :=\ \Phi(w_1)\cdots \Phi(w_d),
% \quad w_1,\dots,w_d\in W,
% \]
% and extended by linearity.
% We also write
% \[
% \Sym(\Phi):\Sym(W)=\bigoplus_{d\ge 0}\Sym^d(W)\longrightarrow
% \Sym(V)=\bigoplus_{d\ge 0}\Sym^d(V)
% \]
% for the induced homomorphism of graded $k$-algebras, whose degree-$d$ component is $\Sym^d(\Phi)$.

%Via the usual identifications between symmetric tensors and homogeneous forms (or divided powers), $\Sym^d(\Phi)$ corresponds to linear substitution of variables.

In addition, we say that two tensors $T_1\in V_1\otimes \cdots \otimes V_d$, $T_2\in W_1\otimes \cdots \otimes W_d$ are \emph{isomorphic} if there are $d$ isomorphisms $\phi_i:V_i\to W_i$, $i=1, \ldots, d$, such that $T_2$ is the image of $T_1$ under the induced map; i.e.\[
T_2 = (\phi_1 \otimes \cdots \otimes \phi_d)(T_1).
\]
In the symmetric case $V_1=\cdots=V_d=V$ and $W_1=\cdots=W_d=W$, the notion of isomorphisms corresponds to a linear change of coordinates on the variables of the associated homogeneous polynomials.

\medskip

We now fix $V$ to be a $\mathbb{K}$-vector space of dimension $n+1$. \\Denote by $S := \bigoplus_{d\geq 0} \Sym^d(V^*)\cong 
    \mathbb{K}[x_0, \ldots, x_n]$ the symmetric algebra of $V^*$, identified with the homogeneous coordinate ring of projective space $\mathbb{P}^n\cong \mathbb PV=\operatorname{Proj}(S)$. Let also $R := \mathbb{K}[x_1, \ldots, x_n]$ be the coordinate ring of the standard affine chart $\{x_0 \neq 0\} \subset \mathbb{P}^n$.

\medskip
Let $\alpha=(\alpha_0,\ldots,\alpha_n)\in \mathbb N^{n+1}$ be a multi-index and denote $x^\alpha=x_0^{\alpha_0}\cdots x_n^{\alpha_n}\in S$.
Consider the monomial basis $\{x^\alpha\}_{\alpha\in \mathbb N^{n+1}}$ of $S$ and its dual basis $\{Y^{(\alpha)}\}_{\alpha\in \mathbb N^{n+1}}$ in $S^*:= \mathrm{Hom}_{\mathbb{K}}(S, \mathbb{K})$, defined by $Y^{(\alpha)}(x^\beta)=\delta(\alpha, \beta)$. For every graded piece $S_d$, its dual $S_d^*$ can be identified with the linear span of $\{Y^{(\alpha)}\}_{|\alpha|=d}$. We therefore consider the \emph{ring of divided powers}
\[\K_{\mathrm{dp}}[Y_0,\ldots,Y_n]:= \bigoplus_{d \geq 0} \mathrm{Hom}_{\mathbb{K}}(S_d, \mathbb{K}),
\]
where multiplication is given by $Y^{(\alpha)}\cdot Y^{(\beta)}=\binom{\alpha+\beta}{\alpha}Y^{(\alpha + \beta)}$.

\medskip

Define the \emph{contraction} action of $S$ on $S^*$ as the dual operation to multiplication:
\[
(p \aprod \Lambda)(q) := \Lambda(pq), \quad p,q \in S,\; \Lambda \in S^*.
\]

Note that for the monomial basis and its dual, contraction acts as a scaled partial differentiation, that is, $x^\alpha\aprod Y^{(\beta)}=Y^{(\beta-\alpha)}$, where $\beta-\alpha=0$ if $\alpha_i>\beta_i$ for some $i$ (see \Cref{ex:contraction_divided_powers}). 

\begin{definition}
    Let $\Lambda \in S^*$. We define
    \[
    \operatorname{Ann}(\Lambda) := \{ p \in S \mid p \aprod \Lambda = 0 \},
    \]
    the \emph{annihilator} of $\Lambda$
    with respect to the contraction action. An ideal $I \subseteq S$ is \emph{apolar} to a homogeneous polynomial $F \in \K_{\mathrm{dp}}[Y_0, \ldots, Y_n]_d$ if $I \subseteq \operatorname{Ann}(F)$. A zero-dimensional scheme $Z\subseteq \mathbb PV$ is apolar to $F$ if $I(Z)$ is apolar to $F$.
\end{definition}

This naturally leads to considering quotients of the polynomial ring by annihilators of polynomials. The resulting algebra is a special instance of more general self-dual objects:
\begin{definition}
    Let $A$ be an Artinian $\K$-algebra. Its dual $A^*$ is an $A$-module via contraction. We say that $A$ is \emph{Gorenstein} if $A^*$ is isomorphic to $A$ as an $A$-module, i.e., it is generated by a single element (called the \emph{dual generator}).
\end{definition}

\begin{example}\label{ex:contraction_divided_powers}
To make the above constructions explicit, let us consider the case $n=1$, so that
$S = \mathbb{K}[x_0,x_1]$ and 
\[
\mathbb{K}_{dp}[Y_0,Y_1] = \bigoplus_{d\geq 0} \mathrm{Hom}_{\mathbb{K}}(S_d,\mathbb{K}).
\]
%For instance, the degree-$2$ piece $S_2$ has basis
%\[
%x_0^2,\; x_0x_1,\; x_1^2,
%\]
%and its dual $S_2^*$ has basis
%\[
%Y_0^{(2)},\; Y_0Y_1,\; Y_1^{(2)},
%\]
%characterized by
%\[
%Y_0^{(2)}(x_0^2)=1,\quad
%(Y_0Y_1)(x_0x_1)=1,\quad
%Y_1^{(2)}(x_1^2)=1,
%\]
%and all other pairings equal to $0$.
%
%The divided power multiplication satisfies, for example,
%\[
%Y_0^2 = 2\,Y_0^{(2)},\qquad
%Y_0Y_1 = Y^{(1,1)},\qquad
%Y_1^2 = 2\,Y_1^{(2)},
%\]
%and more generally $Y_0^d = d!\,Y_0^{(d)}$, $Y_1^d = d!\,Y_1^{(d)}$.
%
%Now 
Take the element
\[
\Lambda := Y_0^{(2)} \in S_2^*.
\]
Using the contraction rule $x^\alpha \aprod Y^{(\beta)} = Y^{(\beta-\alpha)}$ (with the convention that 
$Y^{(\gamma)}=0$ if some $\gamma_i<0$), we obtain
\[
x_0 \aprod \Lambda =  Y_0,\qquad
x_0^2 \aprod \Lambda = 1,\qquad
x_0^3 \aprod \Lambda = 0,
\]
and
\[
x_1 \aprod \Lambda = 0,\qquad
x_0^i x_1^j \aprod \Lambda = 0 \quad\text{as soon as } j>0 \text{ or } i\ge 3.
\]
Hence, 
\[\operatorname{Ann}(\Lambda)=(x_1,x_0^3) \subset S,\]
and the Artinian algebra 
\[\K[x_0]/(x_0^3)\]
defines a zero dimensional scheme $Z$ of length 3 supported at the point $[0:1]\in \mathbb P^1$. Since $I(Z)=(x_0^3)\subset \operatorname{Ann}(\Lambda)$, $Z$ is apolar to $\Lambda$.
\end{example}

\medskip

\begin{remark}\label{rmk: DividedPowersvsDiff}
Suppose that the characteristic is zero or larger than $d$.
Then the canonical isomorphism between $(V\otimes \cdots \otimes V)^*\cong V^*\otimes\cdots \otimes V^*$ yields an isomorphism $S_d^*=(\Sym^d V^*)^*$ and $\Sym^d V$. In characteristic zero, the ring $\K_{\mathrm{dp}}[Y_1,\ldots,Y_n]$ is isomorphic to $T:=\K[Y_1,\ldots,Y_n]$ via the identity
\begin{equation}\label{eq: DivPowerConversion}
    Y^{(\alpha)}=\frac{1}{\alpha!}Y_0^{\alpha_0}\cdots Y_n^{\alpha_n}.
\end{equation}
 Moreover, if we define an action of $S$ on $T$, such that $p\in S$ is identified with a differential operator obtained by substituting for all $i$ the variable $x_i$ in $p$ by $\partial/\partial x_i$, the same ring isomorphism given by \eqref{eq: DivPowerConversion} is an $S$-module isomorphism (see e.g. \cite{Geramita}). Thus, if we consider $F\in k[x_0,\ldots,x_n]_d$, i.e., as a usual homogeneous polynomial, then $\operatorname{Ann}(F)$ is the set of polynomials, which, viewed as differential operators on $k[x_0,\ldots,x_n]$, annihilate $F$. 
\end{remark}

%From time to time we will need the language of tensors. 
% Let 
% $S := \Sym(V^*) \simeq K[x_0,\dots,x_n]$ and $S_d \simeq Sym^d V^*$.
% A \emph{symmetric tensor} of order $d$ is an element
% $
% F \in \Sym^d V^* \simeq S_d,
% $
% which we also view as a homogeneous polynomial $F(x_0,\dots,x_n)$ of degree $d$.

% Later, once we introduce the apolar ideal $\mathrm{Ann}(F)\subset K_{dp}[Y_0,\dots,Y_n]$, this will be equivalent to the condition $
% \mathrm{Ann}(F)_1 = \{0\}$, 
% namely the absence of nonzero linear forms that annihilate $F$ in the apolar action.

\section{Structure tensors of Artinian Gorenstein algebras}\label{Section2}

In this section we fix notation and recall the basic facts on Artinian Gorenstein algebras and their structure tensors. 
Our guiding point of view is that the multiplication in a finite-dimensional $\K$-algebra $A$ of length $r$ can be encoded by a tensor (equivalently, by the family of multiplication matrices with respect to a chosen basis). 
In the Artinian Gorenstein case this representation is tightly linked to Macaulay inverse systems, and it provides the algebraic bridge between a polynomial (or tensor) and the scheme-theoretic data that will be exploited in the algorithms of the following sections.

\begin{definition}
    Let $A$ be an Artinian algebra. For $d\geq 2$, we define the $d$-th iterated multiplication tensor of $A$ as the multilinear map
\begin{equation}\label{eq:MultTensorAsMap}
    \mu_A^{(d)}:\underbrace{A\times \dots\times A}_{d \text{ times}}\to A \quad \quad (a_1,\ldots, a_{d})\mapsto a_1
\cdots a_d
\end{equation}
that is, $\mu_A^{(d)}\in A^*\otimes\cdots \otimes A^*\otimes A$ is an order-$(d+1)$ tensor.     
\end{definition}

\begin{remark}
    For every $d\geq 2$, the tensor $\mu_A^{(d)}$ is concise, since for all $r\in A$, $\mu_A^{(d)}(1,\ldots,1,r)=r$, so the map $\mu_A^{(d)}$ is surjective as the multilinear map \eqref{eq:MultTensorAsMap}. 
\end{remark}

\begin{example}\label{Ex: MultTensAG}
    Let $A=\K[x]/(x^3)$, which is Artinian  Gorenstein of length $3$, with $\K$-basis $\{1,x,x^2\}$ and dual generator $Y^{(2)}$.

    The structure tensor of order 2 corresponds to the multiplication table 
\[
%\resizebox{1.5}{!}{$
\begin{array}{r|ccc}
& 1 & x & x^2\\
\hline 
1 & 1   & x & x^2 \\
x & x & x^2 & 0 \\
x^2 & x^2 & 0 & 0\\
\end{array}\]
That is, to the tensor 
\[
\mu_A^{(2)}=1^*\otimes 1^*\otimes 1 + 1^*\otimes x^*\otimes x +1^*\otimes (x^2)^*\otimes x^2 + x^*\otimes 1^*\otimes x + x^*\otimes x^*\otimes x^2 + (x^2)^*\otimes 1^*\otimes x^2%\in A^*\otimes A^*\otimes A
\]
    Consider the isomorphism $A\cong A^*$ which sends $1\mapsto 1\aprod Y^{(2)}=(x^2)^*$, $x\mapsto x^*$, $x^2\mapsto 1^*$. Via this isomorphism, this tensor is transformed into 
\[\mu_{A,s}^{(2)}=1^*\otimes 1^*\otimes (x^2)^* + 1^*\otimes x^*\otimes x^* +1^*\otimes (x^2)^*\otimes 1^* + x^*\otimes 1^*\otimes x^* + x^*\otimes x^*\otimes 1^* + (x^2)^*\otimes 1^*\otimes 1^*%\in A\otimes A\otimes A
,\]
which is the symmetric tensor identified with the divided power polynomial $G=zy^{(2 )} + z^{(2)} x$. 
%Note that $G$ has non vanishing Hessian.
\end{example}

In recent works (cf. \cite{MinBrk, NewPaperJoachim,Jelisiejew_Bedlewo, CosimoWeronikaTim}), the notion of \emph{centroid} %(or $\mathcal{A}_{111}^T$ algebra) 
of a tensor $T$ has been developed. Given a tensor $T\in V_1\otimes\cdots\otimes V_d$, for every $i= 1, \ldots , d$, there is an action of $\operatorname{End}(V_i)$ on $T$, viewed as the action $\operatorname{Id}_{V_1}\times \cdots \times \operatorname{Id}_{V_{i-1}}\times \operatorname{End}(V_i)\times \operatorname{Id}_{V_{i+1}}\times \cdots \times \operatorname{Id}_{V_d}$
%$GL(V_1)\times \cdots \times GL(V_d)$ 
in $T$, which is classically called \emph{multilinear multiplication} and in more modern language coming from Quantum Information Theory (cf. eg. \cite{derksen_et_al:LIPIcs.CCC.2022.9}) it is also called \emph{restriction}: it is said that $T$ restricts to $S$, and is denoted $S \le T$, if $S = (A_1\otimes\cdots\otimes A_d)T$ for some linear maps $A_i$, $i=1, \ldots, d$.

For $X\in \operatorname{End}(V_i)$, we will denote this action by $X\circ_i T$.

\begin{definition}
    Let $T\in V_1\otimes\cdots\otimes V_d$. The centroid $\operatorname{Cen}_T$ %, or 111-algebra of $T$,
    is the set of tuples $(X_1,\ldots,X_d)\in \operatorname{End}(V_1)\times \cdots \times \operatorname{End}(V_d)$ such that 
    \[X_1 \circ_1 T=\cdots = X_d \circ_d T.\]
\end{definition}

In the language of tensor restriction (see e.g. \cite{derksen_et_al:LIPIcs.CCC.2022.9}),
the centroid $\mathrm{Cen}_T$ consists of tuples of local operators whose actions on the
different tensor factors produce the same resulting tensor. It also coincides with
the \emph{$111$-algebra} of $T$ of~\cite[Def.~1.10]{MinBrk}.

\begin{remark}
    As shown in \cite[Theorem 1.11]{MinBrk}, if $T$ is concise then $\operatorname{Cen}_T$ is a commutative algebra. Moreover, for all $i=1, \ldots, d$, an element $(X_1,\ldots,X_n)\in \operatorname{Cen}_T$ acts on $V_i$ and $V_i^*$ simply by the action of $X_i$. This action induces a structure of $\operatorname{Cen}_T$-module on $V_i$ and $V_i^*$ (see \cite[$\S$5]{MinBrk}), which makes the $\K$-multilinear map \eqref{eq: TensorAsMultMap} a $\operatorname{Cen}_T$-multilinear map, that is, for all $r\in \operatorname{Cen}_T$ and every $\omega_i\in V_i^*$, $i\in 1,\ldots,d$, $T(r\circ \omega_1,\ldots,\omega_d)=\cdots=T(\omega_1,\ldots,r\circ\omega_d)$.
\end{remark}

\begin{remark}\label{rmk: CentroidIsoToAlgebra}
    For example, identifying an element $r\in A$ as the endomorphism $r:A\to A$ of multiplication by $r$, and the contraction by $r$ map $r^T:A^*\to A^*$, the assignment $r\mapsto (r^T,\ldots, r^T, r)$ gives an isomorphism $A\cong \operatorname{Cen}_{\mu_A^{(d)}}$. 

\end{remark}
\begin{definition}
    Let $T\in V_1\otimes\cdots\otimes V_d$ be a concise tensor with $\dim V_i=m$ for all $i=1, \ldots, d$. We say that $T$ is centroid abundant if $\dim_\K \operatorname{Cen}_T\geq m$.
\end{definition}

\begin{remark}\label{Rmk: CharactMultTensors}
    Multiplication tensors have been characterized in \cite[Theorem 5.5]{MinBrk}. Namely, a concise tensor $T\in V_1\otimes\cdots\otimes V_d$ with $\dim V_i=m$ for all $i=1, \ldots,d$, is isomorphic to a multiplication tensor of a finite algebra if and only if it is centroid abundant and $V_i^*$ is a cyclic $\operatorname{Cen}_T$-module for every $i=1,2,\ldots, d-1$. In this case, $T\cong \mu_{\operatorname{Cen}_T}^{(d-1)}$ (cf. \cite{Jelisiejew_Bedlewo}).
\end{remark}

\begin{remark}[{\cite{Jelisiejew_Bedlewo}}]\label{Rmk: SymmetricMultTensor}
    In the symmetric setting, if $F\in \K_{\mathrm{dp}}[x_0,\ldots,n]_d$ is a concise homogeneous polynomial then:
    \begin{enumerate}
        \item\label{Rmk:SymmetricMultTensor:1} The dimension of the centroid of $F$ (i.e., centroid abundance condition) can be computed as
        \begin{equation}\label{eq: dimCentroid}
            \dim_\K \operatorname{Cen}_F= \left(\frac{\K[x_0,\ldots, x_n]}{(\operatorname{Ann}_{\leq d-1})} \right)_d=    
        \end{equation}
        
        \[\dim_\K \{G\in S_d^* \; | \; \forall i \; x_i\aprod G\in \langle (x_j\aprod F)_{j=0,\ldots,n}\rangle\}.\]
        Here the subspace $\langle (x_j\aprod F)_{j=0,\ldots,n}\rangle \subset S_{d-1}^*$ is the image 
        of the so-called \emph{first catalecticant}:
        $C_F^{1,d-1}\colon S_1 \to S_{d-1}^*,\quad p\mapsto p\aprod F$, so the condition above is equivalent to
        \[
        \mathrm{Im}\, C_G^{1,d-1} \subset \mathrm{Im}\, C_F^{1,d-1}.
        \]
      
     \item\label{Rmk:SymmetricMultTensor:2} Using \cite[Theorem 5.5]{MinBrk}, one can show that if $F$ is centroid abundant, then $V^*$ is a cyclic $\operatorname{Cen}_F$-module if and only if $F$ has non-vanishing Hessian, which is the determinant of the matrix
     \[\left( x_ix_j\aprod F\right)_{i,j=0,\ldots,n}.\]
                                                            
    \item\label{Rmk:SymmetricMultTensor:3} Assume $F\in \Sym^dV$ is a multiplication tensor. We say $F$ is a \emph{direct sum} if $V=V_1\oplus V_2$ and $F\in \Sym(V_1)\oplus \Sym(V_2)$, i.e., for some linear change of coordinates $F$ can be expressed as a sum of two polynomials using a disjoint set of variables. Then $F$ is a direct sum if and only if the centroid is not a local algebra~\cite{Wilson, Brooksbank_Maglione_Wilson, NewPaperJoachim}.

    \item \label{Rmk:SymmetricMultTensor:4} Let $A$ be an Artinian Gorenstein algebra with dual generator $\Lambda\in A^*$. The isomorphism $A\cong A^*$ induces an isomorphism of tensors $\mu_{A}^{(d)}\cong \mu_{A, s}^{(d)}\in (A^{*})^{\otimes d+1}$ such that $\mu_{A, s}^{(d)}(a_1,\ldots,a_d)=(a_1\cdots a_d)\aprod \Lambda$. Thus, as a multilinear map \eqref{eq: TensorAsMultMap}, $\mu_{A,s}^{(d)}$ is defined by 
    \begin{equation}\label{eq: MultTensorOfAG}
        \mu_{A,s}^{(d)}(a_1,\ldots,a_d,a_{d+1})=\Lambda(a_1\cdots a_d\cdot a_{d+1}).
    \end{equation}
    From this expression we see that $\mu_{A,s}^{(d)}$ is a symmetric tensor, that is, $\mu_{A,s}^{(d)}\in \Sym(A^*)$. Note also that the contraction $p\aprod \Lambda$ corresponds to the flattening $\mu_{A,s}^{(d)}(p,1,\ldots,1,-)$, see~\cite{Jelisiejew_Bedlewo}. 
\end{enumerate}
\end{remark}

\begin{example}\label{ex:ArtinianGorensteinPart2}
    Let $\Lambda=Y^{(2)} + Y^{(2)}Z\in \K_{\mathrm{dp}}[Y,Z]$. The Hilbert function of the Artinian Gorenstein algebra $A=\K[y,z]/\operatorname{Ann}(\Lambda)=\K[y,z]/(y^3,z^2)$ is $(1,2,2,1)$, so its dimension is $6$.
    %, and let $\operatorname{Der}(\Lambda)=\{p\aprod \Lambda \; | \; p\in \K[y,z]\}$ denote the vector space of all its contractions. The annihilator of $\Lambda$ is the ideal $\operatorname{Ann}(\Lambda)=(y^3,z^2)\subseteq \K[y,z]$, and the inclusion $\operatorname{Der}(\Lambda)\subseteq \K[y,z]^*$ gives a natural isomorphism of vector spaces
    %$$\operatorname{Der}(\Lambda)\cong \left(\K[y,z]/\operatorname{Ann}(f)\right)^*,$$  which is in fact an isomorphism of $A$-modules, for $A=\K[y,z]/\operatorname{Ann}(\Lambda)$. In particular, $A$ is Gorenstein  and $A^*$ is generated by $\Lambda$. 
    %The Hilbert function of $A$ is 
    A possible basis of $A$ is $\{1,y,z,y^2,yz,y^2z\}$, and we will denote $\{a,b,c,d,e,f\}$ its dual basis. The tensor $\mu_{A,s}^{(2)}$ is a symmetric multilinear map $\Sym^3:A\to k$, i.e. an element in $(\Sym(A))^*$, defined by
    \[\mu_{A,s}^{(2)}(r_1,r_2,r_3)=(Y^{(2)} + Y^{(2)}Z)(r_1\cdot r_2\cdot r_3) \quad r_1,r_2,r_3\in A\]
    (cf. \eqref{eq: MultTensorOfAG}), 
    and we can write its representation in the divided power basis of the algebra as
    %\jjcomment{We probably should have more consistent notation}  
    \[\mu_{A,s}^{(2)}=ab^{(2)}+b^{(2)}c+a^{(2)}d+abe+a^{(2)}f.\]
    
   % One can check that the Hessian of this polynomial is $-a^6$, and it has its centroid has dimension $6$.
   % A direct computation shows that the apolar ideal of   $\mu_{A,s}^{(2)}$ in $\K[a,b,c,d,e,f]$ has the same truncation in degrees $\le 2$ as $\mathrm{Ann(f)}$, and that  $\dim_\K \mathrm{Cen}_{\mu_{A,s}^{(2)}}=6$.

%\oriol{I changed the $a_i$ to $r_i$, please let me know if that makes the notation clearer}
      Computing the Hessian of $\mu_{A,s}^{(2)}$ with respect to the 
      variables $a,b,c,d,e,f$ one finds
      \[
      \det\mathrm{Hess}(\mu_{A,s}^{(2)}) = -a^6 \neq 0,
      \]
      so $\mu_{A,s}^{(2)}$ has non-vanishing Hessian. Since using \eqref{eq: dimCentroid}, the centroid of $\mu_{A,s}^{(2)}$ has dimension 6, that is, $\dim_\K \mathrm{Cen}_{\mu_{A,s}^{(2)}}=\dim_\K A$, 
      \cite[Theorem~5.5]{MinBrk} implies that $A^*$ is a cyclic 
      $\mathrm{Cen}_{\mu_{A,s}^{(2)}}$-module. 
    Note that, since $A\cong \mathrm{Cen}_{\mu_{A,s}^{(2)}}$, this is precisely the definition of being Gorenstein.  
      
%       Explicitly, since $\mathrm{Cen}_{\mu_{A,s}^{(2)}}\cong A$ with the isomorphism given by multiplication and contraction maps (cf \Cref{rmk: CentroidIsoToAlgebra}), the action of the centroid on $A$ is just the contraction action of $A$ on $A^*$, $(p\cdot\varphi)(u)=\varphi(pu)$. A direct
% computation shows that 
%       %For instance, if we identify $\mathrm{Cen}_{\mu_{A,s}^{(2)}}\cong A$ and let
% %$A$ act on $A^*$ by $(p\cdot\varphi)(u)=\varphi(pu)$, a direct
% %computation shows that
% \[
% y\cdot \Lambda = e,\quad z\cdot f = d,\quad y^2\cdot f = c,\quad
% yz\cdot f = b,\quad y^2z\cdot f = a,\quad 1\cdot f = f
% \]
% so the orbit of $f$ under $\mathrm{Cen}_{\mu_{A,s}^{(2)}}$ contains the whole
% dual basis $\{a,b,c,d,e,f\}$. Hence $A^*$ is a cyclic
% $\mathrm{Cen}_{\mu_{A,s}^{(2)}}$-module. We note that since the Hilbert function is (1,2,2,1), $A$ is canonically graded and it is isomorphic to the apolar algebra of the highest degree term of $\Lambda$ (see \cite[Theorem 3.3]{Elias2012IsomorphismCO}), i.e., $A^*$ can be generated by $f$. 

Finally, the centroid $\mathrm{Cen}_{\mu_{A,s}^{(2)}}$ is isomorphic, as an
      algebra, to $A=\K[y,z]/(y^3,z^2)$, which is local with maximal ideal
      generated by the images of $y$ and $z$. By \Cref{Rmk:SymmetricMultTensor:3} of
      \Cref{Rmk: SymmetricMultTensor} it follows that 
      $\mu_{A,s}^{(2)}$ is not a direct sum: there is no decomposition of
      $A$ as a direct sum $V_1\oplus V_2$ such that $\mu_{A,s}^{(2)}$ splits
      as a sum of polynomials in disjoint sets of variables.
\end{example}

\begin{remark}\label{rmk: MultMatricesFromTensor}
   Let $A=\K[x_1,\ldots,x_n]/\operatorname{Ann}(\Lambda)$ be an Artinian Gorenstein algebra with basis $B=\{1,x_1,\ldots,x_n,b_{n+1},\ldots,b_{r-1}\}$. Let $\mu_{A,s}^{(d-1)}\in \K_{\mathrm{dp}}[1^*,x_1^*,\ldots,x_n^*,b_{n+1}^*,\ldots,b_{r-1}^*]_d$ be its multiplication tensor with respect to the dual basis of $B$. We can reconstruct the matrices $M_{x_i}$ of the map of multiplication by $x_i$ in $A$ from $\mu_{A,s}^{(d-1)}$, even without knowing $\Lambda$, as follows. First, since the set $\{b\aprod \Lambda \; | \; b\in B\}\subseteq A^*$ is linearly independent, each element $b_j^*\in B^*$ can be written as
   \begin{equation}\label{eq:ChangeOfDualBasis}
       b_j^*=\sum_l \gamma_l b_l\aprod \Lambda.
   \end{equation}
This expression corresponds to a linear endomorphism on $A^*$, whose inverse, which expresses each $b_j\aprod \Lambda$ as a linear combination of the dual basis of $B$, has matrix 
\[(\Lambda(bb'))_{b,b'\in B}=(\Lambda(1\cdots 1\cdot bb'))_{b,b'\in B}=\left((1^{d-2}\aprod \mu_{A,s}^{(d-1)})(bb')\right)_{bb'\in B}.\]
This corresponds to the symmetric matrix associated to the quadric $1^{d-2}\aprod \mu_{A,s}^{(d-1)}$, and it is the matrix $H_\Lambda^B$ in \Cref{Section:The cactus algorithm}). Next, we build the transpose of $M_{x_i}$, that is, the map of contraction by $x_i,$ using that for each $b_j^*\in B^*$ 
\begin{equation}\label{eq: MultMatrcesFromTensor}
(x_i\aprod (b_j\aprod\Lambda))(b_k)= (x_i\aprod \Lambda)(b_j\cdot b_k)=(1^{d-3}\cdot x_i)\aprod \Lambda(b_jb_k)=(1^{d-3}\cdot x_i)\aprod \mu_{A,s}^{(d-1)}(b_jb_k)    
\end{equation}
As before, this is represented by the matrix associated to the bilinear form $(1^{d-3}\cdot x_i)\aprod \mu_{A,s}^{(d-1)}(b_jb_k)$. To obtain the matrix $M_{x_i}^{\mathrm{tr}}$ in the basis dual to $B$, we need to apply the change of basis to $1^{d-3}\cdot x_i\aprod \mu_{A,s}^{(d-1)}$. Thus, we arrived at the following formula: 
\begin{equation}\label{eq:multMatricesFromTensor}
    M_{x_i}^{\mathrm{tr}}=M_{1^{d-3}\cdot x_i\aprod \mu_{A,s}^{(d-1)}} \cdot \left( M_{1^{d-2}\aprod \mu_{A,s}^{(d-1)}}\right)^{-1},
\end{equation}

where $M_{1^{d-2}\aprod \mu_{A,s}^{(d-1)}}$ and $M_{1^{d-3}x_i\aprod \mu_{A,s}^{(d-1)}}$ are the matrices associated to $1^{d-2}\aprod \mu_{A,s}^{(d-1)}$ and $1^{d-3}x_i\aprod \mu_{A,s}^{(d-1)}$, respectively.
\end{remark}

\begin{example}
    We continue with the previous \Cref{ex:ArtinianGorensteinPart2}. We have the tensor 
    \[\mu_{A,s}^{(2)}=a^{(2)}d + a^{(2)}f + ab^{(2)}+acd+abe+b^{(2)}c\]
       in the dual basis of $\{1,y,z,y^2,yz,y^2z\}$.
    We follow \Cref{rmk: MultMatricesFromTensor} to find the matrices corresponding to multiplication by the variables.
    \[1\aprod \mu_{A,s}^{(2)}=ad+af+b^{(2)}+cd+be \quad \sim\left(\!\begin{array}{cccccc}
      0&0&0&1&0&1\\
      0&1&0&0&1&0\\
      0&0&0&1&0&0\\
      1&0&1&0&0&0\\
      0&1&0&0&0&0\\
      1&0&0&0&0&0
      \end{array}\!\right)\]

   \[y\aprod \mu_{A,s}^{(2)}=ab+ae+bc \quad \sim\left(\!\begin{array}{cccccc}
      0&1&0&0&1&0\\
      1&0&1&0&0&0\\
      0&1&0&0&0&0\\
      0&0&0&0&0&0\\
      1&0&0&0&0&0\\
      0&0&0&0&0&0
      \end{array}\!\right)\]

    \[z\aprod \mu_{A,s}^{(2)}=ad+b^{(2)} \quad \sim\left(\!\begin{array}{cccccc}
      0&0&0&1&0&0\\
      0&1&0&0&0&0\\
      0&0&0&0&0&0\\
      1&0&0&0&0&0\\
      0&0&0&0&0&0\\
      0&0&0&0&0&0
      \end{array}\!\right)\]

    Using \eqref{eq: MultMatrcesFromTensor}, we obtain 
    \[M_y^T=
\begin{pmatrix}
0 & 1 & 0 & 0 & 0 & 0 \\
0 & 0 & 0 & 1 & 0 & 0 \\
0 & 0 & 0 & 0 & 1 & 0 \\
0 & 0 & 0 & 0 & 0 & 0 \\
0 & 0 & 0 & 0 & 0 & 1 \\
0 & 0 & 0 & 0 & 0 & 0
\end{pmatrix} \quad M_z^T= 
\begin{pmatrix}
0 & 0 & 1 & 0 & 0 & 0 \\
0 & 0 & 0 & 0 & 1 & 0 \\
0 & 0 & 0 & 0 & 0 & 0 \\
0 & 0 & 0 & 0 & 0 & 1 \\
0 & 0 & 0 & 0 & 0 & 0 \\
0 & 0 & 0 & 0 & 0 & 0
\end{pmatrix}
\]
In this case we know the dual generator of $A$ from \Cref{ex:ArtinianGorensteinPart2}, so the entries can be easily verified.
\end{example}

\section{The cactus algorithm by degree extension}\label{Section:The cactus algorithm}

In this section we recall the cactus algorithm based on degree extension, in the formulation of \cite{Alessandra}.
Starting from a homogeneous form $F\in \Sym^d(V)$ (or its divided-power counterpart), the method searches for a finite-dimensional Artinian Gorenstein algebra apolar to $F$ by completing suitable Hankel (catalecticant) data.
Once a flat extension is found, the associated multiplication matrices encode the candidate algebra and allow one to recover a 0-dimensional scheme apolar to $F$, thus producing the cactus rank.

\begin{definition}
    For $F\in S_d^*$, the \emph{cactus rank} of $F$ is the minimal length of a scheme apolar to $F$. The \emph{local cactus rank} is the minimal length of a local apolar scheme.
\end{definition}

The algorithm described in \cite{Alessandra}, which computes apolar schemes to a given form $F=\sum_{|\beta|=d}f_\beta X^{(\beta)} \in S_d^*=\K_{\mathrm{dp}}[Y_0,\ldots,Y_n]_d$, such that the associated Artinian Gorenstein algebra $A=\K[x_1,\ldots,x_n]/\operatorname{Ann}(\Lambda)$ has a dual generator $\Lambda\in R^*$ extending $f:=F(X_0=1)\in R_{\leq d}^*$. Therefore the representation of $\Lambda$ in the divided power basis
\[\Lambda=\sum_{\alpha\in \mathbb N^{n}} \Lambda(x^\alpha) Y^{(\alpha)}\]
is parametrized as 
\begin{equation}\label{eq: extenionOfF}
    \Lambda(x^{\alpha})=\begin{cases}
    f_{\beta} \quad \text{for } \beta=(d-|\alpha|, \alpha_1,\ldots,\alpha_n)  \quad \text{if } |\alpha|\leq d \\
    h_\alpha \quad \text{else}
\end{cases}.
\end{equation}

After a guess on the monomial basis $B$ of $r$ elements of the algebra $R/\operatorname{Ann}(\Lambda)$ (see \Cref{Section 6}), 
one builds the \emph{Hankel operator} of $\Lambda\in R^*$, that is the map $H_\Lambda:R\to R^*$, $p\mapsto p\aprod \Lambda$, and the Hankel operator restricted to $\langle B\rangle $ and its dual, denoted $H_\Lambda^B$. We can then obtain the matrices $M_{x_i}(h)$ of the maps of multiplication by $x_i$ in the algebra $A$ using the identity 

\begin{equation}\label{eq: MultMatrixHankel}
    H_{x_i\aprod \Lambda}^B=M_{x_i}^t\circ H_\Lambda^B
\end{equation}

(cf. \cite[Proposition 3.9]{Bernard}). 
%These yield the parametrized border relations, which generate the ideal $I=\operatorname{Ann}(\Lambda)\subseteq R=\K[x_1,\ldots,x_n]$. By \cite{NormalForm}, $B$ is a basis of the quotient $A=R/I$ if and only if the matrices $M_{x_i}(h),M_{x_j}(h)$ commute for all $i$ and $j$. By construction, $A$ is a Gorenstein Artinian algebra with dual generator $\Lambda$. This yields the following algorithm (see \cite{Alessandra}):
Since the matrix $M_{x_i}$ corresponds to the multiplication by $x_i$ in the algebra $\K[x_1,\ldots,x_n]/\operatorname{Ann}(\Lambda)$, $M_{x_i}$ commutes with $M_{x_j}$ for all $i,j=1,\ldots,n$. This yields the following algorithm (see \cite{Alessandra}):
\medskip

\vspace{0.5cm}
\begin{mdframed}[]
\begin{alg}[Cactus rank and decomposition]\label{alg:cactus}\end{alg}
%\begin{center}
%\underline{\textbf{Algorithm:} Cactus rank and decomposition}
%\end{center}
\vspace{0.1cm}
\noindent\textbf{Input:} A degree $d \geq 2$ polynomial $F \in S_d^*$ \\
    \textbf{Output:} Cactus rank of $F$.

\begin{enumerate}
    \item Construct the matrix $H_{\Lambda(h)}$ with parameters $\{h_{\alpha}\}_{\alpha\in \mathbb N^n}$, $|\alpha|>d$.
    \item\label{alg:H:star} Set $r$ as the highest rank of a numerical subminor of $H_{\Lambda(h)}$.
    \item\label{alg:basis} Take $B\subseteq R$ a complete staircase of monomials with $|B|=r$, do: 
    \begin{itemize}
        \item\label{alg:h's} Find $h$'s such that:
        \begin{itemize}
            \item $H_{\Lambda(h)}^B$ has nonzero determinant
            \item The multiplication operators $(M_{x_i})^t$ commute for all $i=1,\ldots,n$.
        \end{itemize}
        \item If found, the cactus rank of $F$ is $r$. If not, go to \Cref{alg:basis} with another choice of bases $B$. If all choices of $B$ with $|B|=r$ have been already performed, go to \Cref{increase:r}.
    \end{itemize}
    \item\label{increase:r} Set $r\to r+1$ and go to \Cref{alg:basis}.
    \end{enumerate}
\end{mdframed}
\vspace{0.5cm}

Once a set of numerical values of $h$ have been found in \Cref{alg:h's}, one can efficiently identify the support and structure of the scheme defined by the Artinian Gorenstein algebra. We discuss it in \Cref{Section:Scheme structure from multiplication matrices}.

\section{The cactus rank through multiplication tensors}\label{Section4}

In this section we show that apolarity of 0-dimensional schemes can also be naturally expressed in terms of multiplication tensors of the corresponding Artinian algebras. 
%where essentially we parametrize the symmetric structure tensor and we impose the existence of a restriction to $F$ (see \ref{Thm: Restriction}), centroid abundance and $1_*$-genericity conditions.

\begin{example}\label{ex: WaringRestriction}
    It is easy to see that the $(d-1)$-th iterated multiplication tensor of the algebra $A=\K^r$ is 
    \[\mu_{A}^{(d-1)}=X_1^{(d)} + X_2^{(d)} + \cdots + X_r^{(d)}\in \Sym^d(A^*),\]
    where $X_i$ is the dual element of the canonical vector $e_i$.
    Now let $F\in S_d^*=\K_{\mathrm{dp}}[X_0,\ldots, X_n]_d$ and assume we have a Waring decomposition
    \[F=L_1^{(d)} + \cdots + L_r^{(d)}\]
    with $L_i\in S_1^*$ and $r>n+1$. Then we have a surjection $\phi:A^*\to V$ given by $X_i\mapsto L_i$, such that $F$ is the image of $F'$ under the map induced by the \emph{restriction} $\phi$. 
\end{example}

In \cite{Buczynska_Buczynski__border}, the authors show the existence of this restriction when a general tensor $T$ has border rank at most $m$. Geometrically, this shows that we can produce all symmetric tensors lying in (the cone of) the $m$-th secant variety of the Segre variety by considering only those living in $V_1\otimes \cdots \otimes V_d$ with $\dim V_i=m$ (called \emph{tensors of minimal border rank}), and then produce all possible restrictions.  

\medskip

In analogy with the border rank, we define the \emph{minimal cactus rank}.

\begin{definition} We
say that a concise polynomial $F\in \Sym^d V$ has \emph{minimal cactus rank} if the cactus rank of $F$ is equal to the dimension of $V$.
\end{definition}

The next result %, which is from a private communication with J. Jelisiejew, 
is the symmetric version of \cite[Theorem 9.1]{MinBrk}.

\begin{proposition}\label{prop: ApolarIsMultTensor}
    Let $G\in S^dV^*$ concise, $\dim_\K V=r$ and $Z=\operatorname{Spec}(R)$ with $R$ an Artinian Gorenstein algebra of length $r$ defining a zero dimensional scheme $Z$ in $\mathbb PV$. The following are equivalent:
    \begin{itemize}
        \item $G\in \langle \nu_d(Z)\rangle$, that is, $Z$ is apolar to $G$.
        \item $G$ is isomorphic to the $(d-1)$-th iterated structure tensor of $R$.
    \end{itemize}
\end{proposition}

\begin{proof}
    For every $e\geq 1$ we have the exact sequence of sheaves 
    \[0\to \mathcal{I}_Z(e)\to \mathcal{O}_{\mathbb PV}(e)\to \mathcal{O}_{Z}(e)\to 0,\]
    which induces the exact sequence of global sections
    \[0\to H^0(\mathcal{I}_Z(e))\to H^0(\mathcal{O}_{\mathbb PV}(e))\to H^0(\mathcal{O}_{Z}(e)).\]
    Now if $G$ is concise and $G\in \langle \nu_d(Z)\rangle$, then  
    %By definition, $\langle\nu_d(Z)\rangle$ is the perpendicular space (with respect to the canonical duality pairing) to $I(Z)_d$, that is, of 
    the kernel of the map
    \begin{equation}\label{eq: LineBundles}
        H^0(\mathcal{O}_{\mathbb PV}(1))\to H^0(\mathcal{O}_{Z}(1))
    \end{equation}
    %\[ H^0(\mathcal{O}_{\mathbb PV}(d))\to H^0(\mathcal{O}_{Z}(d)).\]
    is trivial. Indeed, if $\alpha$ is an element of the kernel, that is, an element in $H^0(\mathcal{I}_Z(1))=I(Z)_1$, then we have that 
    \[\alpha\cdot \Sym^{d-1}V^*\in I(Z)_d \subseteq \operatorname{Ann}(G)_d,\]
    where the last inclusion is the definition of $G\in \langle\nu_d(Z)\rangle$. Hence $\alpha\aprod G=0$, which contradicts the conciseness of $G$. Since $Z$ is a finite scheme, for all $e$ we have $H^0(\mathcal{O}_{Z}(e))\cong H^0(\mathcal{O}_{Z})\cong R$, so by a dimension argument we have   
    \[V^*=H^0(\mathcal{O}_{\mathbb PV}(1)\cong H^0(\mathcal{O}_{Z})\cong R.\] 
    %Similarly, $H^0(\mathcal{O}_{\mathbb PV}(d))\cong H^0(\mathcal{O}_{Z}(d))$.
    Thus, the multiplication of sections of line bundles 
    \[\Sym^{d}(H^0(\mathcal{O}_{\mathbb PV}(1))\cong H^0(\mathcal{O}_{\mathbb PV}(d))\]
    is the multiplication in $R$:
    \[R\times \cdots \times R\to R.\]
    Now $G$ is a symmetric multilinear map $G:V^*\times \ldots \times V^*\to \K$, which using $V^*\cong R$ factors as 
    \[\begin{tikzcd}
	{G:R\times \cdots \times R} & \K \\
	R
	\arrow[from=1-1, to=1-2]
	\arrow[from=1-1, to=2-1]
	\arrow[from=2-1, to=1-2]
\end{tikzcd}.\]
Since the vertical map is the multiplication in $R$, $G$ is a multiplication tensor in $R$.

Conversely, fix an isomorphism $\phi:V^*\cong R$ and define $$\Sym^dV^*\to R \quad v_1\otimes\cdots \otimes v_d\mapsto \phi(v_1)\cdots \phi(v_d),$$
which allows us to construct the nondegenerate Veronese embedding $R\hookrightarrow \mathbb P(\Sym^dV)$. Then $G$ vanishes at every polynomial of $I(Z)_d$, that is, $Z$ is apolar to $G$.
\end{proof}

From \Cref{prop: ApolarIsMultTensor} and \Cref{Rmk: SymmetricMultTensor} we can deduce an efficient algorithm for determining whether a concise homogeneous form has minimal cactus rank:

\vspace{0.5cm}
\begin{mdframed}[]
\begin{alg}[Minimal Cactus rank]\label{alg:minimalcactus}\end{alg}
%\begin{center}
%\underline{\textbf{Algorithm:} Cactus rank and decomposition}
%\end{center}
\vspace{0.1cm}
\noindent\textbf{Input:} A degree $d \geq 2$ concise tensor, identified with a polynomial $F \in \K[x_0,\ldots, x_n]$. \\
\textbf{Output:} TRUE if the cactus rank of $F$ is $n+1$, FALSE otherwise.

\begin{enumerate}
    \item Compute $\operatorname{Ann}(F)_{\leq d-1}$ (e.g., via kernels of catalecticant matrices) and the ideal $I=(\operatorname{Ann}(F)_{\leq d-1})$.
    \item Compute the Hessian of $F$, 
    \[\operatorname{Hess}_F=\left( \frac{\partial^2 F }{\partial x_i\partial x_j}\right)_{i,j=0,\ldots, n}.\]
    \item If $\operatorname{HF}(\K[x_0,\ldots, x_n]/I, d)=n+1$ and $\det \operatorname{Hess}_F\neq 0$, then return TRUE. Otherwise return FALSE.
    \end{enumerate}
\end{mdframed}

\bigskip

The same idea of the proof can be used for tensors with nonminimal cactus rank.

\begin{proposition}
    \label{Thm: Restriction}
    Let $F\in S^dV$ be concise and assume that the cactus rank of $F$ is $\leq r$, with $r\geq n+1$. Then there exists a concise tensor $G\in S^d W$ of cactus rank $r$ with $\dim_\K W=r$, and a restriction $\phi:W\to V$ such that $\Sym(\phi)(G)=F$. 
\end{proposition}
\begin{proof}
Let $s \leq r$ and let $Z_s \subset \mathbb P V$ be a zero-dimensional scheme of length s apolar to F. By \cite[Proposition 2.3]{JarekWer}, we can assume $Z_s$ is Gorenstein. If $s<r$, choose $r-s$ distinct reduced points
$P_1,\ldots,P_{r-s}$ disjoint from $Z_s$ and set $Z := Z_s \sqcup \{P_1,\ldots,P_{r-s}\}$.
Then Z is still Gorenstein, and it is still apolar to F since $I(Z)=I(Z_s)\cap I(P_1)\cap\cdots\cap I(P_{r-s})\subseteq I(Z_s)\subseteq \operatorname{Ann}(F)$.
Thus, replacing $Z_s$ by $Z$, we may assume that $Z$ has length $r$.
Write $Z=\operatorname{Spec}(R)$. Since F is concise, as in the proof above, we have that the map \eqref{eq: LineBundles} has zero kernel, so there is an injection 
    \[\phi:V^*\cong H^0(\mathcal{O}_{\mathbb PV}(1))\hookrightarrow H^0(\mathcal{O}_Z (1))\cong R,\]
    and therefore a surjection $R^*\to V$. Again by multiplication of global sections of line bundles, we have the diagram
    \[\begin{tikzcd}
	{F:\Sym^d H^0(\mathcal{O}_{\mathbb PV}(1))} & \K \\
	H^0(\mathcal{O}_{\mathbb PV}(d))
	\arrow[from=1-1, to=1-2]
	\arrow[from=1-1, to=2-1]
	\arrow[from=2-1, to=1-2]
\end{tikzcd}.\]
Via the map $\phi$ we have that this is the restriction of the multiplication tensor of $R$, that is, $\phi$ induces a restriction of $\mu_{A}^{(d-1)}$ equal to $F$. Since $\mu_{R}^{(d-1)}$ is concise, it has cactus rank at least $r$,  and by the previous result $Z$ is apolar to $\mu_{R}^{(d-1)}$, so $\mu_{R}^{(d-1)}$ has cactus rank $r$.
\end{proof}

\begin{comment}
\begin{remark}
Moreover, its $111$-algebra is isomorphic to $R$, and the map above is is fact $A_{111}^{\mu_R^{(d)}}$-multilinear, that is, for every $r\in R$ we have
    \[\mu_R^{(d)}(ra_1,\ldots,a_d)=\mu_R^{(d)}(a_1,ra_2,\ldots,a_d)=\ldots =\mu_R^{(d)}(a_1,\ldots,ra_d)=r\mu_R^{(d)}(a_1,\ldots,a_d)\]

\end{remark}

\begin{remark}
If $A$ is AG with dual generator $\Lambda$, the isomorphisms $A\cong A^*$ induce an isomorphic symmetric tensor in $\Sym^{d+1} R^*$, defined as
\[(a_1,\ldots, a_d)\mapsto \Lambda(a_1\cdots a_d)\]  
\end{remark}
\end{comment}

In the proof above we see that the polynomial $G$ in \Cref{Thm: Restriction} is actually very peculiar.

\begin{corollary}\label{Corollary: ApolarityAndRestriction}
    Let $F\in S_d^*$ and let $\operatorname{Spec}(R)\subseteq \mathbb P V$ be a zero dimensional Gorenstein scheme apolar to $F$ of length $r\geq n+1$. Then there exists a restriction $\phi:R^*\to V$ such that $\Sym(\phi)(\mu_{R}^{(d-1)})=F$.
\end{corollary}

\bigskip
We can illustrate the previous result with an example in a simplified setting. If $F\in S_d^*$ factors as $F=X_0^{(d-k)}\cdot G$ for some $G\in S_k^*$, then $\operatorname{Spec}(R/\operatorname{Ann(F(X_0=1))})$ is a local apolar scheme to $F$ (\cite[Corollary 31]{BJPR}).   

\begin{example}\label{Ex: MultTensorLocal}
    Let $F=XY^{(2)} + Y^{(2)}Z\in \K_{\mathrm{dp}}[X,Y,Z]$ and $f=F(X_0=1)=Y^{(2)} + Y^{(2)}Z$. 
    Consider the Artinian Gorenstein algebra $$A=\K[y,z]/\operatorname{Ann}(Y^{(2)} + Y^{(2)}Z)=\K[y,z]/(y^3,z^2).$$
    As we have seen in \Cref{ex:ArtinianGorensteinPart2}, its structure tensor of order $d$ in the dual of the basis $\{1,y,z,y^2,yz,y^2z\}$ is 
    \[\mu_{A,s}^{(2)}=ab^{(2)}+b^{(2)}c+a^{(2)}d+abe+a^{(2)}f.\]
    %Now by \cite[Proposition~4]{BJPR}, $A$ defines a scheme that is apolar to $F=XY^{(2)}+Y^{(2)}Z\in \K_{\mathrm{dp}}[X,Y,Z]_3$. 
    Consider now 
    %$F=XY^{(2)} + Y^{(2)}Z\in \Sym^3V$, $V=\langle X,Y,Z\rangle$, and 
    the linear restriction $\phi:A^*\to V$ that sends $a\mapsto X$, $b\mapsto Y$, $c\mapsto Z$ and $d,e,f\mapsto 0$. Then we have $F=\Sym^3(\phi)(\mu_{A,s}^{(2)})$, as predicted by \Cref{Corollary: ApolarityAndRestriction}.
\medskip
    Moreover, if we do the same calculation of $\mu_{A,s}^{(d-1)}$ for $d-1$ bigger than the socle degree of $A$ (which is 3 in this case), $\mu_{A,s}^{(d-1)}$ is the homogenization of $\mu_{A,s}^{(2)}$ with respect to $a$ in degree $d$. This is consistent with the fact that $A$ is apolar to a homogenization of $Y^{(2)} + Y^{(2)}Z$ in any degree.

    %Indeed, after the socle degree of $A$, or alternatively the degree of the dual generator in a polynomial representation, all the nontrivial combinations of products in $A$ to obtain the summands of the dual generator have been exhausted, and to get higher degree combinations we can only multiply by one. Then, the same restriction works and we have that $A$ is apolar to $X^{(2)}Y^{(2)} + XY^{(2)}Z$, as expected.
\end{example}

Generalizing this example, we can give a more explicit expression of multiplication tensors of apolar schemes to a polynomial. 

\begin{proposition}\label{Prop: MultTensorExtendsf}
    Let $F\in S_d^*=\K_{\mathrm{dp}}[X_0,\ldots,X_n]_d$ be identified with a concise symmetric tensor, and $Z$ be a minimal apolar scheme to $F$ of degree $r$ defined by $Z=\operatorname{Spec}(A)$. Assume $x_0$ is a nonzero divisor of $S/I(Z)$. Then $r\geq n+1$ and the structure tensor of $A$ is isomorphic to $F+G$, where $G\in 
    %(X_0,\ldots,X_n, X_{n+1},\ldots, X_{r-1})
    (X_{n+1},\ldots, X_{r-1}) \subset \K_{\mathrm{dp}}[X_0 ,\ldots, X_{r-1}]_d
    $. 
\end{proposition}
\begin{proof}
    The fact that $r\geq n+1$ is a consequence of conciseness, see \cite[Theorem 3.9]{Alessandra}. We know that $A$ is Gorenstein (\cite{JarekWer}) with dual generator $\Lambda\in R^*$ extending $f:=F(X_0=1)\in R_{\leq d}$. Now take $B=\{1,x_1,\ldots, x_n,b_{n+1},\ldots, b_{r-1}\}$ a $\K$-basis of $A$. We build the restriction $\phi:A^*\to V$ sending $\phi(1^*)=X_0$, $\phi(x_i^*)=X_i$ and $\phi(b^*)=0$ for the rest of the elements in $B$. We want to show that the form $\mu_{A,s}^{d-1}\in (\Sym(A))^*$ defined by 
    $\mu_{A,s}^{d-1}(a_1,\ldots,a_d)=\Lambda(a_1\cdots a_d)$, which in the divided power basis has the expression
    \begin{equation}\label{eq: MultTensorDivPower}
        \mu_{A,s}^{d-1}= \sum_{|\alpha|=d} \Lambda(1^{\alpha_1}\cdots b_{r-1}^{\alpha_{r-1}}) ((1^*)^{(\alpha_1)}\cdots (b_{r-1}^*)^{(\alpha_{r-1})}),
    \end{equation}
    
    satisfies $\Sym^d(\phi)(\mu_{A,s}^{d-1})=F$. We have 
    \[\Sym^d(\phi)(\mu_{A,s}^{d-1})= \sum_{|\alpha|=d} \Lambda(1^{\alpha_1}\cdots(x_n)^{\alpha_n}) ((\phi(1^*))^{(\alpha_1)}\cdots (\phi(b_{r-1}^*))^{(\alpha_{r-1})})=\]\[ \sum_{|\alpha|=d} \Lambda(1^{\alpha_1}\cdots x_n^{\alpha_n}) (X_0^{(\alpha_1)}\cdots X_n^{(\alpha_n)})=F\]
    where the last equality follows from $\Lambda$ extending $F(X_0=1)$ in degree at most $d$.  

    From \eqref{eq: MultTensorDivPower} we see that $\mu_{A,s}^{(d-1)}$ can be separated into a summand containing only monomials in $1^*,x_1^*,\ldots, x_n^*$, and the rest. Identifying $1^*$ with $X_0$, $x_i^*$ with $X_i$ and $b_j^*$, $j=n+1,\ldots,r-1$, with $X_{j}$, we get the desired expression.   
\end{proof}

\medskip

In view of this result we formalize the notion of extension by variables of a homogeneous polynomial. Let $V$ be a vector space of dimension $n+1$ and $F\in \Sym(V^*)^*$ be concise. Fix an integer $r\ge n+1$ and a vector space $W$ of dimension $r$ containing $V$.
\begin{definition}
\label{def:MT-CMG}
A length-$r$ completion of $F$ with marked generators consists of:
\begin{itemize}
\item an Artinian Gorenstein $\K$-algebra $A$ with $\dim_\K A=r$;
\item an injective linear map (marking)
\[
\iota:\ V^* \hookrightarrow A,\qquad 1_A\in \operatorname{Im}(\iota);\]
\item an identification $W\cong A^*$ such that, under the induced
restriction map $\Sym(W^*)^*\to \Sym(V^*)^*$, the symmetric iterated multiplication tensor of $A$ satisfies
\[
\mu^{(d-1)}_{A,s}\big|_{V^*}=F.
\]
\end{itemize}
We call this a \emph{variable-extension completion} of $F$.
\end{definition}

Thus, in the notation \Cref{Prop: MultTensorExtendsf}, the proposition states that the structure tensor of $A$ is a variable-extension completion of $F$ with marked generators $\iota(x_i)=[x_i]$.

\begin{remark}
\label{rmk:gauge-freedom}
Choosing bases $\{X_0,\ldots,X_n\}$ of $V$ and $\{X_0,\ldots,X_{r-1}\}$ of $W$, the identification $W\cong A^*$ is a choice of basis in $A$, hence solutions are naturally defined
up to the $GL(A)$-action. In particular, the additional coordinates
$\langle X_{n+1},\dots,X_{r-1}\rangle$ are only determined up to $GL(r-n-1)$. 
\end{remark}

\begin{remark}
    The proof of \Cref{Prop: MultTensorExtendsf} shows that, since the dual generator of a Gorenstein algebra computing the cactus rank of $F\in S_d^*$ extends the dehomogenisation of $F$, we can simply take as restriction the one induced by dehomogenisation. On the other hand, we have seen in \Cref{ex: WaringRestriction} that for smooth apolar schemes to $F$ there is also a restriction associated to the decomposition of $F$. In particular, the restriction is not unique. This phenomenon also occurs in nonreduced cases, as the next example shows:
    % For instance, if we take the local algebra in \Cref{Ex: MultTensAG} and make the affine translation $x_i\mapsto x_i-\zeta_i$ for some nonzero $(\zeta_1,\ldots,\zeta_n)\in k^n$, then we have the restriction $\phi:A^*\to V$ in \Cref{Ex: MultTensAG} translated as $\varphi\circ \phi$, where $\varphi: V\to V$ sends $X_0\to X_0-\zeta_1-\ldots - \zeta_n X_n$, and $X_i\mapsto X_i$ for $i=1,\ldots, n$. In this case, the element $1^*$ of the dual algebra is sent to the support of the generalized additive decomposition of $F$. 

    % More generally, when $A=A_1\oplus \cdots \oplus A_m$ defines an apolar scheme to $F$ associated to the generalized additive decomposition $F=\sum_i L_i^{d-k_i}G_i$ (see \cite{GAD}), then we have the restriction $\phi:A^*\to V$ that sends $e_{i}^*$ to $L_i$, where $e_i\in A_i$ and $1_A=e_1+\ldots+e_m$. This follows from \Cref{Rmk: SymmetricMultTensor}, \cite[Proposition 4]{BJPR} and the fact that $A_i=e_{i}\cdot A$. 
\end{remark}

\begin{example}
Let 
\[
F=X_{0}^{(3)}+X_{0}^{(2)}X_{1}+2\,X_{0}X_{1}^{(2)}+3\,X_{1}^{(3)}-X_{0}^{(2)}X_{2}+X_{0}X_{1}X_{2}-3\,X_{1}X_{2}^{(2)}-X_{2}^{(3)}.\]
Then $F$ admits a generalized additive decomposition $F=L_1\cdot G_1 + L_2 \cdot G_2$ for $L_1=X_0+X_1$, $L_2=X_0-X_2$, $G_1=\frac{1}{2}X_0X_1 + X_1^{(2)}- X_2^{(2)}$ and $G_2=\frac{1}{3}X_0^{(2)}-\frac{1}{3}X_0X_2 + X_1X_2 + \frac{1}{3}X_2^{(2)}$. One checks that $Z=\operatorname{Spec}(A)$ with 
    \[A=\K[x_1,x_2]/\operatorname{Ann}(\Lambda) \quad \Lambda=\left(\frac{-1}{2}x_2^2+x_1\right)(\partial)\circ \mathds 1_{(1,1,0)} + (x_2x_1+1)(\partial)\circ \mathds 1_{(1,0,-1)}.\]
    has length 7 and it is apolar to $F$. The principal $7\times 7$ subminor of the Hankel matrix of $\Lambda$ does not vanish, and therefore a possible basis of $A$ is $\{1,x_1,x_2,x_1^2,x_1x_2,x_2^2,x_1^3\}$. In this basis, the structure tensor is

\[
\begin{aligned}
\mu_{A,s}^{(2)} =&
Y_{0}^{(3)}+Y_{0}^{(2)}Y_{1}+2\,Y_{0}Y_{1}^{(2)}+3\,Y_{1}^{(3)}-Y_{0}^{(2)}Y_{2}+Y_{0}Y_{1}Y_{2}-3\,Y_{1}Y_{2}^{(2)}-\\
&Y_{2}^{(3)}+2\,Y_{0}^{(2)}Y_{3}+3\,Y_{0}Y_{1}Y_{3}+ \cdots 
-Y_{5}Y_{6}^{(2)}+9\,Y_{6}^{(3)}
\end{aligned}
\]
As expected from \Cref{Prop: MultTensorExtendsf}, under the substitution $Y_0\mapsto X_0, Y_1\mapsto X_1, Y_2\mapsto X_2, Y_3, \ldots, Y_6\mapsto 0$, $\mu_{A,s}^{(2)}$ is mapped to $F$. This is the restriction corresponding to dehomogenisation. We can obtain a different expression of the multiplication tensor and a different explicit restriction. We observe that 
\[A\cong \K[x_1,x_2]/\operatorname{Ann}(-X_2^{(2)}+X_1) \bigoplus \K[x_1,x_2]/\operatorname{Ann}(X_2X_1+1)=\]\[\K[x_1,x_2]/(x_1^2, x_1x_2, x_2^2+x_1) \bigoplus \K[x_1,x_2]/(x_1^2,x_2^2).\]
A basis of the first algebra is $B_1=\{1,x,y\}$, while $B_2=\{1,x,y,y^2\}$ is a basis of the second. Then we can write the structure tensor of $A$ in the corresponding dual basis as 
\[\mu_{A,s}^{(2)}=W_{0}^{(2)}W_{1}-W_{0}W_{2}^{(2)} +  Z_{0}^{(3)}+Z_{0}Z_{1}Z_{2}+Z_{0}^{(2)}Z_{3}.\]
This is a direct sum expression, where the first summand in the $W$ (respectively $Z$) variables is the multiplication tensor of the first (respectively second) algebra. We take the restriction $\phi=\phi_1+\phi_2: A^*\to V$ with $\phi_1$ defined as $1^*_{B_1}\mapsto X_0+X_1, x^*_{B_1}\mapsto X_1, y^*_{B_1}\mapsto X_2, (y^2)^*_{B_1}\mapsto 0$, and $\phi_2$ as $1^*_{B_2}\mapsto X_0-X_2, x^*_{B_2}\mapsto X_1, y^*_{B_2}\mapsto X_2$. Then $F$ is the image of $\mu_{A,s}^{(2)}$ under $\phi$, and the unit element of each summand of the direct sum of algebras is sent to the support of the generalized additive decomposition.

\end{example}
% \begin{example}
% Take $A=\K[x_1,x_2]/\operatorname{Ann}(X_1^{(2)} + X_1^{(2)}X_2)=k[x_1,x_2]/(x_1^3,x_2^2)$ as in \ref{Ex: MultTensorLocal}, and make the change of coordinates $\varphi$ in $\mathbb PV$ induced by $x_0\mapsto x_0+x_1-2x_2$, $x_1\mapsto x_1$ and $x_2\mapsto x_2$ in $V^*$. By [OurLocalPaper], the scheme $\operatorname{Spec}(A)$ is transformed into $\Tilde{Z}=\operatorname{Spec}(\Tilde{A})$, with 
% \[\Tilde{A}=k[y,z]/\operatorname{Ann}(\Lambda) \quad \quad \Lambda = \left(\frac{1}{2}\partial^2/\partial y^2 + \frac{1}{3}\partial^3/\partial y^2z\right)\circ \mathds{1}_{(-1,2)}, \]
% where $\mathds{1}_{(-1,2)}\in k[y,z]^*$ is the form of evaluation at $(-1,2)$. Then $F=X_0^{(2)}X_1^{(2)} + X_0X_1^{(2)}X_2$ is transformed into 
% \[\Tilde{F}:=\varphi^*(F)=(X_0+X_1-2X_2)\cdot (),\] $\Tilde{Z}$ is apolar to $\Tilde{F}$ and we can compose the restriction of \ref{Ex: MultTensorLocal} $\phi: A^*\to V$ with $\varphi$, obtaining a restriction $\Tilde{\varphi}:A\to V$ that sends $1$ to $L$.   
% \end{example}

\begin{remark}
    By \Cref{Rmk:SymmetricMultTensor:3} and \Cref{prop: ApolarIsMultTensor}, a tensor with minimal cactus rank is a direct sum if and only if the cactus rank is not computed by a local scheme. 
    For example, in \cite[Corollary 1.14]{BBKT_direct_sums} it is shown that if $F\in S_d^*$ and $\operatorname{Ann}(F)$ does not contain nonzero elements of degree 2, then $F$ is not a direct sum. Thus, if in addition $F$  has minimal cactus rank then the cactus rank is computed by a local scheme. 

    %In \cite{BBKT_direct_sums}, it is shown that, since the apolar algebra of a generic polynomial is compressed, a generic polynomial is not a direct sum. By \Cref{Rmk: SymmetricMultTensor}, this means that if we have a generic polynomial with minimal cactus, its local cactus rank and cactus rank agree. Moreover, by \cite[Corollary 1.14]{BBKT_direct_sums}, if $F$ has minimal cactus rank and $\operatorname{Ann}(F)$ does not contain nonzero elements of degree 2, then the cactus rank is also computed by a local scheme. 
\end{remark}

% \begin{remark}
%     If $F\in S_d^*$ has minimal cactus rank, then the scheme revealing the cactus rank is unique (\cite[$\S 9$]{MinBrk}). One easily sees that, in this case, $F$ is isomorphic to the multiplication tensor of $R/\operatorname{Ann}(f)$ for $f=F(X_0=1)$. However, it is not true in general that the minimal apolar scheme is $\operatorname{Spec}(\operatorname{Cen}_F)$, that is, $F$ need not be the multiplication tensor of the algebra that computes its rank. For example, let 
%     \[F=144\,X_{0}^{3}+140\,X_{0}^{2}X_{1}+20\,X_{0}X_{1}^{2}+150\,X_{0}^{2}X_{2}+140\,x_{0}X_{2}^{2}+360\,X_{2}^{3}+140\,X_{0}^{2}X_{3}+360\,X_{0}X_{2}X_{3}+120\,X_{2}^{2}X_{3}+120\,X_{0}X_{3}^{2}+360\,X_{0}^{2}X_{4}+120\,X_{0}X_{2}X_{4}+20\,X_{0}^{2}X_{5}     \]
%     and let $f=F(X_0=1)$. One checks that it has minimal cactus rank (either with \Cref{alg:minimalcactus} or \cite[Example 4.6]{GAD}). The algebra $k[x_1,\ldots,x_5]/\operatorname{Ann}(f)$ has dimension 7, so it does not compute the cactus rank of $F$. Instead, $k[x_1,\ldots,x_5]/\operatorname{Ann}(120x_4^{(4)}+f)$ has length 6 and it is apolar to $f$, so it is the unique scheme of length 6 that defines an apolar scheme to $F$. 
%     \end{remark}

\section{Recovering the cactus algorithm}\label{Sec:Recovering the cactus algorithm}

Fix $F\in S_d^*= \K_{\mathrm{dp}}[X_0,\ldots,X_n]$ a concise symmetric tensor with cactus rank $r>n+1$, and look for a minimal apolar scheme $\operatorname{Spec}(A)\subseteq \mathbb PV$ lying in the first affine chart $x_0\neq 0$, and therefore $A$ is an Artinian Gorenstein algebra that admits a presentation  $A=\K[x_1,\ldots,x_n]/\operatorname{Ann}(\Lambda)$ for some $\Lambda\in R^*=\K[x_1,\ldots,x_n]^*$. Assume we know that $B=\{1,x_1,\ldots,x_n,b_{n+1},\ldots,b_{r-1}\}$ is a monomial basis of $A$.
\medskip
Let $\mu_{A,s}^{(d-1)}\in \K_{\mathrm{dp}}[1^*,x_1^*,\ldots,x_n^*,b_{n+1}^*,\ldots,b_{r-1}^*]_d$ be the $(d-1)$-th iterated multiplication tensor of $A$. We can identify $1^*$ with $X_0$, $x_i^*$ with $X_i$ and for $j=n+1,\ldots,r-1$, $b_j^*$  with $X_{j}$. After this identification we can see $\mu_{A,s}^{(d-1)}\in \K_{\mathrm{dp}}[X_0,\ldots,X_n,\ldots,X_{r-1}]_d$.  Let $\mathcal{M}_R$ be the set of monomials in $R_{\leq d}$, and $\mathcal{M}_S$ the corresponding monomials in $\K_{\mathrm{dp}}[X_0,\ldots,X_n,X_{n+1},\ldots,X_{r-1}]_d$. We write $\mu_{A,s}^{(d-1)}$ as 
\[\mu_{A,s}^{(d-1)}= \sum_{m\in \mathcal{M}_S} \lambda_m m + \sum_{m\not \in \mathcal{M}_S} \lambda_m m\]
Remark that the coefficients $\lambda_m$ of the first summand are determined by $F$. Indeed, $\lambda_m=\Lambda(m)$ and we know that $\Lambda|_{R\leq d}=F(X_0=1)$. Thus, only the coefficients of the second summand are not known, and we leave them as variables. We denote by $\mu_{A,s}^{(d-1)}(\lambda)$ the polynomial with these unknown coefficients.
\medskip

From the tensor $\mu_{A,s}^{(d-1)}(\lambda)$ we can build the matrices 
\begin{equation}\label{eq: MatMultLambda}
    \Tilde{M}_{x_i}^t(\lambda) \quad i=1,\ldots,n
\end{equation}
of multiplication by $x_i$ in $A$ as in \Cref{rmk: MultMatricesFromTensor}, parametrized by $\lambda_m$. We compare them to those in \Cref{Section:The cactus algorithm}, $M_{x_i}(h)$, which are obtained from Hankel operators. 

\begin{claim}\label{clm:identification-mult-matrices}
Under the assumptions above, the symbolic multiplication matrices
\[
\widetilde{M}_{x_i}(\lambda),\quad i=1,\ldots,n,
\]
constructed from the tensor $\mu_{A,s}^{(d-1)}(\lambda)$, coincide with the symbolic multiplication matrices
\[
M_{x_i}(h),\quad i=1,\ldots,n,
\]
constructed via the Hankel operators in \Cref{Section:The cactus algorithm}, after a canonical bijection between their sets of variables.

More precisely, there exists a bijection
\[
\Psi: H \longrightarrow \Delta
\]
between the moment variables $h_\alpha$ appearing in the matrices $M_{x_i}(h)$ and the undetermined coefficients $\lambda_m$
appearing in $\widetilde{M}_{x_i}(\lambda)$, such that
\[
M_{x_i}(\Psi(h))=\widetilde{M}_{x_i}(\lambda)
\quad \text{for all } i=1,\ldots,n.
\]
\end{claim}

\begin{proof}
By construction (see \Cref{rmk: MultMatricesFromTensor}), the entries of the matrices $\widetilde{M}_{x_i}(\lambda)$ are given by
values of the linear functional $\Lambda$ on monomials of the form
\[
\Lambda(b_\ell b_j),\qquad \Lambda(x_i b_\ell b_j),
\]
with $i=1,\ldots,n$ and $j,\ell=1,\ldots,r$.
Whenever $\deg(b_\ell b_j)>d$ or $\deg(x_i b_\ell b_j)>d$, these values are not determined by $F$ and therefore appear as independent
variables $\lambda_m$, where $m$ is an element in $\K[x_0,\ldots,x_{r-1}]$, such that the corresponding element in $R$ is $b_\ell b_j$ or $x_i b_\ell b_j$. Moreover, $\lambda_m=\lambda_{m'}$ if $m,m'$ correspond to the same element in $R$.

On the other hand, by \eqref{eq: extenionOfF} and \eqref{eq: MultMatrixHankel}, the symbolic matrices $M_{x_i}(h)$ are obtained from
the same products $b_\ell b_j$ and $x_i b_\ell b_j$, whose evaluations are encoded by the moment variables $h_\alpha$ whenever the
corresponding degree exceeds $d$.

Thus, both constructions introduce one independent variable for each monomial $m\in R$ with $\deg(m)>d$ that appears in the products
defining the multiplication matrices. This yields a canonical bijection
\begin{equation}\label{eq: Psi bijection of variables}
\Psi:H\to \Delta,
\qquad h_{\alpha} \mapsto \lambda_m \ \text{, with } m \text{ corresponding to } x^\alpha,
\end{equation}
which preserves the position of each variable inside the matrices.

With this identification, the two families of matrices have identical numerical entries and identical symbolic entries, and hence
\[
M_{x_i}(\Psi(h))=\widetilde{M}_{x_i}(\lambda)
\quad \text{for all } i=1,\ldots,n,
\]
as claimed.\qedhere
\end{proof}

%We claim that, as symbolic matrices, they are equal, meaning that they have the same numerical entries and a bijective correspondence between their variables. 
%
%\medskip

%To see this, let $H$ be the set of moment variables $h_\alpha$ appearing in $M_{x_i}(h)$, $i=1,\ldots n$, and $\Delta$ the set of variables $\lambda_m$ appearing as unknowns of $\Tilde{M}_{x_i}(\lambda)$, $i=1,\ldots n$. We observe that by \Cref{rmk: MultMatricesFromTensor}, the entries of $\Tilde{M}_{x_i}$ are determined by $\Lambda(b_lb_j)$ and $\Lambda(x_ib_lb_j)$, $i=1,\ldots,n$ and $l,j=1,\ldots,r$, which are coefficient variables $\lambda_m$ if $\deg x_ib_lb_j>d$ and $\deg b_lb_j>d$, respectively. By \eqref{eq: extenionOfF} and \eqref{eq: MultMatrixHankel}, the variables in $H$ are determined in the same way. Hence, the map
%\begin{equation}\label{eq: Psi bijection of variables}
%    \Psi:H\to \Delta \quad h_{\alpha} \mapsto \lambda_m \text{ for } m=x^\alpha
%\end{equation}
%is a bijection, and moreover we also have 
%\[M_{x_i}(\Psi(h))=\Tilde{M}_{x_i}(\lambda) \quad  i=1,\ldots,n.\]

Therefore, by imposing the commutation of the matrices $\Tilde{M}_{x_i}$ parametrized by $\lambda_m$ we recover the cactus algorithm in \Cref{Section:The cactus algorithm}. 

\begin{remark}\label{rmk: det Hankel vs contraction}
    Keeping the same notation, by \Cref{rmk: MultMatricesFromTensor} the matrix $H_{\Lambda}^B=(\Lambda(bb'))_{b,b'\in B}$ is the matrix associated to $1^{d-2}\aprod \mu_{A,s}^{(d-1)}$. In particular, the condition $\det H_{\Lambda}^B\neq 0$ is equivalent to $1^{d-2}\aprod \mu_{A,s}^{(d-1)}$ being concise. 
\end{remark}

% \begin{claim}\label{rmk: det Hankel vs contraction}
%     Let $\Lambda\in R^*=\K[x_1,\ldots,x_n]$ and let $B$ be a set of monomials in $R$ containing $1$ and $V:=\langle B\rangle$. Let $G:V\times \cdots \times V\to k$ be a tensor defined by $G(a_1,\ldots,a_d)=\Lambda(a_1\cdots a_d)$. Then the matrix $H_{\Lambda}^B=(\Lambda(b,b'))_{b,b'\in B}$ is the matrix associated to $1^{d-2}\aprod G$. In particular, the condition $\det H_{\Lambda}^B\neq 0$ is equivalent to $1^{d-2}\aprod G$ being concise. Moreover, the multiplication matrices of the algebra can be computed as  
%     \begin{equation}
%     \widetilde{M}_{x_i}(\lambda)^t\cdot (1^{d-2} \aprod G)=(1^{d-3}x_i \aprod G)
% \end{equation}

%     \end{claim}

% \begin{proof}
%     For every $b,b'\in B$, $1^{d-2}\aprod \Lambda(b,b')=\Lambda(1,\ldots,1,b,b')=\Lambda(b,b')$, so the two matrices agree. The formula for the multiplication matrices follows from \eqref{eq: MultMatrixHankel}.

% \end{proof}

% \begin{proof}
%     The Hankel matrix of $\Lambda$ restricted to $V$ is invertible is the map $a\mapsto a\aprod \Lambda$ is injective, if and only if $a\mapsto a\aprod (1\aprod \Lambda)$ is injective, if and only $G(1,1,1,\ldots,1,a,-)$ is non zero for all $a\in V$, if and only if $1^{d-2}\aprod G$ is concise.    

% \end{proof}

We can compile all the previous results as follows:

\begin{theorem}\label{thm: Summary}
    Let $F\in S_d^*= \K_{\mathrm{dp}}[X_0,\ldots,X_n]_d$ be a concise homogeneous polynomial. Let $B=\{1,x_1,\ldots,x_n,\allowbreak b_{n+1},\ldots,b_{r-1}\}$ be a complete staircase (see \Cref{Def: CompleteStaircase}) of $r\geq n+1$ monomials, and identify $1^*$ with $X_0$, $x_i^*$ with $X_i$ and $b_j^*$ with $X_{j}$ for $j=n+1,\ldots,r-1$. Let
    \[G= F+ \sum_{m}\lambda_m m \in \K_{\mathrm{dp}}[X_0,\ldots,X_n,\ldots,X_{r-1}]_d,\]
    where $m$ runs over all monomials that contain at least one of $X_j$, $j=n+1,\ldots ,r-1$. Set $\lambda_m=F(x_0^{d-\deg m'} m')$ if $\deg m'\leq d$, where $m'$ is the associated monomial of $m$ in $\K[x_1,\ldots,x_n]$. Let $\Tilde{M}_{x_i}(\lambda)$ be the parametrized multiplication matrix by $x_i$ \eqref{eq: MatMultLambda} obtained from $G$. Then the following are equivalent:
    \begin{enumerate}[label=\roman*)]
        \item\label{item:i} There exists an apolar scheme defined by an Artinian algebra with basis $B$. In particular, $F$  has cactus rank at most $r$.
        \item\label{item:ii} There exist numerical values of the unknowns $\lambda_m$ such that $G$ is a multiplication tensor; equivalently, after the conciseness condition is imposed, G is centroid abundant and has non-vanishing Hessian.
        \item\label{item:iii} For some values of $\lambda_m$, the matrices $\Tilde{M}_{x_i}(\lambda), \Tilde{M}_{x_j}(\lambda)$ commute for all $i,j=1,\ldots,n$ and $x_0^{d-2}\aprod G$ is concise.  
    \end{enumerate}
    \end{theorem}
  
\begin{proof}
\Cref{item:i} $\iff$ \Cref{item:ii} follows from \Cref{Rmk: SymmetricMultTensor}, \Cref{prop: ApolarIsMultTensor}, \Cref{Thm: Restriction} and \Cref{Prop: MultTensorExtendsf}. 

\Cref{item:i} $ \iff$ \cref{item:iii} is \cite[$\S$ 6.1]{Alessandra} together with \Cref{rmk: det Hankel vs contraction}. 

Even if it is not necessary, we like to give also a clarifying proof of \Cref{item:ii} $\iff$ \Cref{item:iii} in the language of multiplication tensors.

Assume \Cref{item:ii}. Then $G=\mu_{A,s}^{(d-1)}$ for some Artinian Gorenstein algebra $A=\K[x_1,\ldots,x_n]/\operatorname{Ann}(\Lambda)$ with basis $B$. Hence, by the previous discussion the coefficient of the monomial $m$ in $G$ is $\Lambda(m')$, and the matrix $\Tilde{M}_{x_i}$ correspond to the multiplication by $x_i$ in $A$. Since $\operatorname{Cen}_G\cong A$ is a commutative algebra and $(x_i,\ldots,x_i)\in \operatorname{Cen}_G$, the matrices $\Tilde{M}_{x_i}(\lambda)$ and $\Tilde{M}_{x_j}(\lambda)$ commute for all $i,j$. Moreover, $A$ is generated by $1$ as a $A\cong \operatorname{Cen}_G$-module and $G$ is concise, hence by $\operatorname{Cen}_G$ multilinearity the flattening $\mu_{A,s}^{(d-1)}(1,\ldots,1,-,-)$ is concise, i.e., $x_0^{d-2}\aprod G$ is concise. 

Conversely, if $\lambda_0$ is a set of solutions for the commutation of the matrices $\Tilde{M}_{x_i}(\lambda)$, then by \cite[Theorem 4.2]{Bernard} the data $(B, \{\Tilde{M}_{x_i}(\lambda_0)\}_i)$ defines an Artinian Gorenstein algebra $A$ of length $r$, that admits a presentation $A=\K[x_1,\ldots,x_n]/\operatorname{Ann}(\Lambda)$ with basis $B$, for some $\Lambda\in \K[x_1,\ldots,x_n]^*$. Thus, every $b\in B$ defines an endomorphism in $A^*$ in the centroid of $G$, so $G$ is centroid abundant. Finally, $x_0^{d-2}\aprod G$ concise implies $G$ is concise and generic at every coordinate, so by \cite[Proposition 5.3]{MinBrk} $G$ is a multiplication tensor. 
\begin{comment}
Non-vanishing Hessian is equivalent to require that there exists some $a\in A$ such that the map 
        $a\aprod \mu_{A,s}:A\times\cdots\times A\to k$
        is non degenerate. By centroid multilinearity, this reduces to $$A\times A\to k \quad (b,b')\mapsto (a\aprod \Lambda)(bb')=(bb'\aprod \Lambda)(a)$$
        Then $A$ is cyclic $A^*$ module generated by $a$, but since $A$ is Gorenstein with dual generator $\Lambda$, we have $a=1$. That is, if $B$ is a basis of $A$, we want 
        \[\det((bb'\aprod \Lambda(1))_{b,b'\in B})=\det((\Lambda(bb'))_{b,b'\in B})\neq 0\]
        Which is precisely the condition $\det H_{\Lambda}^{B,B}\neq 0$. By \cite{Bernard}, this is equivalent to $B$ being a basis of $A$.    
    \begin{itemize}
        \item \textbf{Imposing $\det H_{\Lambda}^{B,B}\neq 0$ is equivalent to non-vanishing Hessian:} 
        \item  \textbf{Commutation of matrices is equivalent to centroid abundant}
        With the notation of the paragraph above, if the matrices $M_{x_i}$ commute then $B$ is a $k$-basis of $A$, and therefore for every $b\in B$, the element $(b^t,\ldots,b^t,b^t)$ is an element of the centroid (see Remark). Conversely, $\operatorname{Cen_T}$ is centroid abundant, then
        \[\{(b^t,\ldots,b^t,b^t) \; | \; b\in B\}\]
        is a set of linearly independent elements of the centroid. Since $\operatorname{Cen}_T$ is a commutative algebra, $(x_i^t,\ldots,x_i^t,x_i)$ and $(x_j^t,\ldots,x_j^t,x_j)$ commute for all $i,j$, that is the multiplication matrices commute.
    \end{itemize}
\end{comment}
\end{proof}

\begin{example}\label{ex:hankel-vs-mult}
Let $k$ be a field of characteristic $> 3$ and set
\[
S=\K[x_0,x_1,x_2],\qquad R=\K[x,y]\ \text{with}\ x=\frac{x_1}{x_0},\ y=\frac{x_2}{x_0}.
\]
Consider the cubic (in divided powers notation)
\[
F \;=\; 4\,X_0^{(3)}+4\,X_0^{(1)}X_1^{(2)}+4\,X_0^{(1)}X_2^{(2)}
\ \in\ \K_{\mathrm{dp}}[X_0,X_1,X_2]_3.
\]
The form $F$ determines a truncated functional $\Lambda|_{R_{\le 3}}$ by
\[
\Lambda(x^a y^b)\;:=\;F\big(x_0^{3-(a+b)}x_1^a x_2^b\big),\qquad a+b\le 3.
\]
Hence
\[
\Lambda(1)=4,\qquad \Lambda(x)=\Lambda(y)=0,\qquad 
\Lambda(x^2)=\Lambda(y^2)=4,\qquad \Lambda(xy)=0,
\]
and moreover $\Lambda(x^3)=\Lambda(x^2y)=\Lambda(xy^2)=\Lambda(y^3)=0$.

Let $B=\{1,x,y,xy\}$ be a complete staircase monomial set of size $r=4$.

\medskip
\noindent\textbf{(A) Hankel approach.}
Introduce the unknown moments of degree $>3$ that will appear in the Hankel construction:
\[
u:=\Lambda(x^2y^2),\qquad
b:=\Lambda(x^3y),\qquad
c:=\Lambda(x^3y^2),\qquad
d:=\Lambda(xy^3),\qquad
e:=\Lambda(x^2y^3).
\]
The Hankel matrix with respect to $B$ is
\[
H^{B,B}_\Lambda=\big(\Lambda(bb')\big)_{b,b'\in B}
=
\begin{pmatrix}
4&0&0&0\\
0&4&0&0\\
0&0&4&0\\
0&0&0&u
\end{pmatrix}.
\]
The shifted Hankel matrices are
\[
H^{xB,B}_\Lambda=\big(\Lambda(x\,bb')\big)_{b,b'\in B}
=
\begin{pmatrix}
0&4&0&0\\
4&0&0&b\\
0&0&0&u\\
0&b&u&c
\end{pmatrix},
\qquad
H^{yB,B}_\Lambda=\big(\Lambda(y\,bb')\big)_{b,b'\in B}
=
\begin{pmatrix}
0&0&4&0\\
0&0&0&u\\
4&0&0&d\\
0&u&d&e
\end{pmatrix}.
\]
The multiplication matrices are defined by
\[
M_x(h)^t=(H^{B,B}_\Lambda)^{-1}H^{xB,B}_\Lambda,\qquad
M_y(h)^t=(H^{B,B}_\Lambda)^{-1}H^{yB,B}_\Lambda,
\]
so (since $H^{B,B}_\Lambda$ is diagonal) we get the \emph{symbolic} matrices
\begin{equation}\label{eq:example-Mx-My-hankel}
M_x(h)=
\begin{pmatrix}
0&1&0&0\\
1&0&0&\frac{b}{4}\\
0&0&0&\frac{u}{4}\\
0&\frac{b}{u}&1&\frac{c}{u}
\end{pmatrix},
\qquad
M_y(h)=
\begin{pmatrix}
0&0&1&0\\
0&0&0&\frac{u}{4}\\
1&0&0&\frac{d}{4}\\
0&1&\frac{d}{u}&\frac{e}{u}
\end{pmatrix}.
\end{equation}

\medskip
\noindent\textbf{(B) Iterated multiplication tensor approach.}
Identify
\[
1^*\leftrightarrow X_0,\qquad x^*\leftrightarrow X_1,\qquad y^*\leftrightarrow X_2,\qquad (xy)^*\leftrightarrow X_3.
\]
Let $G=\mu^{(2)}_{A,s}(\lambda)\in \K_{\mathrm{dp}}[X_0,X_1,X_2,X_3]_3$ be the symmetric iterated multiplication tensor
associated with $(A,B)$.
Equivalently, $G$ is the cubic whose coefficients are the triple products
\[
\text{coeff}(X_iX_jX_k)=\lambda_{x_ix_jx_k}=\Lambda(b_i b_j b_k),\qquad b_0=1,\ b_1=x,\ b_2=y,\ b_3=xy.
\]
The coefficients not involving $X_3$ are fixed by $F$ (i.e.\ by $\Lambda|_{R_{\le 3}}$), while the remaining ones
are unknown and denoted by $\lambda$.
Among them, the only coefficients needed to build the multiplication matrices in the basis $B$ are:
\begin{equation}\label{eq: ex:symbolicBijectionVariables}
\lambda_{X_0X_3^2}=\lambda_{X_1X_2X_3}=\Lambda((xy)^2)=\Lambda(x^2y^2)=u,\qquad
\lambda_{X_1^2X_3}=\Lambda(x^2\cdot xy)=\Lambda(x^3y)=b,
\end{equation}
\[
\lambda_{X_2^2X_3}=\Lambda(y^2\cdot xy)=\Lambda(xy^3)=d,\qquad
\lambda_{X_1X_3^2}=\Lambda(x\cdot (xy)^2)=\Lambda(x^3y^2)=c,\]\[
\lambda_{X_2X_3^2}=\Lambda(y\cdot (xy)^2)=\Lambda(x^2y^3)=e.
\]

% \begin{equation}\label{eq: ex: symbolicBijectionVariables}
% \lambda_{x^2y^2}=\text{coeff}(X_0X_3^{(2)})=\text{coeff}(X_1X_2X_3)=\Lambda(x^2y^2)=u,
% \end{equation}
% \[
% \lambda_{x^3y}=\text{coeff}(X_1^{(2)}X_3)=\Lambda(x^2\cdot xy)=\Lambda(x^3y)=b\]\[\lambda_{xy^3}=\text{coeff}(X_2^{(2)}X_3)=\Lambda(y^2\cdot xy)=\Lambda(xy^3)=d\]\[
% \lambda_{x^3y^2}=\text{coeff}(X_1X_3^{(2)})=\Lambda(x\cdot (xy)^2)=\Lambda(x^3y^2)=c\]\[
% \lambda_{x^2y^3}=\text{coeff}(X_2X_3^{(2)})=\Lambda(y\cdot (xy)^2)=\Lambda(x^2y^3)=e.
% \]
We use \Cref{rmk: det Hankel vs contraction} to construct the multiplication matrices $\widetilde M_x(\lambda),\widetilde M_y(\lambda)$ from $G$:
\[x_0\aprod G = 4\,X_0^{(2)}+4\,X_1^{(2)}+4\,X_2^{(2)}+\lambda_{x_0x_3^2}X_3^{(2)} \; \to \; M_{x_0\aprod G}:=\left(\!\begin{array}{cccc}
      4&0&0&0\\
      0&4&0&0\\
      0&0&4& 0\\
      0 & 0 & 0 & \lambda_{x_0x_3^2}
      \end{array}\!\right)
\]
\[x_1\aprod G = 4\,X_0X_1+\lambda_{x_1^2x_3}X_1X_3+\lambda_{x_1x_2x_3}X_2X_3+\lambda_{x_1x_3^2}X_3^{(2)} \; \to \; M_{x_1\aprod G}:= \left(\!\begin{array}{cccc}
      0&4&0&0\\
      4&0&0&\lambda_{x_1^2x_3}\\
      0&0&0& \lambda_{x_1x_2x_3}\\
      0 & \lambda_{x_1^2x_3} & \lambda_{x_1x_2x_3} & \lambda_{x_1x_3^2}
      \end{array}\!\right)
\]

\[x_2\aprod G = 4\,X_0X_2+\lambda_{x_1x_2x_3}X_1X_3+\lambda_{x_2^2x_3}X_2X_3+\lambda_{x_2x_3^2}X_3^{(2)} \; \to \; M_{x_2\aprod G}:=\left(\!\begin{array}{cccc}
      0&0&4&0\\
      0&0&0&\lambda_{x_1x_2x_3}\\
      4&0&0& \lambda_{x_2^2x_3}\\
      0 & \lambda_{x_1x_2x_3} & \lambda_{x_2^2x_3} & \lambda_{x_2x_3^2}
      \end{array}\!\right)
\]
We get the multiplication operators $    \widetilde{M}_{x_i}(\lambda)$ using 
$ \widetilde{M}_{x_i}(\lambda)^t\cdot M_{x_0\aprod G}=M_{x_i\aprod G}
$.

% One obtains
% \begin{equation}\label{eq:example-Mx-My-tensor}
% \widetilde M_x(\lambda)=
% \begin{pmatrix}
% 0&1&0&0\\
% 1&0&0&\frac{\lambda_{X_1^2X_3}}{4}\\
% 0&0&0&\frac{\lambda_{X_0X_3^2}}{4}\\
% 0&\frac{\lambda_{X_1^2X_3}}{\lambda_{X_0X_3^2}}&1&\frac{\lambda_{X_1X_3^2}}{\lambda_{X_0X_3^2}}
% \end{pmatrix},
% \qquad
% \widetilde M_y(\lambda)=
% \begin{pmatrix}
% 0&0&1&0\\
% 0&0&0&\frac{\lambda_{X_0X_3^2}}{4}\\
% 1&0&0&\frac{\lambda_{X_2^2X_3}}{4}\\
% 0&1&\frac{\lambda_{X_2^2X_3}}{\lambda_{X_0X_3^2}}&\frac{\lambda_{X_2X_3^2}}{\lambda_{X_0X_3^2}}
% \end{pmatrix}.
% \end{equation}

\medskip
\noindent\textbf{Symbolic equality and the bijection of variables.}
Equation \eqref{eq: ex:symbolicBijectionVariables} gives a bijection between the set of variables appearing in $M_{x_i}(h)$ and the ones in $\widetilde{M}_{x_i}(\lambda)$, and modulo this identification of parameters, $M_{x_i}(h)$ agrees with $\widetilde{M}_{x_i}(\lambda)$ for $i=1,2$.
% Let
% \[
% H=\{u,b,c,d,e\},\qquad 
% \Delta=\{\lambda_{X_0X_3^2},\lambda_{X_1^2X_3},\lambda_{X_1X_3^2},\lambda_{X_2^2X_3},\lambda_{X_2X_3^2}\}.
% \]
% Define the map
% \[
% \Psi:H\to \Delta,\qquad
% u\mapsto \lambda_{X_0X_3^2},\ \ 
% b\mapsto \lambda_{X_1^2X_3},\ \ 
% c\mapsto \lambda_{X_1X_3^2},\ \ 
% d\mapsto \lambda_{X_2^2X_3},\ \ 
% e\mapsto \lambda_{X_2X_3^2}.
% \]
% Then $\Psi$ is a bijection and, comparing \eqref{eq:example-Mx-My-hankel} and \eqref{eq:example-Mx-My-tensor}, we have
% \[
% M_x(\Psi(h))=\widetilde M_x(\lambda),\qquad
% M_y(\Psi(h))=\widetilde M_y(\lambda).
% \]
% Hence the two constructions yield the \emph{same symbolic multiplication matrices}: the numerical entries match,
% and the variables correspond bijectively via $\Psi$.

Hence the two constructions yield the \emph{same symbolic multiplication matrices}: the numerical entries match,
and the variables correspond bijectively via this bijection.

\end{example}

\section{Scheme structure from multiplication matrices}\label{Section:Scheme structure from multiplication matrices}

In \cite{Alessandra} commuting multiplication operators are constructed from a (truncated) Hankel extension
$\Lambda$ and a complete staircase basis $B$, and then recover the cactus scheme
$Z=\mathrm{Spec}(A)$ from those operators
via common eigenvectors and joint generalized eigenspaces.

\medskip

In our setting one can reproduce the same pipeline starting from the symmetric iterated multiplication
tensor. Indeed, by \Cref{thm: Summary}, determining a variable-extension completion of $F$ is equivalent to a choice of basis and a numerical specialization of the parameters $\lambda_m$ in $\Tilde{M}_{x_i}(\lambda)$ (cf \Cref{Sec:Recovering the cactus algorithm}). This corresponds to a solution of the commutation of the matrices $M_{x_i}(h)$ coming from Hankel extension after the bijective renaming of variables $\Psi$ in \eqref{eq: Psi bijection of variables}.

\medskip
Consequently, imposing the commutation relations on $\widetilde M_{x_i}(\lambda)$ yields the same
commutative algebra $A$ of length $r$ as in \cite{Alessandra}. Once a numerical specialization
$\lambda=\lambda_0$ is found with $\det H^{B,B}(\lambda_0)\neq 0$, the algebra structure on
$A=\langle B\rangle_k$ (hence the cactus scheme $Z=\mathrm{Spec}(A)$) is already encoded by the multiplication
operators. In particular, one may:
\begin{itemize}
\item recover the support of $Z$ from common eigenvectors of $(\widetilde M_{x_i}(\lambda_0))^t$,
\item recover local lengths and nonreduced structure from joint generalized eigenspaces
(or, equivalently, the Jordan structure of the nilpotent parts on each local block),
\item and write explicit equations for $Z$ using the staircase rewriting relations induced by $B$.
\end{itemize}
Therefore, every reconstruction step of the \cite{Alessandra} cactus algorithm can be carried out
verbatim starting from the iterated multiplication tensor, replacing Hankel moments by the
corresponding tensor coefficients.

\begin{remark}
\label{rmk:jordan-data}
In \cite{Alessandra} the support of $Z$ is recovered from the common rank-$1$ eigenvectors of the commuting family
$\{(M^B_{x_i})^t\}_{i=1}^n$, and the local lengths from joint generalized eigenspaces (equivalently,
from the nilpotent parts on each local block). One may further extract the local Jordan structure by restricting
\[
N_{i}(\zeta) := \big((M^B_{x_i})^t-\zeta_i I\big)\big|_{W_\zeta}
\]
to the joint generalized eigenspace $W_\zeta\subset A^*$ at a point $\zeta$; the sizes of Jordan chains encode
the nonreduced directions of the local scheme at $\zeta$.

A complementary viewpoint (beyond \cite{Alessandra}) is that for a generic linear form $\ell=\sum_i \lambda_i x_i$,
a Schur/Jordan-type factorization of the multiplication by $\ell$ operator yields a block decomposition
corresponding to the local algebras, and the nilpotency indices inside each block recover local invariants;
see \cite{barrilliMourrainTaufer2025GAD} for an explicit block-structured reconstruction framework.
\end{remark}

\begin{example}
We consider the matrices $M_x(h)$, $M_y(h)$ obtained in \Cref{ex:hankel-vs-mult}, where a numerical solution of the commutation equations
$M_x(h)M_y(h)=M_y(h)M_x(h)$ together with $\det H^{B,B}_\Lambda\neq 0$
determines an Artinian algebra structure on $A=\langle B\rangle_k$ with basis
$B=\{1,x,y,xy\}$. In this example, consider the linear functional
\[
\Lambda_0(p)\;:=\;\sum_{\varepsilon,\delta\in\{\pm 1\}} p(\varepsilon,\delta),
\]
which extends the truncated data coming from $F$ and satisfies
\[
u=\Lambda_0(x^2y^2)=4,\qquad b=c=d=e=0.
\]
Substituting into \eqref{eq:example-Mx-My-hankel}, we obtain
\[
M_x(\Lambda_0)=
\begin{pmatrix}
0&1&0&0\\
1&0&0&0\\
0&0&0&1\\
0&0&1&0
\end{pmatrix},
\qquad
M_y(\Lambda_0)=
\begin{pmatrix}
0&0&1&0\\
0&0&0&1\\
1&0&0&0\\
0&1&0&0
\end{pmatrix},
\]
which commute.

\medskip
Following the reconstruction step in the Hankel method (cf. \cite{Alessandra}), we consider common eigenvectors of
$M_x(\Lambda_0)^t$ and $M_y(\Lambda_0)^t$. For each $(\alpha,\beta)\in\{\pm1\}^2$, set
\[
w_{\alpha,\beta}\;:=\;(1,\alpha,\beta,\alpha\beta)^t\ \in k^4,
\]
which represents the evaluation functional $\mathrm{ev}_{(\alpha,\beta)}\in A^*$ in the dual basis
$B^*=\{1^*,x^*,y^*,(xy)^*\}$. A direct computation shows that
\[
M_x(\Lambda_0)^t\,w_{\alpha,\beta}=\alpha\,w_{\alpha,\beta},
\qquad
M_y(\Lambda_0)^t\,w_{\alpha,\beta}=\beta\,w_{\alpha,\beta}.
\]
Hence the simultaneous eigenvalues of the commuting pair $(M_x(\Lambda_0)^t,M_y(\Lambda_0)^t)$
are exactly the pairs $(\alpha,\beta)\in\{(\pm1,\pm1)\}$. Since $\dim_\K A=4$, we obtain four
linearly independent common eigenvectors $w_{\alpha,\beta}$, and therefore the apolar scheme in the
affine chart $x_0\neq 0$ is the reduced set of four points
\[
\mathrm{Spec}(A)=\{(\alpha,\beta)\mid \alpha,\beta\in\{\pm1\}\}\subset \mathbb{A}^2.
\]
Homogenizing yields the corresponding reduced subscheme of $\mathbb{P}^2$ supported on the same four points
in the chart $x_0\neq 0$.
\end{example}

\begin{remark}
\label{rmk:scheme-already-there-BT}

In our setting we can also obtain the explicit generators of the ideal of the scheme from the complete staircase basis $B$ (see \Cref{Def: CompleteStaircase}) and the commuting multiplication operators $M^B_{x_1},\dots,M^B_{x_n}$. 

\medskip
Define the $\K$-algebra morphism
\[
\rho:\ R\longrightarrow \mathrm{End}_\K(A),\qquad x_i\longmapsto M^B_{x_i},
\]
where $A=\langle B\rangle_\K$ with multiplication induced by the matrices.
Then $\ker(\rho)=I_\Lambda$, and $Z=\mathrm{Spec}(R/\ker\rho)$.

\medskip
To obtain explicit generators from the staircase, consider the \emph{border} of $B$,
\[
\partial B:=\{\,x_i b:\ b\in B,\ i=1,\dots,n\,\}\setminus B.
\]
For each border monomial $m=x_i b_j\in\partial B$, the column $j$ of $M^B_{x_i}$ gives the normal form
\[
[x_i b_j]\ =\ \sum_{\ell=1}^r (M^B_{x_i})_{\ell j}\,[b_\ell]\quad\text{in }A,
\]
hence we define the corresponding \emph{staircase rewriting relation}
\[
g_{m}\ :=\ x_i b_j-\sum_{\ell=1}^r (M^B_{x_i})_{\ell j}\,b_\ell \ \in R.
\]
Let
\[
J_B\ :=\ \langle\, g_m \ :\ m\in \partial B \,\rangle \ \subset R.
\]
Then $R/J_B$ has $k$-basis $\overline{B}$ and multiplication by $x_i$ represented by $M^B_{x_i}$.
In particular, there is an induced isomorphism of $\K$-algebras
\[
R/J_B\ \cong\ A\ \cong\ A_\Lambda,
\]
so $J_B=\ker(\rho)=I_\Lambda$, and therefore
\[
Z\ =\ V(I_\Lambda)\ =\ V(J_B)\subset \mathbb{A}^n.
\]

For the projective scheme in $\mathbb{P}^n$, we simply homogenize the ideal $J_B$ and saturate w.r.t.\ $(x_0,\dots,x_n)$. See \cref{Ex: embedded dimension} for an application.

\begin{remark}\label{rem:QI-commuting}
After a successful completion, the multiplication matrices $\{M_{x_i}\}$ provide a concrete representation of a finite-dimensional commutative algebra.
From a quantum-information perspective, one may think of them as a family of mutually commuting operators whose joint (generalized) eigenspaces encode the support scheme; in the reduced case this corresponds to simultaneous diagonalization, while nonreduced structure is reflected by Jordan blocks.
\end{remark}
\end{remark}

\section{On the choice of bases}\label{Section 6}

The tensor formulation above is basis-free. For computations one must choose a basis. In particular, we stress that in \Cref{thm: Summary}, the variables $X_{n+1},\ldots,X_{r-1}$ are not purely symbolic variables, since the coefficients $\lambda_m$ for $\deg m'\leq d$ have been fixed with respect to some basis chosen \emph{a priori}. Therefore, first we make a guess on the basis of the Artinian algebra, and then we write the multiplication tensor as a homogeneous extension of $F$ using more variables, which we identify with elements of the basis. This  choice should be done
\emph{without losing solutions}.

\begin{definition}\label{Def: CompleteStaircase}
    Let $B\subseteq R$ be a set of monomials. We say $B$ is connected to 1 if $\forall m\in B$ either $m=1$ or there is $1\leq i\leq n$ and $m'\in B$ such that $m=x_im'$. We say that $B$ is a staircase if for all $i=1,\ldots,n$ $x_ix^{\beta}\in B$ implies $x^{\beta}\in B$. If in addition $B$ contains all the degree one monomials, we say that $B$ is a complete staircase.
\end{definition}

Let $F\in S_d^*$ be a concise polynomial and $Z=\operatorname{Spec}(A)\subseteq \mathbb P^n$ a finite apolar scheme to $F$, so $A\cong R/I$ with $R=\K[x_1,\dots,x_n]$ for some 0-dimensional ideal $I$. Then for any term order $\prec$, the set of standard monomials
\[
B_\prec(I):=\{\,\text{monomials not in }\mathrm{LT}_\prec(I)\,\}
\]
is a staircase basis of $A$, and as a consequence of conciseness of $F$ we can take $B$ a complete staircase (see {\cite[Proposition 3.2]{Alessandra}}).

\medskip

Therefore restricting computations to staircase bases is a gauge choice and does not
discard solutions. Moreover, staircase bases typically yield sparse multiplication matrices and minimize
the number of rewriting relations, hence are also favorable computationally.
Moreover, if $A$ is local, by \cite[Lemma 2.6]{bernardi2025refinementlocalcactusrank} we can choose a basis using the possible Hilbert function of the algebra.

\medskip

In the following example we highlight the importance of choosing an adequate basis. 

\begin{example}
    Let $F=X_0^{(5)}X_1^{(2)} + X_0^{(5)}X_1X_2\in \K_{\mathrm{dp}}[X_0,X_1,X_2]_7$. One checks that its centroid has dimension 1, so its cactus rank $r$ is not $3$ (see \Cref{alg:minimalcactus}). We set $r=4$ and look for an apolar scheme to $F$ of length $r$, which, by \cite[Corollary 2.2.1]{Buczy_ski_2013}, must be unique. We set $B=\{1,x_1,x_2,x_2^2\}$ and write a variable-extension completion of $F$ as 
    \[G=X_0^{(5)}X_1^{(2)} + X_0^{(5)}X_1X_2 + \lambda_{x_2^2}X_0^{(6)}X_3 + \lambda_{x_1x_1^2}X_0^{(5)}X_1X_3 + \ldots + \lambda_{x_2^{14}}X_3^{(7)}.\]
    We know that $F$ determines the coefficients $\lambda_m$ corresponding to monomials in $\K[x_1,x_2]_{\leq 7}$, such as $\lambda_{1^6 x_2^2}=F(x_0^5x_2^2)=0$, so we obtain
    \[G=X_0^{(5)}X_1^{(2)} + X_0^{(5)}X_1X_2 + \lambda_{x_1^6x_2^2}X_1^{(6)}X_3 + \lambda_{x_1^5x_2^3}X_1^{(5)}X_2X_3 + \ldots + \lambda_{x_2^{14}}X_3^{(7)}\]
    However, $x_0^5\aprod G=X_1^{(2)} + X_1X_2$, which is not concise in 4 variables, thus by \Cref{thm: Summary} it is not a multiplication tensor. We obtain a different result if we set $B=\{1,x_1,x_2,x_1^2\}$. Indeed, the unique apolar scheme to $F$ of length $4$ is defined by 
    \begin{equation}\label{eq: embedded dimension}
        A=\K[x_1,x_2]/\operatorname{Ann}(f)=\K[x_1,x_2]/(x_{2}^{2},\,x_{1}^{2}-x_{1}x_{2})        
    \end{equation}

    for $f=X_1^{(2)} + X_1X_2$, which admits $\{1,x_1,x_2,x_1^2\}$ as a basis but not $\{1,x_1,x_2,x_2^2\}$. In the first basis, the structure tensor is $\mu_{A,s}^{(6)}=X_0^{(5)}X_1^{(2)} + X_0^{(5)}X_1X_2 + X_0^{(6)}X_3$. 
    Hence, only the first basis allows us to find the scheme and to conclude that the cactus rank of $F$ is $4$. 
\end{example}

\begin{remark}
As in this example, whenever we have $F\in S_d^*$ and we solve a length $r>n+1$ variable-extension completion of it, we find an Artinian Gorenstein algebra $A$ defining a scheme $\operatorname{Spec}(A)\subseteq \mathbb P^n$ apolar to $F$. But we also find a scheme $\operatorname{Spec}(A)\subseteq \mathbb P^{r-1}$ apolar to $\mu_{A,s}^{(d-1)}$. This corresponds to the fact that the \emph{embedded dimension} $e(A)$ of $A$, that is, the minimum integer $t$ such that $A$ is a quotient of $\K[x_1,\ldots,x_t]$, is strictly less than $r$. 
    
\end{remark}

\begin{example}\label{Ex: embedded dimension}
    The scheme $\operatorname{Spec}(A)\subseteq \mathbb P^2$ for $A$ in \eqref{eq: embedded dimension} from the previous example has length 4 and Hilbert function $(1,3,4,\ldots)$.
    Say we start with $\mu_{A,s}^{(6)}=X_0^{(5)}X_1^{(2)} + X_0^{(5)}X_1X_2 + X_0^{(6)}X_3\in \K_{\mathrm{dp}}[X_0,X_1,X_2,X_3]_7$ and look for its cactus rank. With \Cref{alg:minimalcactus} we already find that the cactus rank is $4$. The only candidate basis for the algebra is $\{1,x_1,x_2,x_3\}$. The multiplication operators are 
    \[M_{x_1}=\left(\!\begin{array}{cccc}
       0&0&0&0\\
       1&0&0&0\\
       0&0&0&0\\
       0&1&1&0
       \end{array}\!\right),\:M_{x_2}=\left(\!\begin{array}{cccc}
       0&0&0&0\\
       0&0&0&0\\
       1&0&0&0\\
       0&1&0&0
       \end{array}\!\right),\: M_{x_3}=\left(\!\begin{array}{cccc}
       0&0&0&0\\
       0&0&0&0\\
       0&0&0&0\\
       1&0&0&0
       \end{array}\!\right).\]
To obtain the presentation of the algebra, we follow \Cref{rmk:scheme-already-there-BT}: we compute the normal forms of the elements in the border basis:
\[x_1^2\equiv x_3 \quad x_1x_2\equiv x_3\]
and the rest of the elements are congruent to 0. Thus, 
\[A\cong \K[x,y,z]/I \quad \quad I=(x_1^2-x_3,x_1x_2-x_3,x_1x_3, x_2^2,x_2x_3, x_3^2).\]
Clearly this presentation is isomorphic to \eqref{eq: embedded dimension}.
Homogenizing $I$ with respect to $x_0$ yields a scheme in $\mathbb P^3$ with Hilbert function $(1,4,\ldots)$.

\end{example}

\subsection{Bases allowing coordinate change of the tensor}

\newcommand{\monomialIdealPicture}[1]{
    \begin{tikzpicture}[x=(220:1cm), y=(-40:1cm), z=(90:0.707cm),scale=0.5]
        \foreach \m [count=\y] in {#1}{
          \foreach \n [count=\x] in \m {
          \ifnum \n>0
              \foreach \z in {1,...,\n}{
                \draw [fill=orange!30] (\x+1,\y,\z) -- (\x+1,\y+1,\z) -- (\x+1, \y+1, \z-1) -- (\x+1, \y, \z-1) -- cycle;
                \draw [fill=orange!40] (\x,\y+1,\z) -- (\x+1,\y+1,\z) -- (\x+1, \y+1, \z-1) -- (\x, \y+1, \z-1) -- cycle;
                \draw [fill=orange!10] (\x,\y,\z)   -- (\x+1,\y,\z)   -- (\x+1, \y+1, \z)   -- (\x, \y+1, \z) -- cycle;  
              }
             \fi
          }
        }
    \end{tikzpicture}
}

    Recall that we work over a characteristic zero field $\K$.
    Let $I\subseteq R = \K[x_1, \ldots ,x_n]$ be a zero-dimensional monomial ideal. Recall that $I$ is \emph{Borel-fixed}
    if for every $i \geq j$ and every monomial $x_im'\in I$ we have $x_jm'\in I$.

    \begin{example}\label{ex:borel-fixed}
        Take $n = 3$ and consider monomial ideals with $\dim_{\K} R/I = 4$. The Borel-fixed ones are 
        $$I_1=(x_1, x_2, x_3^4),  \quad I_2=(x_1, x_2^2, x_2x_3, x_3^3), \quad I_3=(u_1^2, u_1u_2, u_2^2, u_1u_3, u_2u_3, u_3^2).$$
    \end{example}
    Borel-fixed ideals are widely used to provide bases of ideals, not necessarily in the zero-dimensional setting, see
    \cite{Alberelli20191, Bertone2013263}.

    The following notion, in contrast, is specific for the zero-dimensional case and seems new.
    \begin{definition}[Squat ideals]\label{ref:squat:def}
        A zero-dimensional Borel-fixed ideal $I$ is a \emph{squat ideal} if for every $i$ and every monomial $x_im'\in I$, we have
        $(x_1, \ldots ,x_n)^2m'\subset I$. %\joa{Feel free to suggest alternative names.}
    \end{definition}
    The name comes from the shape of the staircases, for example for $n=3$, and $22$ boxes, we have the following possibilities
    \begin{center}
        \monomialIdealPicture{{5, 3, 1}, {4, 2}, {3, 1}, {2}, {1}}
        \monomialIdealPicture{{5, 3, 2}, {4, 2}, {3, 1}, {2}}
        \monomialIdealPicture{{6, 4, 2}, {4, 2}, {2, 1}, {1}}
        \monomialIdealPicture{{5, 4, 2}, {4, 2}, {3, 1}, {1}}
        \monomialIdealPicture{{5, 3, 2}, {4, 2, 1}, {3, 1}, {1}}
        \monomialIdealPicture{{5, 4, 2}, {4, 2, 1}, {2, 1}, {1}}
        \monomialIdealPicture{{5, 3, 2, 1}, {4, 2, 1}, {2, 1}, {1}}
    \end{center}

    \begin{example}\label{ex:only-one-squat-length4} In the previous \Cref{ex:borel-fixed} only $I_3$ is squat.
Indeed, since $x_1=x_1\cdot 1\in I_1$, if $I_1$ were squat we would have
$(x_1,x_2,x_3)^2\cdot 1\subset I_1$, hence $x_3^2\in I_1$, a contradiction. The same argument applies to $I_2$:
as $x_1=x_1\cdot 1\in I_2$, squatness would force $x_3^2\in I_2$, but $x_3^2\notin I_2$.

Finally, $I_3=(x_1,x_2,x_3)^2$ is squat because whenever $x_i m'\in I_3$ the monomial $m'$ has degree $\ge 1$.
\end{example}

        Squat staircases are named so because of their plump shape. Let us make this precise.
        %     \joa{
        % \begin{lemma}\label{ref:smallRegularity:lemma}
        %     Let $I \subset S = \K[x_0, \ldots, x_n]$ be a squat ideal and $k\geq 0$ be such that $I\not\subset (x_0,\ldots,x_n)^{k+1}$. Then $(x_0,\ldots, x_n)^{2k} \subseteq I$. In particular, if $k\geq 0$ is such that the length of $S/I$ is smaller than $\binom{k+n}{n}$, then all elements
        %     in the staircase of $I$ have degree less than $2k$.
        % \end{lemma}
        % }
        % \ale{We defined the squat ideals in the affine setting (cf. \Cref{ref:squat:def}): $x_im' \in I \Rightarrow (x_1, \ldots , x_n)^2m' \in I$, while here the Lemma is written for $S=\K[x_0, \ldots, x_n]$. I think we should rewrite it for the affine setting. I.e.:
        
        \begin{lemma}\label{ref:smallRegularity:lemma}
        Let $I\subset R=\mathbb K[x_1,\ldots,x_n]$ be a squat monomial ideal, and set $\mathfrak m=(x_1,\ldots,x_n)$. If $k\ge 0$ is such that $I\not\subset \mathfrak m^{k+1}$, then $\mathfrak m^{2k}\subseteq I$. In particular, if
        $$\ell(R/I)<\binom{n+k}{n},$$
        then every monomial in the staircase of $I$ has degree $<2k$.
        \end{lemma}

        \begin{proof}
            Since $I$ is not contained in $(x_1,\ldots,x_n)^{k+1}$, there is a monomial $m\in I$ of degree at most $k$.
            Repetitively applying the condition from \Cref{ref:squat:def} to $m$, we obtain that $I$ contains $(x_1,\ldots, x_n)^{2\deg(m)}$.
            This concludes the first part of the proof.
            The length of $R/(x_1,\ldots, x_n)^{k+1}$ is equal to $\binom{n+k}{k}$, so that if the length of $I$ is smaller, then $I$ cannot be contained
            in $(x_1,\ldots,x_n)^{k+1}$, hence by the part above $I$ contains $(x_1,\ldots, x_n)^{2k}$, so every monomial not in $I$ has degree less than $2k$.
            This proves the second part.
        \end{proof}
        \Cref{ref:smallRegularity:lemma} sharply distinguishes squat ideals among all monomial and Borel-fixed ones. Indeed, monomial or
        Borel-fixed ideals of length $d$ can have an element of degree $d-1$ in their staircase, which is a significant drawback for the
        algorithms presented in~\cite{Alessandra}.

    There are much fewer squat staircases than Borel staircases and much fewer Borel staircases than all staircases.
    We provide their numbers for some small lengths in three variables for illustration.

{\footnotesize{
    
    \[\hspace*{-2cm}
        \begin{tabular}{c c c c c c c c c c c c c c c c c c}
            \hfill $\K[x_1, x_2,x_3]$ \hfill & $1$ & $2$ & $3$ & $4$ & $5$ & $6$ & $7$ & $8$ & $9$ & $10$ &
            $11$ & $12$ & $13$ & $14$ & $15$ & $16$
            \\
            \#staircases & $1$ & $3$ & $6$ &
            $13$ & $24$ & $48$ & $86$ & $160$ & $282$ & $500$ & $859$ & $1479$& $2485$& $4167$& $6879$& $11297$\\
            \#Borel staircases & $1$ & $1$ & $2$ & $3$ & $4$ & $6$ &
            $9$ & $12$ & $17$ & $24$& $32$& $44$& $60$& $80$& $107$& $143$\\
            \#Squat staircases & $1$ & $1$ & $1$ & $1$ & $1$ & $1$ & $2$ & $2$ & $2$ & $2$ & $2$ & $2$ & $3$ & $3$ & $3$ & $4$
        \end{tabular}
    \]
    }}
    
    Both the notions of Borel-fixed and squat ideals are motivated by group actions. For Borel-fixed, we consider the
    natural action of $GL_n$ on $\K[x_1, \ldots,x_n]$. We have a Borel subgroup
    \[
        B_n := \begin{pmatrix}
            * & * &  \ldots & * & *\\
            0 & * & \ldots  & * & *\\
              && \ldots\\
            0 & 0 &  \ldots & * & *\\
            0 & 0 &  \ldots & 0 & *
        \end{pmatrix}
    \]
    and a zero-dimensional ideal $I\subset \K[x_1, \ldots ,x_n]$ satisfies $B_n\cdot I \subset I$ if and only if it is
    Borel-fixed. Borel-fixed ideals are frequently called strongly-stable, as the latter notion works better in positive characteristics.

    \newcommand{\Rbar}{\overline{R}}
    \newcommand{\Aut}{\operatorname{Aut}}
    \newcommand{\Der}{\operatorname{Der}}
    \newcommand{\Autbar}{\overline{\Aut}}

    To procure the group action corresponding to squat ideals, we need to restrict our scope further. Fix $r\geq 2$ and consider ideals
    $I\subset \K[x_1, \ldots ,x_n]$ such that the length of $A = \K[x_1, \ldots ,x_n]/I$ is $r$ and $A$ is supported only at zero.
    In this case, we have $I\supset (x_1, \ldots ,x_n)^r$, so that we can consider $\bar{I} \subset \K[x_1, \ldots ,x_n]/(x_1, \ldots ,x_n)^r$.
    Let
    \[
        \Rbar := \frac{\K[x_1, \ldots ,x_n]}{(x_1, \ldots ,x_n)^r}.
    \]
    The automorphism group $\Aut(\Rbar)$ of $\Rbar$ contains $GL_n$, but it is much larger. Namely,
    for every elements $r_1, \ldots ,r_n\in \Rbar$ such that $r_1, \ldots ,r_n$ are $\K$-linearly independent modulo $(x_1, \ldots ,x_n)^2$,
    we have a unique automorphism $\phi\colon \Rbar\to \Rbar$, defined by $\phi(x_i) = r_i$ for $i=1, \ldots , n$.
    Let $\Autbar(\Rbar)\subset \Aut(\Rbar)$ be the subgroup defined by
    \[
        \Autbar(\Rbar) = \left\{ \varphi\in \Aut(\Rbar)\ |\ \varphi(x_i) \equiv x_i \mod (x_1, \ldots ,x_n)^2 \right\}.
    \]
    This is a unipotent group.
    We have a map $\pi\colon \Aut(\Rbar)\to GL_n$ given by $\varphi\mapsto \varphi \mod (x_1, \ldots ,x_n)^2$ and $\Autbar(\Rbar)$ is
    exactly the kernel of this map. Consider the subgroup $\pi^{-1}(B_n)$, where $B_n$ is the Borel subgroup above.
    Since $B_n$ is solvable and $\Autbar(\Rbar)$ is unipotent, also $\pi^{-1}(B_n)$ is solvable.
    The Lie algebra of $\Autbar(\Rbar)$ is the algebra of derivations $\Der(\Rbar, \Rbar)$. The Lie algebra of $\pi^{-1}(B_n)$
    inside is spanned by
    \[
        \{x_j \partial_{x_i}\ | i\geq j\} \cup \left\{ x_jx_k\partial_{x_i}\ |\ i,j,k \mbox{ any}  \right\}.
    \]
    This shows that a monomial ideal $I\subseteq \K[x_1, \ldots ,x_n]$ is a squat ideal iff it is fixed under the
    above derivations iff it is fixed under the group $\pi^{-1}(B_n)$.

\begin{example}\label{ex:autbar-squat}
Let $n=2$ and $r=5$. Set $R=\K[x_1,x_2]$ and $\m=(x_1,x_2)$, and consider the truncated local ring $\Rbar=\frac{R}{\m^5}$.
The $\K$-derivation $D:=x_2^2\partial_{x_1}$ preserves $\m^5$, hence it induces a derivation of $\Rbar$. Exponentiating $D$ gives a one parameter subgroup of $\Aut(\Rbar)$:
\[
\varphi_t=\exp(tD)\in \Aut(\Rbar),\qquad
\varphi_t(x_1)=x_1+t x_2^2,\ \ \varphi_t(x_2)=x_2.
\]
In particular, $\varphi_t\in \Autbar(\Rbar)$ because $\varphi_t(x_i)\equiv x_i\!\!\pmod{\m^2}$, hence $\pi(\varphi_t)=\mathrm{Id}$.

Now consider the following two monomial ideals in $R$ (both define length $5$ algebras supported at $0$):
\[
I_{\mathrm{s}}=(x_1^2,\ x_1x_2^2,\ x_2^3),\qquad
I_{\mathrm{ns}}=(x_1^2,\ x_1x_2,\ x_2^4).
\]
We claim that $I_{\mathrm{s}}$ is squat, while $I_{\mathrm{ns}}$ is not, and that this is detected by $D$.

\smallskip
\noindent\emph{(1) $I_{\mathrm{ns}}$ is not squat.}
Indeed, $x_1x_2=x_1\cdot x_2\in I_{\mathrm{ns}}$ with $m'=x_2$, but
\[
(x_1,x_2)^2\cdot x_2 \ni x_2^3 \notin I_{\mathrm{ns}}
\]
(since $I_{\mathrm{ns}}$ contains $x_2^4$ but not $x_2^3$).

\smallskip
\noindent\emph{(2) $I_{\mathrm{s}}$ is squat.}
One checks it on generators: for $x_1^2=x_1\cdot x_1$ we have $(x_1,x_2)^2x_1\subset (x_1^2,x_1x_2^2)$;
for $x_1x_2^2=x_1\cdot x_2^2$ and $x_2^3=x_2\cdot x_2^2$ we have $(x_1,x_2)^2x_2^2\subset (x_1^2,x_1x_2^2,x_2^3)$.

\smallskip
\noindent\emph{(3) Stability under the derivation.}
We have
\[
D(x_1^2)=2x_1x_2^2\in I_{\mathrm{s}},\qquad
D(x_1x_2^2)=x_2^4\in I_{\mathrm{s}}\ \ (\text{since }x_2^3\in I_{\mathrm{s}}),\qquad
D(x_2^3)=0,
\]
so $D(I_{\mathrm{s}})\subset I_{\mathrm{s}}$, i.e.\ $I_{\mathrm{s}}$ is fixed by $D$.

On the other hand,
\[
D(x_1x_2)=x_2^3\notin I_{\mathrm{ns}},
\]
so $D(I_{\mathrm{ns}})\nsubseteq I_{\mathrm{ns}}$ and $I_{\mathrm{ns}}$ is \emph{not} fixed by $D$.

Equivalently, the unipotent automorphism $\varphi_1$ sends
\[
\varphi_1(x_1x_2)=(x_1+x_2^2)x_2=x_1x_2+x_2^3,
\]
so $\varphi_1(I_{\mathrm{ns}})$ is not contained in $I_{\mathrm{ns}}$ (it forces a $x_2^3$ term).
This illustrates concretely how the ``quadratic'' part of $\pi^{-1}(B_n)$ detects the squat condition.
\end{example}

    \begin{theorem}[ubiquity of squat ideals]
        Let $A$ be a zero-dimensional algebra. Let $\mathcal{V}_A$ be the locus of ideals
        $I\subset R = \K[x_1, \ldots ,x_n]$ such that $R/I$ is isomorphic to $A$. Then the closure of $\mathcal{V}_A$
        in the Hilbert scheme contains a squat ideal.
    \end{theorem}
    \begin{proof}
        Pick any point $I\subseteq R$ in $\mathcal{V}_A$. Degenerating $I$ towards zero, we obtain that
        the closure of $\mathcal{V}_A$ contains an ideal $I'\subset R$ such that $R/I'$ is supported only at zero.
        Let $A' = R/I'$. The closure of $\mathcal{V}_A$ contains the whole $\mathcal{V}_{A'}$, so it is enough to
        find a squat ideal in the closure of $\mathcal{V}_{A'}$.

        Recall the ring $\Rbar$ defined above.
        Consider the locus $\mathcal{W}_{A'}$ that consists of ideals $\bar{I} \subset \Rbar$ such that
        $\Rbar / \bar{I}$ is isomorphic to $A'$. Its closure $\overline{\mathcal{W}_{A'}}$ is a projective subvariety
        of the Hilbert scheme of $\Rbar$ and the solvable group $\pi^{-1}(B_n)$ acts on this variety, so
        by Borel's fixed point theorem~\cite[Theorem III.10.4]{Borel}, 
        the variety $\overline{\mathcal{W}_{A'}}$
        contains a $\pi^{-1}(B_n)$-fixed point. Such a point corresponds to the desired squat ideal.
    \end{proof}

    Now we transform the above result into reducing the number of staircase bases.
    For a squat ideal $I\subseteq R$,  its \emph{squat staircase} $E_I\subseteq \mathbb{N}^n$ is the set of all monomials not in $I$.
    By abuse of notation, we will view $E_I$ both as a subset of $\mathbb{N}^n$ and a set of monomials in $R$.

    \begin{corollary}[Staircase after nonlinear change of basis]\label{ref:nonlinearBaseChange:cor}
        Let $I\subseteq R=\K[x_1, \ldots,x_n]$ be a zero-dimensional ideal such that $R/I$ has length $r$.
        Pick general $n$ elements of $R/I$ and their preimages $v_1, \ldots ,v_n\in R$.
        Then there exists a squat ideal $I_0$ such that the set
        \[
            \left\{ v_1^{e_1} \ldots v_n^{e_n}\ |\ (e_1, \ldots ,e_n)\in E_{I_0} \right\}
        \]
        is a basis of $R/I$.
    \end{corollary}
    \begin{proof}
        Assume first that $R/I$ is supported only at zero, so that $R/I\simeq \Rbar/\bar{I}$. Then $\Autbar(\Rbar)\cdot [\bar{I}]$
        contains a squat ideal $I_0$. The monomials in $E_{I_0}$ form a basis of $R/I_0$, hence also a basis of
        $R/\varphi(I)$, for a general $\varphi\in \Autbar(\Rbar)$. This means that $\{\varphi^{-1}(m)\ |\ m\in E_{I_0}\}$
        is the required basis for $R/I$. This proves the claim in this special case.

        Assume now that $R/I$ is general. As discussed above, the closure of $GL_n\cdot [I]$ contains an ideal $I'$ supported
        only at zero. The claim of the theorem is true for $R/I'$ by the argument above. But then, the claim is true for
        a general element of $GL_n\cdot [I]$. The claim is invariant under $GL_n$, so it is true for every element of $GL_n\cdot [I]$,
        in particular for $I$.
    \end{proof}

%\ale{I think that also \Cref{ref:nonlinearBaseChange:cor} is correct since it regards only the algebra. What it has to be clarified is how we use it in practice. So I would suggest to leave the Corollary above as it is and add the following.}

\begin{remark}\label{rem:squat-marking}
\Cref{ref:nonlinearBaseChange:cor} concerns only the abstract algebra
$A=R/I$. In the completion problem, one also has to keep track of the
marking
\[
\iota:V^*\hookrightarrow A
\]
introduced in \Cref{def:MT-CMG}. 
Explicitly, let 
\[
u_i:=\iota(x_i)\in A,\qquad i=1,\ldots,n.
\]
These need not coincide with the new generators, and in general we have, in the same notation as \Cref{ref:nonlinearBaseChange:cor}, a unique expression for each $u_i$
\[
u_i=c_i+L_i(v)+D_i(v),
\]
where $c_i\in\K$, $L_i(v)$ is linear in $v_1,\ldots,v_n$, and $D_i(v)$ is a
$\K$-linear combination of basis monomials in $\mathcal B_E$ of total degree at least $2$. Moreover, by definition of marked completion, for every multi-index $\alpha$ with
$|\alpha|\le d$ one has the matching equations
\[
f_\alpha=\Lambda\bigl(\iota(x)^\alpha\bigr)=\Lambda(u^\alpha).
\]

%After passing to a squat basis, the images of the original variables under
%$\iota$ need not coincide with the new generators. This is the reason why the
%matching equations must be written in terms of the transported marked
%generators, and not in terms of the squat generators themselves.
\end{remark}

\begin{example}
\label{ex:squat-basis-transported-marking}
Let $f=Y+X^{(2)}+Y^{(2)}$ and let $I:=\mathrm{Ann}(f)=(xy,\ y^2-x^2,\ x^3)\subset \K[x,y]$. Set $A:=\K[x,y]/I$. Then $A$ is a local Artinian Gorenstein algebra of length $4$. Indeed, one checks
that $\{1,\bar x,\bar y,\bar x^2\}$ is a basis of $A$, and the socle is generated by $\bar x^2$.

Now define two new elements of $A$ by
\[
v:=\bar x,\qquad w:=\bar y-\bar x^2.
\]
Using the relations in $A$, we compute
\[
vw=\bar x(\bar y-\bar x^2)=\bar x\bar y-\bar x^3=0,
\]
\[
w^2=(\bar y-\bar x^2)^2=\bar y^2-2\bar y\bar x^2+\bar x^4=\bar x^2=v^2,
\]
and
\[
w^3=w\,w^2=w\,\bar x^2=(\bar y-\bar x^2)\bar x^2=0.
\]
Therefore
\[
\{1,v,w,w^2\}
\]
is again a basis of $A$.

This basis is indexed by the squat staircase
\[
E=\{1,v,w,w^2\},
\]
namely by the staircase of the squat monomial ideal
\[
I_0=(v^2,vw,w^3)\subset \K[v,w].
\]
Notice, however, that the actual presentation ideal of $A$ in the generators
$v,w$ is
\[
J=(vw,\ v^2-w^2,\ w^3),
\]
so the point is not that $A$ is the quotient by the squat ideal $I_0$, but that
$A$ admits a basis indexed by the squat staircase $E_{I_0}$.

We now describe the marking in this squat basis. The original variables $x,y$
are represented in $A$ by
\[
u_x:=\bar x,\qquad u_y:=\bar y.
\]
Since
\[
\bar x=v,\qquad \bar y=w+v^2,
\]
the transported marked generators are
\[
u_x=v,\qquad u_y=w+v^2.
\]
Thus the second marked generator acquires a quadratic correction in the squat
basis.

Let $\Lambda_f:A\to\K$ be the functional associated with $f$, namely $\Lambda_f(a):=(a \lrcorner f)(0)$.
Then
\[
\Lambda_f(u_x)=0,\qquad \Lambda_f(u_y)=1,
\]
\[
\Lambda_f(u_x^2)=1,\qquad \Lambda_f(u_xu_y)=0,\qquad \Lambda_f(u_y^2)=1.
\]
Hence the coefficients of $f$ are recovered from the moments of the transported
marked generators $u_x,u_y$.

On the other hand, the moments of the squat generators themselves are
\[
\Lambda_f(v)=0,\qquad \Lambda_f(w)=0,
\]
\[
\Lambda_f(v^2)=1,\qquad \Lambda_f(vw)=0,\qquad \Lambda_f(w^2)=1.
\]
These are the coefficients of $X^{(2)}+Y^{(2)}$, not of $f=Y+X^{(2)}+Y^{(2)}$.

Therefore this example shows that, after passing to a squat basis, the correct
matching equations are
\[
f_\alpha=\Lambda_f(u^\alpha),
\]
where $u_i$ are the transported marked generators, and not $f_\alpha=\Lambda_f(v^\alpha)$.
\end{example}

\subsection{Using nonlinear changes of coordinates}

Let $R=\K[x_1,\ldots,x_n]$ and $A=R/\operatorname{Ann(\Lambda)}$ an Artinian Gorenstein algebra. Choose general elements $\bar v_1,\ldots,\bar v_n\in A$, and $v_1,\ldots,v_n\in R$. By \Cref{ref:nonlinearBaseChange:cor}, there exists
a squat ideal $I_0$ with staircase $E=E_{I_0}$ such that $\{\bar v^e:e\in E\}$
is a $k$-basis of $A$. Let
\[
\phi:\K[t_1,\ldots,t_n]\to A,\qquad t_i\mapsto \bar v_i.
\]
Then $\phi$ is surjective and $\ker\phi=\operatorname{Ann}(\phi^*\Lambda)$. Indeed, $p\in \operatorname{Ann}(\phi^*\Lambda)$ if and only if $p\aprod (\phi^* \Lambda)(q)=\Lambda(\phi(p)\phi(q))=0$ for all $q\in \K[t_1,\ldots,t_n]$. Since $\phi$ is surjective and $\Lambda$ defines a nondegenerate pairing in $A$, $\phi(p)=0$.

% Let $R=\K[x_1,\ldots,x_n]$ and $A=R/\operatorname{Ann(\Lambda)}$ an Artinian Gorenstein algebra. Take elements $p_1,\ldots,p_n\in R$ of order at least 2, that is, without linear of constant terms, set  $v_i=x_i + p_i(x)$ and consider the map 
% \[\phi: \K[t_1,\ldots,t_n]\to A \quad t_i\mapsto \overline{v_i}.\] 
% By \Cref{ref:nonlinearBaseChange:cor}, it is a surjective algebra homomorphism. Moreover, $\ker \phi=\operatorname{Ann(\phi^*\Lambda)}$. Indeed, $p\in \operatorname{Ann}(\phi^*\Lambda)$ if and only if $p\aprod (\phi^* \Lambda)(q)=\Lambda(\phi(p)\phi(q))=0$ for all $q\in \K[t_1,\ldots,t_n]$. Since $\phi$ is surjective and $\Lambda$ defines a nondegenerate pairing in $A$, $\phi(p)=0$.}

% \ale{I think that the above green part (I made it green only to highlight it) as it is is not correct because $\phi$ is surjective if we have a squat basis... in general it may not be. I suggest to riwrite it as follows:}

\medskip

% \ale{Choose general elements $\bar v_1,\ldots,\bar v_n\in A$, and $v_1,\ldots,v_n\in R$. By \Cref{ref:nonlinearBaseChange:cor}, there exists
% a squat ideal $I_0$ with staircase $E=E_{I_0}$ such that $\{\bar v^e:e\in E\}$
% is a $k$-basis of $A$. Let
% \[
% \phi:\K[t_1,\ldots,t_n]\to A,\qquad t_i\mapsto \bar v_i.
% \]
% Then $\phi$ is surjective and $\ker\phi=\operatorname{Ann}(\phi^*\Lambda)$.}

\begin{example}
    Let $F=X^{(2)}YZ^{(5)}\in (\Sym^8V^*)^*$ 
    %for $V$ a vector space over $k=\mathbb Z/101\mathbb Z$. 
    It is known that its cactus rank is 6, and the scheme is $\operatorname{Spec(A)}$ for $A=\K[x,y]/\operatorname{Ann}(X^{(2)}Y)=\K[x,y]/(x^3, y^2)$ is apolar to $F$ and has length 6. Since it has regularity $3\leq d/2-1$, it is the unique apolar scheme. We note that $\{1,x,y,x^2,xy,y^2\}$, which is a squat staircase, is not a basis of $A$. Instead, $B=\{1,x,y,x^2,xy,x^2y\}$ is a basis. 

    \medskip

    % \oriol{Now take $v_1=x+xy$, $v_2=y+x^2$; the map $\phi^*$ sends $f=X^2Y$ to $g=T^2W + 2T^2 + 2W^2$, and in this case $\phi$ induces an automorphism on $A$. Note that the squat staircase $\{1,t,w, t^2, tw, w^2\}$ is now a basis of $\K[t,w]/\operatorname{Ann}(g)$.     }
    
%\medskip

%     \ale{(Can we rewrite the above as:)}
    
% \medskip

    Now take $\bar v_1=\bar x+\bar x\bar y$ and
$\bar v_2=\bar y+\bar x^2$, and let
\[
\phi:\K[t,w]\to A,\qquad t\mapsto \bar v_1,\quad w\mapsto \bar v_2.
\]
The pullback  $\phi^*\Lambda$ is represented by $g=T^{(2)}W+2T^{(2)}+2W^{(2)}$.
Moreover, the images of the monomials $\{1,t,w,t^2,tw,w^2\}$ form a basis of $A$. Hence $\phi$ induces an isomorphism $\K[t,w]/\operatorname{Ann}(g)\simeq A$.

\medskip

% \ale{
% (This should not be interpreted as an automorphism of the original $A=\K[x,y]/(x^3,y^2)$: in order to have it we should have something that sends $x \mapsto x+xy$ and $y \mapsto y+y^2$ which is not well defined because it does not preserve $y^2$: in $A$ $(y+x^2)^2=2x^2y\neq 0$ so the image of $y^2$ goes in a non-zero element. What is true is that $\phi:\K[t,w] \to A$, $t\mapsto x+xy$, $w \mapsto y+x^2$ induces an isomorphism $\K[t,w] \simeq Ann(\phi^* \Lambda) \simeq A$.
%      }

\end{example}

We can use a nonlinear change of coordinates in our algorithm as in the following example:

\begin{example}
    Let $f=X^2Y$ and let us recover the scheme $A=\K[x,y]/ \operatorname{Ann}(f)$ using the \cite{Alessandra} method. After dehomogenizing and building the Hankel matrix of $F$, we have 
     \[H_{\Lambda}^{\{1,x,y,x^2,xy,y^2\}}=
\begin{array}{c|cccccc}
                 &1 & x & y &x^2 & xy & y^2  \\
\hline 
1^*               & 0& 0& 0 &0 &0 &0  \\
x^*               & 0& 0& 0 &0 &1 &0  \\
y^*               & 0& 0& 0 &1 &0 &0  \\
(x^2)^*           & 0& 0& 1 &0 &0 &0  \\
(xy)^*               & 0& 1& 0 &0 &0 &0  \\
(y^2)^*               & 0& 0& 0 &0 &0 &0  \\
\end{array}\]
which has determinant 0, as expected since $\{1,x,y,x^2,xy,y^2\}$ is not a basis of $A$. For the basis $B=\{1,x,y,x^2,xy,x^2y\}$, 
 \[H_{\Lambda}^{\{1,x,y,x^2,xy,x^2y\}}=
\begin{array}{c|cccccc}
                 &1 & x & y &x^2 & xy & x^2y  \\
\hline 
1^*               & 0& 0& 0 &0 &0 &1  \\
x^*               & 0& 0& 0 &0 &1 &0  \\
y^*               & 0& 0& 0 &1 &0 &0  \\
(x^2)^*           & 0& 0& 1 &0 &0 &0  \\
(xy)^*               & 0& 1& 0 &0 &0 &0  \\
(x^2y)^*               & 1& 0& 0 &0 &0 &0  \\
\end{array}\]
has determinant $-1$. In this basis, the multiplication operators are
\[M_x=\left(\!\begin{array}{cccccc}
      0&0&0&0&0&0\\
      1&0&0&0&0&0\\
      0&0&0&0&0&0\\
      0&1&0&0&0&0\\
      0&0&1&0&0&0\\
      0&0&0&0&1&0
      \end{array}\!\right) \quad  M_y=\left(\!\begin{array}{cccccc}
      0&0&0&0&0&0\\
      0&0&0&0&0&0\\
      1&0&0&0&0&0\\
      0&0&0&0&0&0\\
      0&1&0&0&0&0\\
      0&0&0&1&0&0
      \end{array}\!\right)\]
      and they commute. These are the matrices of multiplication by $x$ and $y$, respectively, in the algebra $A$ with respect to this basis.

Let's now revise the same example by using squat basis. Consider the map $\phi$ as above, which maps the squat staircase $B=\{1,t,w,t^2,tw,w^2\}$ to $\phi(B)=\{1,x+xy, y+x^2, (x+xy)^2, (x+xy)(y+x^2), (y+x^2)^2\}$. By taking linear combinations of the columns in the Hankel matrix of $H_{\Lambda}$, we have
\[H_\Lambda^{\phi(B)}=\left(\!\begin{array}{cccccc}
      0&0&0&2&0&2\\
      0&2&0&0&1&0\\
      0&0&2&1&0&0\\
      2&0&1&0&0&0\\
      0&1&0&0&0&0\\
      2&0&0&0&0&0
      \end{array}\!\right)\]
    which is invertible and therefore $B$ is a basis of $A$. 
Similarly, the multiplication matrices for this basis are

{\tiny{
\[
M_x^{\phi(B)}=\left(\!\begin{array}{cccccc}
      0&0&0&\frac{1}{2}h_{x_{1}^{7}x_{2}^{2}}&\frac{1}{2}h_{x_{1}^{8}x_{2}}&\frac{1}{2}h_{x_{1}^{9}}\\
      1&0&0&h_{x_{1}^{6}x_{2}^{3}}&h_{x_{1}^{7}x_{2}^{2}}&h_{x_{1}^{8}x_{2}}\\
      0&0&0&h_{x_{1}^{5}x_{2}^{4}}-h_{x_{1}^{7}x_{2}^{2}}&h_{x_{1}^{6}x_{2}^{3}}-h_{x_{1}^{8}x_{2}}&h_{x_{1}^{7}x_{2}^{2}}-h_{x_{1}^{9}}\\
      0&1&0&-2\,h_{x_{1}^{5}x_{2}^{4}}+2\,h_{x_{1}^{7}x_{2}^{2}}&-2\,h_{x_{1}^{6}x_{2}^{3}}+2\,h_{x_{1}^{8}x_{2}}&-2\,h_{x_{1}^{7}x_{2}^{2}}+2\,h_{x_{1}^{9}}\\
      -1&0&1&-2\,h_{x_{1}^{6}x_{2}^{3}}&-2\,h_{x_{1}^{7}x_{2}^{2}}&-2\,h_{x_{1}^{8}x_{2}}\\
      0&\frac{-1}{2}&0&2\,h_{x_{1}^{5}x_{2}^{4}}-2\,h_{x_{1}^{7}x_{2}^{2}}&\frac{1}{2}(4\,h_{x_{1}^{6}x_{2}^{3}}-4\,h_{x_{1}^{8}x_{2}}+1)&2\,h_{x_{1}^{7}x_{2}^{2}}-2\,h_{x_{1}^{9}}
      \end{array}\!\right)\]\[M_y^{\phi(B)}=\left(\!\begin{array}{cccccc}
      0&0&0&\frac{1}{2}h_{x_{1}^{6}x_{2}^{3}}&\frac{1}{2}h_{x_{1}^{7}x_{2}^{2}}&\frac{1}{2}h_{x_{1}^{8}x_{2}}\\
      0&0&0&h_{x_{1}^{5}x_{2}^{4}}&h_{x_{1}^{6}x_{2}^{3}}&h_{x_{1}^{7}x_{2}^{2}}\\
      1&0&0&h_{x_{1}^{4}x_{2}^{5}}-h_{x_{1}^{6}x_{2}^{3}}&h_{x_{1}^{5}x_{2}^{4}}-h_{x_{1}^{7}x_{2}^{2}}&h_{x_{1}^{6}x_{2}^{3}}-h_{x_{1}^{8}x_{2}}\\
      -1&0&0&-2\,h_{x_{1}^{4}x_{2}^{5}}+2\,h_{x_{1}^{6}x_{2}^{3}}&-2\,h_{x_{1}^{5}x_{2}^{4}}+2\,h_{x_{1}^{7}x_{2}^{2}}&-2\,h_{x_{1}^{6}x_{2}^{3}}+2\,h_{x_{1}^{8}x_{2}}\\
      0&1&0&-2\,h_{x_{1}^{5}x_{2}^{4}}&-2\,h_{x_{1}^{6}x_{2}^{3}}&-2\,h_{x_{1}^{7}x_{2}^{2}}\\
      1&0&\frac{1}{2}&\frac{1}{2}(4\,h_{x_{1}^{4}x_{2}^{5}}-4\,h_{x_{1}^{6}x_{2}^{3}}+1)&2\,h_{x_{1}^{5}x_{2}^{4}}-2\,h_{x_{1}^{7}x_{2}^{2}}&2\,h_{x_{1}^{6}x_{2}^{3}}-2\,h_{x_{1}^{8}x_{2}}
      \end{array}\!\right)
\]
}}
Imposing its commutation yields a zero dimensional ideal in the polynomial ring 
\[k\mathopen{}\left[h_{x_{1}^{4}x_{2}^{5}},\,h_{x_{1}^{5}x_{2}^{4}},\,h_{x_{1}^{9}},\,h_{x_{1}^{8}x_{2}},\,h_{x_{1}^{6}x_{2}^{3}},\,h_{x_{1}^{7}x_{2}^{2}}\right]\]
with a unique solution, which is all generators equal to zero. These values of the parameters recover the algebra $A$. 
\end{example}
\begin{remark}
    Since for any monomial $m\in \K[t_1,\ldots,t_n]$ we can compute as above $\phi^*\Lambda(m)=\Lambda(\phi(m))$, we can also obtain the multiplication operators $M_{t_1}\ldots, M_{t_n}$ for the algebra $\K[t_1,\ldots,t_n]/\operatorname{Ann}(\phi^*\Lambda)$ with respect to a squat basis $B$, where the matrices are parametrized by the original moment variables $h_{x^{\alpha}}$. Imposing their commutation could lead to
    %generally gives
    a bigger space of solutions than the commutation of $M_{x_1}\ldots, M_{x_n}$ in the basis $\phi(B)$, since commutation in the $t$-coordinates does not by itself impose compatibility with the original $x$-coordinates. To restrict to the locus compatible with the original $x$-coordinates, one %imposes
    imposes this compatibility, namely that multiplication by $t_i$ corresponds to multiplication by $\phi(t_i)$:
    \[M_{t_i}^B=\phi(t_i)(M_x^{\phi(B)}) \quad i=1,\ldots, n\]
    where $M_x^{\phi(B)}=(M_{x_1}^{\phi(B)},\ldots,M_{x_n}^{\phi(B)})$, and $\phi(t_i)(M_x^{\phi(B)})$ is obtained by substituting in the polynomial $\phi(t_i)$ each variable $x_j$ by the matrix $M_{x_j}^{\phi(B)}$.

    % \ale{(Here $M$ denotes genuine multiplication matrices; if  transpose matrices are scaled by adjugate, the identities must be normalized accordingly.)}
\end{remark}

\begin{example}For $\phi$ as in the previous examples, we compute the Hankel matrix of $\phi^*(\Lambda)$, now indexed by monomials in $t,w$. For example, we have
\[\phi^*(\Lambda)(t^5)=\Lambda(x^{5}y^{5}+5\,x^{5}y^{4}+10\,x^{5}y^{3}+10\,x^{5}y^{2}+5\,x^{5}y+x^{5})=h_{x^5y^5} + 5h_{x^5y^4}.\]

The principal $6\times 6$ subminor is
\[
H_{\phi^*\Lambda}^{\{1,t,w, t^2, tw, w^2\}}=
\begin{array}{c|cccccc}
                 &1 & t & w &t^2 & tw & w^2  \\
\hline 
1^*               & 0& 0& 0 &2 &0 &2  \\
t^*               & 0& 2& 0 &0 &1 &0  \\
w^*               & 0& 0& 2 &1 &0 &0  \\
(t^2)^*           & 2& 0& 1 &0 &0 &0  \\
(tw)^*               & 0& 1& 0 &0 &0 &0  \\
(w^2)^*               & 2& 0& 0 &0 &0 &0  \\
\end{array}\]
which has determinant $-4$. Hence, the squat staircase $B=\{1,t,w, t^2, tw, w^2\}$ is a basis of $\K[t,w]/\operatorname{Ann}(\phi^*\Lambda)$. The commutation of the multiplication operators $M_t^{B}$ and $M_w^{B}$ defines a 9 dimensional ideal in

 $$k\mathopen{}\left[h_{x_{1}^{9}x_{2}},\,h_{x_{1}^{8}x_{2}^{2}},\,h_{x_{1}^{6}x_{2}^{4}},\,h_{x_{1}^{5}x_{2}^{4}},\,h_{x_{1}^{8}x_{2}},\,h_{x_{1}^{5}x_{2}^{5}},\,h_{x_{1}^{6}x_{2
      }^{3}},\,h_{x_{1}^{7}x_{2}^{2}},\,h_{x_{1}^{10}},\,h_{x_{1}^{7}x_{2}^{3}},\,h_{x_{1}^{4}x_{2}^{5}},\,h_{x_{1}^{9}}\right]$$
       of degree 8. Imposing 
       \[M_t^{B}=M_x^{\phi(B)} + M_x^{\phi(B)} M_y^{\phi(B)} \quad \quad M_w^{B}=M_y^{\phi(B)} + \left(M_x^{\phi(B)}\right)^2,\]
       where $M_x^{\phi(B)}$, $M_y^{\phi(B)}$ have been computed in the previous example, gives the expected unique solution. Substituting the solution in $M_t^{B}, M_w^{B}$ gives
        \[ M_t=\left(\!\begin{array}{cccccc}
      0&0&0&0&0&0\\
      1&0&0&0&0&0\\
      0&0&0&0&0&0\\
      0&1&0&0&0&0\\
      0&0&1&0&0&0\\
      0&0&0&0&\frac{1}{2}&0
      \end{array}\!\right) \quad M_w=\left(\!\begin{array}{cccccc}
      0&0&0&0&0&0\\
      0&0&0&0&0&0\\
      1&0&0&0&0&0\\
      0&0&0&0&0&0\\
      0&1&0&0&0&0\\
      0&0&1&\frac{1}{2}&0&0
      \end{array}\!\right),\]
      % \ale{Are you sure that these are the multiplication matrices? I expect that in the basis $B$, the multiplication by $t$ in the first column should have the coordinates of $1\cdot t=t$, i.e. $(0,1,0,0,0,0)^t$. Probably these are simply the transposed, indeed the first row is almost correct... I think that these that you wrote are actually $-4M_t^T$ and the analogous for $M_w$, where I gusess that $-4$ is $\det H_{\phi^*\Lambda}^{B}$.} 
      
      which are the multiplication matrices of the algebra $\K[t,w]/\operatorname{Ann}(\phi^*f)$ in the same basis $\{1,t,w,t^2,tw,w^2\}$. 
\end{example}

 \begin{remark}
     In the language of multiplication tensors, we are simply expressing the tensor in a different basis. Note that if the original tensor is written in the dual of a monomial basis $\{1,x_1,\ldots,x_n,b_{n+1},\ldots,b_{r-1}\}$, after the nonlinear change of coordinates $\phi$ the variables of the tensor do not correspond to the variables of the polynomial ring.   
 \end{remark}

\def\OldSectionEight{
\section{From degree-extension to variable-extension} \label{Section8}
% Hankel-based methods (e.g. \cite{Bernard} for Waring and \cite{Alessandra} for cactus rank) start from the truncated
% apolar data determined by a form $F\in \K_{\mathrm{dp}}[X_0,\dots,X_n]_d$ and search for a degree extension
% of the underlying functional (equivalently: unknown higher moments) in the same polynomial ring
% $R=\K[x_1,\dots,x_n]$, so as to recover commuting multiplication operators and then the apolar scheme.

% In this section we isolate a more intrinsic viewpoint: instead of extending by degree inside $R$, we
% seek a completion of $F$ to a multiplication tensor in a larger ambient space, together with a
% marked embedding of the original variables. This provides a basis-free formulation in which the
% \cite{Bernard, Alessandra} pipeline becomes a special case, while also allowing implementations that can be interpreted
% as \emph{variable-extension} rather than moment extension.

Under the staircase gauge of \Cref{Section 6}, imposing that the marked generators coincide
with the degree-$1$ monomials (i.e.\ $\iota(x_i)=[x_i]\in R/I$) and that the remaining basis elements are
monomials produces exactly the degree/moment-extension setting of \cite{Alessandra}:
the missing coefficients needed to build multiplication matrices are higher moments, and the commutation
constraints on the multiplication operators recover the \cite{Alessandra} system.
Thus \cite{Alessandra} is the monomial/degree-extension specialization of the Multiplication Tensor -- Completion with Marked Generators (MT-CMG) completion problem.

\medskip

We now give a template that implements MT-CMG as a tensor completion problem. It reduces to \cite{Alessandra} in the
degree-extension specialization but it also supports a genuine Variable-Extension implementations by allowing
flexibility in the marked embedding $\iota$.
\vspace{0.5cm}
\begin{mdframed}[]
\begin{alg}[Variable-extension via iterated multiplication tensors]\label{alg:var-ext-imt}\end{alg}
\vspace{0.1cm}
\noindent\textbf{Input:} A degree $d\ge 2$ concise form $F\in \K_{\mathrm{dp}}[X_0,\dots,X_n]_d$, and a target length $r\ge n+1$.\\
\textbf{Output:} Either \texttt{FAIL}, or a length-$r$ commutative algebra $A$ with marked generators
$x_1,\dots,x_n\in A$ and the corresponding affine scheme \\
$Z=\mathrm{Spec}(\K[x_1,\dots,x_n]/I)\subset\mathbb{A}^n$ (hence a cactus scheme for $F$ of length $r$).

\begin{enumerate}
\item\label{choose:basis} \textbf{Choose a complete staircase basis.}
Enumerate complete staircase monomial sets
\[
B=\{b_0=1,x_1,\dots,x_n,\dots,b_{r-1}\}\subset \K[x_1,\dots,x_n]
\]
of cardinality $r$ (i.e.\ closed under divisors).

    \item\label{(2)} \textbf{Build a parametric completion in an enlarged variable space.}
    Let $W^*=\langle X_0,\dots,X_n,X_{n+1},\dots,X_{r-1}\rangle$ and consider a general form
    \[
    G(\lambda)\in \K_{\mathrm{dp}}[X_0,\ldots,X_{r-1}]_d
    \]
    such that the coefficients of $G$ under the identification $b_j^*\leftrightarrow X_j$ correspond to the functional $F(X_0=1)\in \K[x_1,\ldots,x_n]_{\leq d}^*$, with the remaining coefficients as unknowns
    $\lambda$.

    If $\det H^{B,B}(\lambda_0) = 0$, equivalently if the first catalecticant of $x_0^{d-2}\aprod G(\lambda)$ is not full rank, go back to \Cref{choose:basis} with a different monomial set.
    
    \item\label{(3)} \textbf{Construct symbolic multiplication matrices from $G(\lambda)$.}
    Using the recipe in \Cref{rmk: MultMatricesFromTensor},
    construct for each $i=1,\dots,n$ a symbolic matrix $\widetilde M_{x_i}(\lambda)\in M_r(k(\lambda))$
    representing multiplication by the marked element $x_i$ in the basis $B$.
    Concretely, form the symbolic matrices 
    \[
    H^{B,B}(\lambda)=\big(\Lambda_\lambda(b_\alpha b_\beta)\big)_{\alpha,\beta},
    \qquad
    H^{x_iB,B}(\lambda)=\big(\Lambda_\lambda(x_i\,b_\alpha b_\beta)\big)_{\alpha,\beta},
    \]
    where $\Lambda_\lambda(\cdot)$ is read off from the coefficients of $G(\lambda)$ under
    $b_j^*\leftrightarrow X_j$, and set
    \[
    \widetilde M_{x_i}(\lambda)=(H^{B,B}(\lambda))^{-1}H^{x_iB,B}(\lambda)
    \quad\text{whenever }\det H^{B,B}(\lambda)\neq 0.
    \]
    
    \item\label{(4)} \textbf{Impose the algebra constraints (commutativity/associativity).}
    Solve for $\lambda=\lambda_0$ such that
    \[
    \det H^{B,B}(\lambda_0)\neq 0,
    \]
    \[
    \widetilde M_{x_i}(\lambda_0)\widetilde M_{x_j}(\lambda_0)=
    \widetilde M_{x_j}(\lambda_0)\widetilde M_{x_i}(\lambda_0)
    \quad \forall\, i,j=1,\dots,n.
    \]
    If no solution exists for the current $B$, go back to \Cref{choose:basis} and try the next staircase.

    \item\label(5) \textbf{Recover the algebra and the scheme.}
    Set $M_{x_i}:=\widetilde M_{x_i}(\lambda_0)$ and define a $\K[x_1,\ldots,x_n]$-module structure on $A=\langle B\rangle_k$ by $p\cdot v:=(p(M_{x_1},\dots,M_{x_n}))(v)$
    and set
    \[
    I:=\operatorname{Ann}_{\K[x_1,\ldots,x_n]}(A)=\{p\in \K[x_1,\dots,x_n]\ :\ p(M_{x_1},\dots,M_{x_n})(1)=0\}.
    \]
    Output the affine scheme $Z=\mathrm{Spec}(\K[x_1,\dots,x_n]/I)\subset\mathbb{A}^n$.

    \item\label{(6)} \textbf{(Optional) Explicit generators via staircase rewriting relations.}
    Let $\partial B=\{x_i b:\ b\in B,\, i=1,\dots,n\}\setminus B$.
    For each $m=x_i b_j\in\partial B$, define
    \[
    g_m:=x_i b_j-\sum_{\ell=0}^{r-1}(M_{x_i})_{\ell j}\,b_\ell\in \K[x_1,\dots,x_n].
    \]
    Then $I=\langle g_m:\ m\in\partial B\rangle$ and these $g_m$ give explicit equations of $Z$.

    \item\label{(7)} \textbf{(Optional) Support and local structure.}
    Compute common eigenvectors of the commuting family $\{M_{x_i}^t\}$ to recover the support points of $Z$.
    Recover local lengths and nonreduced structure from joint generalized eigenspaces (Jordan chains) of
    $\{M_{x_i}^t\}$ on each local block (as in the Bernardi--Taufer reconstruction pipeline).
\end{enumerate}
\end{mdframed}

\ale{We could replace the above Algorithm with the following.}

in the Bernardi–Taufer reconstruction pipeline).

\vspace{0.5cm}

\ale{
\begin{example}
Let $F=X_0^{[3]}X_2+X_0^{[2]}X_1^{[2]}+X_0^{[2]}X_1X_2+X_2^{[4]}$. Its dehomogenization is $f=F(1,X,Y)=Y+X^{[2]}+XY+Y^{[4]}$. The apolar ideal is $I=\mathrm{Ann}(f)=(x^2-xy,\ xy-y^4,\ y^5)\subset \K[x,y]$, and hence $A:=\K[x,y]/I$ is a local Artinian Gorenstein algebra of length $6$.

We run \Cref{alg:min-apolar-squat} with target length $r=6$ and choose the
squat staircase
\[
E=\{1,v,w,w^2,w^3,w^4\}.
\]
Inside $A$, define
\[
v:=\bar x-\bar y^3,
\qquad
w:=\bar y.
\]
Then
\[
vw=(\bar x-\bar y^3)\bar y=\bar x\bar y-\bar y^4=0,
\]
\[
v^2=(\bar x-\bar y^3)^2=\bar x^2-2\bar x\bar y^3+\bar y^6=\bar y^4=w^4,
\]
and
\[
w^5=\bar y^5=0.
\]
Therefore
\[
A\simeq \K[v,w]/(vw,\ v^2-w^4,\ w^5),
\]
and
\[
B_E=(1,v,w,w^2,w^3,w^4)
\]
is a basis of $A$ indexed by the squat staircase $E$.

Let
\[
\Lambda:A\to\K
\]
be the functional extracting the coefficient of $w^4$, namely
\[
\Lambda(1)=\Lambda(v)=\Lambda(w)=\Lambda(w^2)=\Lambda(w^3)=0,
\qquad
\Lambda(w^4)=1.
\]

The original variables are not represented by the squat generators $v,w$, but by the transported marked generators
\[
u_x:=\bar x=v+w^3,
\qquad
u_y:=\bar y+w^4.
\]
Thus, in the basis $B_E$, we have
\[
u_x=b_1+b_4,
\qquad
u_y=b_2+b_5.
\]
In particular, both marked generators involve basis directions beyond the
original degree-$1$ directions $v,w$.

Their moments recover the coefficients of $f$:
\[
\Lambda(u_x)=0,\qquad \Lambda(u_y)=1,
\]
\[
u_x^2=(v+w^3)^2=v^2+2vw^3+w^6=w^4,
\qquad\Rightarrow\qquad
\Lambda(u_x^2)=1,
\]
\[
u_xu_y=(v+w^3)(w+w^4)=vw+vw^4+w^4+w^7=w^4,
\qquad\Rightarrow\qquad
\Lambda(u_xu_y)=1,
\]
\[
u_y^2=(w+w^4)^2=w^2,\qquad
u_y^3=w^3,\qquad
u_y^4=w^4,
\]
hence
\[
\Lambda(u_y^2)=0,\qquad
\Lambda(u_y^3)=0,\qquad
\Lambda(u_y^4)=1.
\]
Therefore the moments of $u_x,u_y$ recover exactly
\[
f=Y+X^{[2]}+XY+Y^{[4]}.
\]

%By contrast, if one used the squat generators $v,w$ themselves, then one would
%obtain the moments of
%\[
%X^{[2]}+Y^{[4]},
%\]
%since
%\[
%\Lambda(v)=\Lambda(w)=0,\qquad
%\Lambda(v^2)=1,\qquad
%\Lambda(vw)=0,\qquad
%\Lambda(w^4)=1.
%\]
%Thus the variable extension is genuinely needed in this example.

Finally, the induced homomorphism
\[
\K[x,y]\longrightarrow A,\qquad x\mapsto u_x,\quad y\mapsto u_y
\]
has kernel
\[
(x^2-xy,\ xy-y^4,\ y^5),
\]
so the algorithm returns
\[
Z=\mathrm{Spec}\bigl(\K[x,y]/(x^2-xy,\ xy-y^4,\ y^5)\bigr).
\]
\end{example}
}
%\begin{remark}
%\label{rmk:marking-template}
%In \Cref{alg:var-ext-imt} we work in the standard marking where the generators are the classes
%of the linear monomials, i.e.\ $\iota(x_i)=[x_i]\in R/I$ once $A\simeq R/I$.
%More general markings (allowing $\iota(x_i)$ to be an arbitrary element of $A$) fit the basis-free \Cref{def:MT-CMG} but lead to additional unknowns and are not pursued in this template.
%\end{remark}

\begin{remark}
This viewpoint is close in spirit to the quantum marginal (state extension) problem; see, e.g.,
\cite{TycVlach2015,Schilling2014}. 
For extension-based consistency criteria (via symmetric extensions), see \cite{DohertyParriloSpedalieri2004}.
\end{remark}

\begin{remark}
Although the focus of this paper is cactus rank, the same computational framework can also be used  toward Waring decompositions.
Indeed, once a completion produces commuting multiplication matrices, one may continue by increasing the size of the basis to recover a reduced apolar scheme, i.e.\ an apolar ideal defining a set of distinct points, for instance by checking whether the reconstructed ideal is radical (equivalently, whether the multiplication operators are simultaneously diagonalizable).
When this succeeds, the resulting points yield a Waring decomposition of the corresponding length.
\end{remark}

% \begin{remark}[Two realizations of the same length-$r$ scheme: apolar to $G$ and apolar to $F$]
% Assume \Cref{alg:var-ext-imt} returns a solution.
% By construction we obtain an Artinian Gorenstein algebra $A$ of length $r$ and a form
% \[
% G=\mu^{(d-1)}_{A,s}\in \K_{\mathrm{dp}}[W^*]_d,
% \]
% hence $\mathrm{Spec}(A)$ is (in the standard apolarity sense) an apolar scheme for $G$ in the ambient space dual to $W^*$.
% On the other hand, the marked morphism $\rho:k[x_1,\dots,x_n]\to A$ defines the scheme
% \[
% Z_F=\mathrm{Spec}\big(k[x_1,\dots,x_n]/\ker\rho\big)\subset \mathrm{A}^n,
% \]
% which is apolar to $F=G|_{V^*}$.
% If $\rho$ is surjective, then $A\simeq k[x_1,\dots,x_n]/\ker\rho$ and the canonical map
% $\mathrm{Spec}(A)\to Z_F$ is an isomorphism.
% Thus the construction produces the same length-$r$ scheme $\mathrm{Spec}(A)$ together with two different
% realizations: one as an apolar scheme for $G$ in (typically) higher dimension, and one as the cactus/apolar
% scheme for $F$ in $\mathbb{A}^n$.
% \end{remark}

}

\section{From degree-extension to variable-extension in a squat basis}\label{Section8}

In the degree-extension approach one works with a monomial staircase basis
containing the marked generators $x_1,\ldots,x_n$. After a nonlinear change of
coordinates, however, the squat generators $t_1,\ldots,t_n$ do not represent the
original marked variables. Therefore, instead of explicitly computing the transported
marked generators as polynomials in the $t$-variables, we keep the original
moment functional in the $x$-coordinates and encode the marking through the
multiplication matrices by $x_1,\ldots,x_n$.

More precisely, let
\[
\phi:\K[t_1,\ldots,t_n]\to \K[x_1,\ldots,x_n],
\qquad
t_i\mapsto v_i(x),
\]
where $v_i=x_i+p_i(x)$, or more generally where the $v_i$'s are chosen so that
a prescribed squat staircase becomes a basis after applying $\phi$.
For a squat staircase $E\subset \mathbb N^n$, set
\[
B_E=\{t^e:e\in E\}.
\]
We compute the Hankel matrices of $\phi^*\Lambda$ in the $t$-basis $B_E$,
but all entries are still expressed in the original moment variables $\Lambda(x^\alpha)$.
In this way the matching with the input form $F$ is imposed directly by the
conditions $\Lambda(x^\alpha)=f_\alpha$, $|\alpha|\le d$, and no inverse change of variables is required.

\vspace{0.5cm}
\begin{mdframed}[]
\begin{alg}[Cactus rank and decomposition via variable-extension in a squat basis]
\label{alg:cactus-squat-variable}
\end{alg}
\vspace{0.1cm}

\noindent\textbf{Input:} A degree $d\geq 2$ polynomial $F\in S_d$.\\
\textbf{Output:} Cactus rank of $F$.

\begin{enumerate}
    \item Construct the matrix $H_{\Lambda(h)}$ with parameters
    $\{h_{\alpha}\}_{\alpha\in\mathbb N^n}$, $|\alpha|>d$.

    \item\label{alg:squat:H:star} Set $r$ as the highest rank of a numerical subminor of
    $H_{\Lambda(h)}$.

    \item\label{alg:squat:basis} Take a squat staircase $E\subset\mathbb N^n$, $|E|=r$,
    and set
    \[
        B_E=\{t^e:e\in E\}\subset \K[t_1,\ldots,t_n].
    \]
    Choose elements $v_1,\ldots,v_n\in R=\K[x_1,\ldots,x_n]$, and consider
    \[
        \phi:\K[t_1,\ldots,t_n]\to R,\qquad t_i\mapsto v_i(x).
    \]
    In local computations one may take $v_i=x_i+p_i(x)$, with $p_i$ of order at
    least $2$. Do:

    \begin{itemize}
        \item Construct the matrix
        \[
            H_{\phi^*\Lambda(h)}^{B_E}
            =
            H_{\Lambda(h)}^{\phi(B_E)}
            =
            \left(
            \Lambda(h)(v^{e+e'})
            \right)_{e,e'\in E}.
        \]
        
        \item\label{alg: cactus; step: shifted mult. matrices} Construct the multiplication operators
        \[            M_{x_1}^{\phi(B_E)},\ldots,M_{x_n}^{\phi(B_E)}
        \]
        in the basis $\phi(B_E)=\{v^e:e\in E\}$,
        where 
        \[   \left(M_{x_j}^{\phi(B_E)}\right)^{T} \cdot H_{\Lambda(h)}^{\phi(B_E)}=\left(
            \Lambda(h)(x_jv^{e+e'})
            \right)_{e,e'\in E},
            \qquad j=1,\ldots,n.\]

        \item \label{alg: cactus; step: Solve commutator} Find $h$'s such that 
        \begin{itemize}
         \item $H_{\phi^*\Lambda(h)}^{B_E}$ has nonzero determinant;
            \item the multiplication operators $M_{x_1}^{\phi(B_E)},\ldots,M_{x_n}^{\phi(B_E)}$
            commute;
        \end{itemize}
            
        \item If found, the cactus rank of $F$ is $r$. Moreover, the apolar algebra is
        recovered from the representation
        \[
            R\to \operatorname{End}_k(\langle B_E\rangle),
            \qquad
            x_j\mapsto M_{x_j}^{\phi(B_E)}.
        \]
        Equivalently,
        \[
            A\simeq R/I,\qquad
            I=
            \left\{
            p\in R:
            p\!\left(
            M_{x_1}^{\phi(B_E)},\ldots,M_{x_n}^{\phi(B_E)}
            \right)\cdot 1=0
            \right\}.
        \]
        If not, go to \Cref{alg:squat:basis} with another choice of $E$. If all choices with $|E|=r$ have been already performed, go to
        \Cref{increase:squat:r}.
    \end{itemize}

    \item\label{increase:squat:r} Set $r\to r+1$ and go to
    \Cref{alg:squat:basis}.
\end{enumerate}
\end{mdframed}
\vspace{0.5cm}

\begin{remark}\label{rmk: MomentsVsDegrees}
 Although taking squat staircases after a general nonlinear change of coordinates considerably reduces the number of candidate bases (cf. \Cref{Section 6}), it can also increase the complexity of the polynomial system given by the commutation of the matrices, subject to the open condition $\det H_{\Lambda(h)}\neq 0$. Let $d=\deg F$ and let $M$ be the maximum degree of a candidate basis $B$. To decide whether to reduce the number of bases or solve simpler polynomial systems, we have the following analysis:
 \begin{itemize}
    \item If $2M+1\leq  d$, then both $H_{\Lambda(h)}^B$ and $M_{x_i}$ are numerical matrices, that is, they have no dependency on the variables $h_{x^\alpha}$. Thus, in this case \Cref{alg: cactus; step: Solve commutator} in \Cref{alg:cactus-squat-variable} consists of checking whether numerical matrices commute. 

    \item If $d$ is even and $2M=d$, then $H_{\Lambda(h)}^B$ is numerical, while $M_{x_i}^B$ have some entries that are linear expressions in the $h$'s, and therefore the commutation step in \Cref{alg:cactus-squat-variable} yields an ideal generated by quadrics. 

    \item If $2M>d$, then both $H_{\Lambda(h)}^B$ and $M_{x_i}^B$ depend on moment variables $h_{x^\alpha}$, and therefore we have a more complex polynomial system arising from the commutator, as well as the open condition $\det H_{\Lambda(h)}^B\neq 0$. 
 \end{itemize}
\end{remark}

% \begin{remark}
% In \Cref{alg:cactus-squat-variable}, the symbols
% $M_{x_j}^{\phi(B_E)}$ and $M_{t_i}^{B_E}$ denote genuine multiplication
% matrices. If the implementation uses adjugate-scaled transpose matrices
% \[
%     \widetilde M_a
%     =
%     H_{a,E}\operatorname{adj}(H_E)
%     =
%     \det(H_E)M_a^t,
% \]
% then the compatibility equations must be imposed after the normalization $M_a=\det(H_E)^{-1}\widetilde M_a^t$.
% \end{remark}

\begin{example}
Let $k$ be a field of characteristic $0$, and let
\[
F=X_0^{(4)}X_1^{(2)}+X_0^{(4)}X_1X_2
\in \K_{\mathrm{dp}}[X_0,X_1,X_2]_6.
\]
After dehomogenizing with respect to $X_0$, we get
\[
f=F(X_0=1)=X^{(2)}+XY\in \K_{\mathrm{dp}}[X,Y]_{\le 6}.
\]
Thus the corresponding truncated functional $\Lambda\in \K[x,y]^*$ satisfies $\Lambda(x^2)=1$, $\Lambda(xy)=1$, and all other moments of degree $\le 6$ are equal to $0$.

The known Hankel matrix contains the nonzero minor indexed by $B=\{1,x,y,x^2\}$.
Indeed,
\[
H_\Lambda^{\{1,x,y,x^2\}}
=
\begin{pmatrix}
0&0&0&1\\
0&1&1&0\\
0&1&0&0\\
1&0&0&0
\end{pmatrix},
\qquad
\det H_\Lambda^{\{1,x,y,x^2\}}=1.
\]
Hence the algorithm starts with $r=4$, and $B$ is a candidate basis. We see that the elements in $B$ have degree at most 2, and therefore by \Cref{rmk: MomentsVsDegrees} the multiplication matrices are numerical:
\[M_{x}^B=\left(\!\begin{array}{cccc}
      0&0&0&0\\
      1&0&0&0\\
      0&0&0&0\\
      0&1&1&0
      \end{array}\!\right) \quad M_y^B=\left(\!\begin{array}{cccc}
      0&0&0&0\\
      0&0&0&0\\
      1&0&0&0\\
      0&1&0&0
      \end{array}\!\right)
\]
These matrices commute and satisfy
\[
M_y^2=0,
\qquad
M_x^2=M_xM_y.
\]
Therefore the recovered ideal is
\[
I=(y^2,\ x^2-xy)\subset \K[x,y].
\]
Hence
\[
A\simeq \K[x,y]/(y^2,\ x^2-xy).
\]
This algebra has basis $\{1,x,y,x^2\}$
and length $4$. Since the initial Hankel lower bound was already $r=4$, the
algorithm returns
$\operatorname{crk}(F)=4$.

If we instead consider a nonlinear change of variables, say $v_1=x+x^2$, $v_2=y+xy$ in \Cref{alg:cactus-squat-variable}, then for the squat staircase $\{1,v_1,v_2,v_1^2\}$ we have 
\[H_{\Lambda(h)}^{\phi(B)}=\left(\!\begin{array}{cccc}
      0&1&1&1\\
      1&1&1&0\\
      1&1&0&0\\
      1&0&0&h_{x_{1}^{8}}+4\,h_{x_{1}^{7}}
      \end{array}\!\right),
      \]
which has determinant $h_{x_{1}^{8}}+4\,h_{x_{1}^{7}}+1$. The commutation $[M_x^{\phi(B)}, M_y^{\phi(B)}]=0$ produces an ideal minimally generated in degrees $1,2,2,2,2,3$. Thus, in this case testing all the possible staircases (which are 3 in total) is more efficient than considering a single squat staircase after a nonlinear change of coordinates.
\end{example}

\bibliographystyle{alphaurl}
\bibliography{bibliography}

\end{document}